\documentclass{amsart}

\usepackage[utf8]{inputenc}
\usepackage{amsmath,amsfonts,amssymb,amsthm,stmaryrd}
\usepackage[all,cmtip]{xy}
 \usepackage{appendix}
 \usepackage{xy}
\usepackage{tikz-cd}
\usepackage{hyperref}
\usepackage{thmtools}

\usepackage{cleveref}
\usepackage{mathtools}
\usepackage{xcolor}
\usepackage[old]{old-arrows}
\usepackage{qtree,tensor}
\usepackage{comment}
\usepackage{soul} 
\def\ddefloop#1{\ifx\ddefloop#1\else\ddef{#1}\expandafter\ddefloop\fi}

\def\ddef#1{\expandafter\def\csname bb#1\endcsname{\ensuremath{\mathbb{#1}}}}
\ddefloop ABCDEFGHIJKLMNOPQRSTUVWXYZ\ddefloop

\def\ddef#1{\expandafter\def\csname ff#1\endcsname{\ensuremath{\mathfrak{#1}}}}
\ddefloop ABCDEFGHIJKLMNOPQRSTUVWXYZ\ddefloop

\def\ddef#1{\expandafter\def\csname cc#1\endcsname{\ensuremath{\mathcal{#1}}}}
\ddefloop ABCDEFGHIJKLMNOPQRSTUVWXYZ\ddefloop

\newcommand{\mrk}{\mathrm{mrk}}
\newcommand{\Diag}{\mathrm{Diag}}
\newcommand{\Gr}{\mathrm{Gr}}
\newcommand{\pre}{\mathrm{pre}}
\newcommand{\Spec}{\mathrm{Spec}\,}
\newcommand{\Symm}{\mathrm{Sym}\,}
\newcommand{\Tot}{\mathrm{Tot}}
\newcommand{\Sect}{\mathrm{Sect}}
\newcommand{\Fun}{\mathrm{Fun}}
\newcommand{\Sch}{\mathrm{Sch}}
\newcommand{\Set}{\mathrm{Set}}
\newcommand{\op}{\mathrm{op}}
\newcommand{\ord}{\mathrm{ord}\,}
\newcommand{\rk}{\mathrm{rk}}
\newcommand{\Bl}{\mathrm{Bl}}
\newcommand{\Bld}{\mathrm{Bl}^{HL}}
\newcommand{\id}{\mathrm{id}}
\newcommand{\tf}{\mathrm{tf}}
\newcommand{\sm}{\mathrm{sm}}

\newcommand{\Coker}{\mathrm{Coker}}
\newcommand{\vir}{\mathrm{vir}}

\newcommand{\wPic}{\widetilde{\mathfrak{Pic}}}
\DeclareMathOperator{\Hom}{Hom}

\newcommand{\cl}{\mathrm{cl}}
\newcommand{\ev}{\mathrm{ev}}

\newcommand{\tor}{\mathrm{tor}}
\newcommand{\norm}[1]{\llbracket #1 \rrbracket}
\newcommand{\ver}{\mathrm{Ver}(\gamma)}
\newcommand{\tver}{\mathrm{Ver}(\gamma)^t}
\newcommand{\vers}{\mathrm{Ver}(\gamma)^*}
\newcommand{\Mon}{\mathrm{Mon}} 

\newcommand{\Proj}{\mathrm{Proj}}
\newcommand{\supp}{\mathrm{supp}\,}

\newcommand{\Ann}{\mathrm{Ann}}
\newcommand{\GIT}{/\!\!/}
\newcommand{\pd}{\mathrm{pd}}
\newcommand{\scl}{\mathrm{cl}^\mathrm{sch}}

\newcommand{\opi}{\overline{\pi}}
\newcommand{\tpi}{\widetilde{\pi}}
\newcommand{\wpi}{\widehat{\pi}}

\newcommand{\Pic}{\mathfrak{Pic}}

\newcommand{\Bun}{\mathfrak{Bun}}

\newcommand{\wBun}{\widetilde{\mathfrak{Bun}}}
\newcommand{\uC}{\underline{\ffC}}
\newcommand{\Gw}{\overline{\ccM}_{g,n}(\bbP^r,d)}
\newcommand{\Gwo}{\ccM_{g,n}(\bbP^r,d)}

\newcommand{\Gwone}{\overline{\ccM}_{1,n}(\bbP^r,d)}
\newcommand{\tGwone}{\widetilde{\ccM}_{1,n}(\bbP^r,d)}
\newcommand{\tGwmone}{\widetilde{\ccM}_{1,n}^\circ(\bbP^r,d)}

\newcommand{\mGw}{\overline{\ccM}^\circ_{g,n}(\bbP^r,d)}

\newcommand{\tGw}{\widetilde{\ccM}_{g,n}(\bbP^r,d)}

\newcommand{\tmGw}{\widetilde{\ccM}^\circ_{g,n}(\bbP^r,d)}
\newcommand{\tmQ}{\widetilde{\ccQ}^\circ_{g,n}(X,\beta)}

\newcommand{\Gwx}{\overline{\ccM}_{g,n}(X,d)}

\newcommand{\mGwx}{\overline{\ccM}^\circ_{g,n}(X,d)}

\newcommand{\tGwx}{\widetilde{\ccM}_{g,n}(X,d)}
\newcommand{\tmGwx}{\widetilde{\ccM}^\circ_{g,n}(X,d)}

\newcommand{\tGwi}{\widetilde{\ccM}^\circ_{g_i,n}(X,d)}
\newcommand{\Qx}{\overline{\ccQ}^{\epsilon}_{g,n}(X,\beta)}
\newcommand{\mQx}{\overline{\ccQ}^\circ_{g,n}(X,\beta)}
\newcommand{\tQx}{\widetilde{\ccQ}_{g,n}(X,\beta)}
\newcommand{\ts}{\widetilde{s}}
\newcommand{\tta}{\widetilde{t}}
\newcommand{\tm}{\widetilde{m}}
\newcommand{\cwt}{\ffM_1^{wt}} 
\newcommand{\blcwt}{\widetilde{\ffM}_1^{wt}} 
\newcommand{\bld}{\Bl_{\ccF}^{HL} X} 
\newcommand{\bldY}{\Bl_{f^*\ccF}^{HL} Y} 

\newcommand{\blHL}{\Bl^{HL}} 
\newcommand{\q}{{q}} 

\newtheorem{theorem}{Theorem}[subsection]
\newtheorem{corollary}[theorem]{Corollary}
\newtheorem{conjecture}[theorem]{Conjecture}
\newtheorem{lemma}[theorem]{Lemma}
\newtheorem{proposition}[theorem]{Proposition}
\newtheorem{main_result}[theorem]{Main result}
\theoremstyle{definition}
\newtheorem{definition}[theorem]{Definition}
\newtheorem{remark}[theorem]{Remark}
\newtheorem{assumption}[theorem]{Assumption}
\newtheorem{construction}[theorem]{Construction}
\newtheorem{notation}[theorem]{Notation}

\theoremstyle{remark}
\newtheorem{example}[theorem]{Example}

\title[Higher genus reduced GWI via desingularizations of sheaves]{ Higher genus reduced Gromov--Witten invariants via desingularizations of sheaves}
\author[A.Cobos Rabano, E. Mann, C. Manolache, R. Picciotto]{Alberto Cobos Rabano, Etienne Mann, Cristina Manolache, Renata~Picciotto}
\address{Alberto Cobos Rabano, KU Leuven, Department of Mathematics, Celestijnenlaan 200B box 2400, BE-3001 Leuven, Belgium}
\email{alberto.cobosrabano@kuleuven.be}
\address{\'Etienne Mann, Univ Angers, CNRS, LAREMA, SFR MATHSTIC, F-49000 Angers, France}
\email{etienne.mann@univ-angers.fr}
\address{Cristina Manolache, Univ Sheffield, Hicks building, Sheffield, UK}
\email{c.manolache@sheffield.ac.uk}
\address{Renata Picciotto, University of Cambridge, Centre for Mathematical Sciences, Wilberforce Road, Cambridge, UK}
\email{rp779@cam.ac.uk}
\date{}

\begin{document}

\maketitle
\begin{abstract} 
The main result is an all-genus construction of reduced
Gromov--Witten invariants for a large class of GIT quotients. This extends earlier results, limited to genus one and two. 

The main tool involves blowings-up of sheaves. More precisely, given $\ffF$ a coherent sheaf on a Noetherian integral algebraic stack $\ffP$, we give two constructions of stacks $\widetilde{\ffP}$, equipped with birational morphisms $p:\widetilde{\ffP}\to \ffP$ such that $p^*\ffF$ is simpler: 
   \begin{enumerate}
 \item \label{Rone} in the Rossi construction, the torsion free part of $p^*\ffF$  is locally free; 
\item in the Hu--Li diagonalization construction, $p^*\ffF$ is a union of locally free sheaves. 
\end{enumerate}
The construction in \cref{Rone} above is an extension of the Nash blow-up and certain flattenings of sheaves to Noetherian integral Artin stacks. We show that reduced Gromov--Witten invariants obtained from each of the constructions above coincide.
\end{abstract}
\tableofcontents

\section{Introduction}

\paragraph{\bf Overview of the problem}
Let $X$ be a smooth projective variety. We denote by $\Gwx$ the moduli space of genus $g$, degree $d \in H_2(X;\bbZ)$ stable maps to $X$ (see ~\cite{k}). By \cite{li1998virtual, Behrend-Fantechi}, the stack $\Gwx$ has a virtual fundamental class $[\Gwx]^{\vir}\in A_{*}(\Gwx)$. Gromov--Witten invariants of $X$ are defined as intersection numbers against this class. They are related to counts of curves in $X$ of genus $g$ and class $d$, but they often encode contributions from degenerate maps with reducible domains. In order to define invariants which discard such contributions, we need to define a virtual fundamental class on a smaller proper stack containing $\ccM_{g,n}(X,d)$. It is not possible to do this directly, we need to replace this stack by a birational one which admits a virtual fundamental class.

In genus one and two, there are several such constructions \cite{Zinger-genus-one-pseudo-holomorphic,Vakil-Zinger-desingu-main-compo-, Hu-Li-GEnus-one-local-VZ-Math-ann-2010,ranganathan2019moduli, Hu-Li-genus-two, battistella2020smooth, niu1, niu2}. The resulting numbers are called \emph{reduced} Gromov--Witten invariants.

\begin{main_result}\label{main_result_intro}
    We define reduced Gromov--Witten invariants in any genus for:
    \begin{enumerate}
        \item hypersurfaces in projective spaces when $d > 2g-2$ (see \Cref{def:reduced_GW_invariants} and Assumption \ref{assum:d,big,2g,minus2}); 
        \item  more generally, certain GIT quotients of vector spaces (see Definition \ref{def:reduced GW general}, Assumption \ref{assum:general} and \Cref{recalling quasimaps}).
    \end{enumerate}
\end{main_result}

 In genus one and two, our reduced Gromov--Witten invariants of complete intersections in a projective spaces agree with the reduced invariants defined previously. The main tool for our construction is the general flattening result of \Cref{main_theorem_blowup_stacks}.

\vspace{0.2cm}
As noted above, reduced invariants should correspond to the contribution of a class supported on the \textit{main component} of the moduli space, that is, the closure of the locus where the source curve is smooth. While the geometry of $\Gwx$ is in general inaccessible, there are ways to study it when the target is $\bbP^r$. Indeed, under the degree assumptions of \Cref{main_result_intro}, the main component $\mGw$ of stable maps to $\bbP^r$ is actually an irreducible component of the correct dimension (see \Cref{prop:main_component_Pr_open_in_tors_free}). 
For $X$ a hypersurface in $\bbP^r$, the main component $\mGwx$ is defined as the zero locus of a section of a \emph{sheaf} on $\mGw$. To endow $\mGwx$ with a virtual fundamental class, we give a modification of $\mGw$ which makes this sheaf locally-free. In the following, we explain how we can achieve this by improving the singularities of $\Gw$ and how this leads to a definition of reduced Gromov--Witten invariants for $X$.

A key ingredient in understanding $\Gw$ is the description of maps to $\bbP^r$ as the data of a genus $g$ curve $C$, a line bundle $L$, and $(r+1)$ sections \cite{marian2011moduli,C-FK}. If $\ffP$ is the Artin stack parametrizing nodal curves and line bundles of degree $d$ (see \eqref{eq:F_sheaf} and \eqref{eq:cone,moduli}), then $\Gw$ is an open subset of  $\Spec\Symm(\ffF)$ -- the total space of a coherent sheaf $\ffF$ over $\ffP$. The fiber of $\ffF$ at $(C,L)$ is the dual of $H^0(C,L)^{\oplus r+1}$, so we have

\[\begin{tikzcd} \Gw \arrow[rd]\arrow[hookrightarrow,"\text{open}"]{r} & \Spec\Symm(\ffF)\arrow[d] & H^0(C,L)^{\oplus r+1}\arrow[l,hook']\arrow[d] \arrow[dl, phantom,"{\rotatebox{-90}{$\ulcorner$}}", very near start]\\ & \ffP&(C, L)\arrow[l,hook']\text{.} \end{tikzcd}
\]

Under the degree assumption in \ref{assum:general}, the main component $\mGw$ is an open subset of the total space of the torsion-free quotient $\ffF^{\tf} \coloneqq \ffF/\tor(\ffF)$. The torsion of $\ffF$ gives rise to components of $\Gw$ supported on the locus of those $(C,L)$ such that $H^1(C,L)$ is nontrivial. This occurs when $C$ has components of high genus and the degree of $L$ on these components is low.

To improve the singularities of $\Gw$, we give a birational modification $p:\widetilde{\ffP}\to\ffP$ such that the torsion-free part of $p^*\ffF$ is locally free. This allows us to define a birational transformation $\tmGw\to \mGw$ with $\tmGw$ open in $\Spec\Symm(p^*\ffF)^{\tf}$, thus smooth over $\widetilde{\ffP}$. This is summarized by the following diagram
\[
\begin{tikzcd} \tmGw \arrow[hookrightarrow, "\text{open}"]{r} & \Spec\Symm(p^*\ffF)^\tf\arrow[d, "\text{smooth}", swap]\arrow[r]& \Spec\Symm(\ffF)\arrow[d]\\ & \widetilde{\ffP}\arrow[r,"p"]&\ffP\text{.} \end{tikzcd}
\]

The base change $\tmGwx$ of $\mGwx$ is a proper stack with a (relative) perfect obstruction theory. Indeed, after base change by $p$, $\tmGwx$ becomes the zero locus of a section of a vector bundle on $\tmGw$ (\Cref{prop: no k}). Since $\tmGw$ is also smooth, this gives a perfect obstruction theory of $\tmGwx$ relative to $\widetilde{\ffP}$. While $\widetilde{\ffP}$ is not necessarily smooth, it is still pure dimensional, so the perfect obstruction theory induces a virtual fundamental class. We define reduced Gromov--Witten invariants for $X$ in \Cref{subsec: def reduced GW} as intersection numbers against this class.

The process above desingularizes the main component of $\mGw$, but it leaves the boundary components of $\Gw$ unchanged. We use a second type of birational modification of $\ffP$ to also make the torsion of $\ffF$ into vector bundles (of possibly different ranks) over closed substacks of a new $\widetilde{\ffP}$. 
Conveniently, this new birational modification also flattens the obstruction sheaves. In \Cref{subsec:maps_with_fields}, we combine this with the $p$-fields approach of Chang--Li \cite{Chang-Li-maps-with-fields} to obtain a decomposition of the virtual fundamental class of $\Gwx$ into cycles supported on irreducible components of the moduli space (see the discussion around \Cref{pfield_decomp}), and in particular a different construction of reduced invariants.

\vspace{0.2cm}
\paragraph{\bf Main constructions}
\Cref{main_theorem_blowup_stacks} below provides two different birational modifications for $\ffP$ which make $\ffF^{\tf}$ locally-free, each of which can be used to give a definition of reduced Gromov--Witten invariants, as explained above. The first is minimal, in the sense it satisfies the universal property in \Cref{nontechmain}, part (1), the second also allows us to associate classes to other boundary components. In \Cref{invariance-red} we show that reduced Gromov--Witten invariants are independent of the birational modification of $\ffP$. In particular, both constructions give the same reduced Gromov--Witten invariants.

\begin{theorem}[See \Cref{thm:construction-blowup-stacks,Rossi univ,theorem:universal_property_diag_stacks}]\label{main_theorem_blowup_stacks}\label{nontechmain}
    Let $\ffP$ be an integral Noetherian Artin stack with affine stabilizers admitting an integral presentation and $\ffF$ be a coherent sheaf on $\ffP$. 
    
    There exist integral Noetherian Artin stacks $\Bl_{\ffF}\ffP$ and  $\blHL_{\ffF}\ffP$ together with representable proper birational morphisms $\pi\colon\Bl_{\ffF}\ffP\to\ffP$ and $\rho\colon \blHL_{\ffF}\ffP\to\ffP$ satisfying the following universal properties:
    \begin{enumerate}
    \item  For any morphism of stacks $p:\ffY\to\ffP$ such that $(p^*\ffF)^\tf$ is locally free of the same generic rank as $\ffF$, there is a unique morphism $p'$, which makes the following diagram 2-commutative
    \[
    \begin{tikzcd}
        \ffY\ar[r,"\exists ! p'", dashed]\ar[dr,"p" '] & \Bl_{\ffF}\ffP\ar[d,"\pi"]\\
         & \ffP
    \end{tikzcd}
    \]
    
    \item For any morphism of stacks $f:\ffY\to \ffP$ such that $f^*\ffF$ is diagonal of the same generic rank as $\ffF$, there is a unique morphism $f'$, which makes the following diagram 2-commutative:
    \[
    \begin{tikzcd}
        \ffY\ar[r,"\exists ! f'", dashed]\ar[dr,"f" '] & \blHL_{\ffF}\ffP\ar[d,"\rho"]\\
         & \ffP
    \end{tikzcd}
    \]
    \end{enumerate}
\end{theorem}

We call the stack $\Bl_\ffF\ffP$ introduced above the Rossi blow-up. This is a stacky version of the Raynaud-Gruson flattening \cite{Raynaud-Gruson-71}. This construction does not change torsion sheaves.

 We call $\Bl_\ffF^{HL}\ffP$ the Hu--Li blow-up. The Hu--Li construction also produces a sheaf $p^*\ffF$ whose torsion-free part is locally free. In addition to this, $p^*\ffF$ also has a well-behaved torsion in the sense of \Cref{def:diagonalizable}. For schemes, this is achieved by a construction of Hu and Li \cite{Hu-Li-diagonalization} and of Grivaux \cite[Proposition 12]{Grivaux}. 

The first construction is a minimal birational modification of $\ffP$. This is sufficient to define reduced Gromov--Witten invariants. However, in view of existing proofs of \Cref{conj gv} it is convenient to work with the Hu--Li blow-up. In general, the Rossi construction is different from the Hu--Li construction (see \Cref{example:Hu-Li_Rossi_different}). By \Cref{example: charts [ab[cd]]}, the Rossi construction gives a new moduli space which is different from the Vakil--Zinger blow-up from \cite{Vakil-Zinger-desingu-main-compo-}. More explanations on this can be found in \Cref{subsec:maps_with_fields}.

 \vspace{0.3cm}
\paragraph{\bf Relation to previous approaches} The structure of this paper is different from the ones in \cite{Vakil-Zinger-desingu-main-compo-, Hu-Li-GEnus-one-local-VZ-Math-ann-2010, Hu-Li-genus-two, niu1, niu2}. In the mentioned papers, the authors have a three-step strategy to constructing the stack $\tmGw$ which compactifies the locus of maps with smooth domain: 
\begin{enumerate}
\item they find equations of local embeddings of charts $W_i\to \Gw$ in smooth spaces $V_i$; 
\item they blow up $W_i$ to obtain $\widetilde{W_i}$; 
\item they show that $\widetilde{W_i}$ glue to a (smooth) stack $\tmGw$. 
\end{enumerate}
The first step becomes involved already in genus two, due to the rather complicated geometry of the moduli space of stable maps. Steps 2 and 3 are done by constructing an explicit blow-up of $\Pic$ (or $\ffM_{g,n}$). Finding a candidate for this blow-up is the hardest part of the constructions in \cite{Vakil-Zinger-desingu-main-compo-} and \cite{Hu-Li-genus-two}. 

In this work, we omit Step 1 completely. For us, Step 2 is minimal in a suitable sense -- it is given by a universal property. The main ingredient in Step 3 is that the (local) constructions proposed in \Cref{sec:desing_sheaf} and in \Cref{sec:diagonalization} commute with flat pullbacks and this allows us to glue them. Explicit equations of the charts $W_i$ are thus not necessary to construct $\tmGw$. The advantage of this approach is that the gluing is conceptual and straightforward. This is similar to what Hu and Li do in \cite{Hu-Li-diagonalization} -- our construction heavily relies on their ideas.

 \vspace{0.3cm}
\paragraph{\bf Relation to Gromov--Witten theory.}  Genus one reduced invariants for varietes of any dimension are related to Gromov--Witten invariants \cite{Zinger-standard-vs-reduced}. For threefolds the relation is conjecturally much simpler. Let $X$ be a threefold under the assumptions in \Cref{recalling quasimaps} and let $\gamma\in H^*(X)^{\oplus n}$ be a collection of cohomology classes of $X$. Let $N_\beta^g(\gamma)$ be the genus $g$ and degree $\beta$ Gromov--Witten invariants of $X$ with insertions given by $\gamma$, and let $r_\beta^g(\gamma)$ be the corresponding reduced invariants. In particular, for a Calabi-Yau threefold $X$, the moduli space has virtual dimension 0 and one defines the Gromov--Witten invariants and reduced Gromov--Witten invariants withouth insertions as $N_{\beta}^g$ and $r_{\beta}^g$.
  
  \begin{conjecture}\label{conj one}\cite{Zinger-reduced-g1-CY}, \cite[Conjecture 1.1]{Hu-Li-diagonalization}
 Let $X$ be a Calabi--Yau threefold. Then, there are universal constants $C_h(g)\in\bbQ$, such that for ${\rm deg}(\beta)>2g-2$, we have
 \[N_\beta^g=\sum_{0\leq h\leq g} C_h(g)r_\beta^h.\]
 \end{conjecture}
 When $X$ is the quintic threefold, the above formula in genus one is the formula in \cite{Zinger-reduced-g1-CY, li2009genus}
 \begin{equation}\label{quintic g one}N_d^1=\frac{1}{12}N_d^0+ r^1_d.
 \end{equation}
 
  If $X$ is a Fano threefold, the reduced invariants are expected to be equal to Gopakumar--Vafa invariants. Indeed, the Gopakumar--Vafa invariants are by definition related to Gromov--Witten invariants by a formula which takes into account degenerate lower genus and lower degree boundary contributions. For Fano varieties, there are no lower degree contributions. Boundary contributions were computed by Pandharipande in \cite{pandharipande1999hodge}. The conjectural equality between reduced Gromov--Witten invariants and Gopakumar--Vafa invariants (see \cite[Section 0.3]{pandharipande1999hodge}) for Fano threefolds gives the following.
 \begin{conjecture}\label{conj gv}\cite{Zinger-reduced-g1-CY, zinger2011comparison}
 Let $X$ be a Fano threefold and let $C_{h,\beta}^X(g)$ be defined by the formula
 \[\sum_{g\geq 0}C_{h,\beta}^X(g)t^{2g}=\left(\frac{\sin (t/2)}{t/2}\right)^{2h-2-K_X\cdot\beta}.\] 
 Then, we have the following
 \[
 N_\beta^g(\gamma)=\sum_{h=0}^g C^X_{h,\beta}(g-h)r_\beta^h(\gamma).
 \]
 \end{conjecture}
The above should also hold in the Calabi--Yau case, where $C_{h,\beta}^X(g)$ do not depend on $X$ and $\beta$.

The above conjectures have been proved in genus one and in genus two for quasimaps \cite{Lee-Li-Oh-quantum-lefschetz}. These are the only cases in which a suitable definition of reduced invariants existed prior to this work.

The conjecture in low genus was proved with symplectic methods \cite{li2007gromov} \cite{li2007gromov} and algebraic methods \cite{Chang-Li-hyperplane-property}, \cite{Lee-Li-Oh-quantum-lefschetz}. The main idea of the algebraic proof is to use an additional space of maps with fields, whose geometry is similar to the geometry of the moduli space of stable maps to a projective space. After a sequence of blow-ups, the moduli space of maps with fields becomes a union of components which are well understood: they are smooth of dimension greater or equal to the expected dimension. This allows an expression of the virtual fundamental class on the moduli space of maps with fields as a sum of classes on the components of the moduli space. An analysis of these classes then gives the relation between standard and reduced Gromov--Witten invariants. The Hu--Li blow-up of $\Pic$ provides such an auxiliary moduli space to tackle the conjecture in all genera.

\vspace{0.3cm}
\paragraph{\bf History and related works: Gromov--Witten Theory}
Reduced genus 1 invariants are the output of a long and impressive project. Reduced invariants in genus one were defined using symplectic methods and compared to Gromov--Witten invariants by Zinger \cite{Zinger-standard-vs-reduced, Zinger-natural-cones, Zinger-reduced-g1-GW, Zinger-reduced-g1-CY}. Li--Zinger showed \cite{li2007gromov, li2009genus} that reduced genus one Gromov--Witten invariants are the integral of the top Chern class of a sheaf over the main component of $\Gw$. This is an analogue, for reduced genus 1 invariants, of the quantum Lefschetz hyperplane property \cite{li2007gromov, li2009genus}. In view of \cite{Zinger-reduced-g1-GW}, this also gives a proof of the formula \eqref{quintic g one}. The algebraic definition requires a blow-up construction for the moduli space of stable maps to projective space, due to Vakil and Zinger \cite{Vakil-Zinger-desingu-main-compo-,Vakil-Zinger-compactification-elliptic-curves}. See~\cite{zinger2020some} for a survey from the symplectic perspective.

Explicit local equations for the Vakil--Zinger blow-up in genus one are given in \cite{Zinger-genus-one-pseudo-holomorphic,Hu-Li-GEnus-one-local-VZ-Math-ann-2010} and in genus two in \cite{Hu-Li-genus-two}. It is expected that the methods used in low genus could provide local equations for general moduli spaces of stable maps to projective spaces, but the combinatorics are likely to be quite involved. 

In \cite{Hu-Li-GEnus-one-local-VZ-Math-ann-2010, Hu-Li-genus-two, niu1, niu2} the authors give a modular interpretation of reduced invariants in terms of graphs of degenerate maps. A modular interpretation via log maps has been given by Ranganathan, Santos-Parker and Wise \cite{ranganathan2019moduli, ranganathan2019algebra}. 

Hu and Li introduce the diagonalization construction in \cite{Hu-Li-diagonalization}. They use this construction to define an Euler class on the moduli space of stable maps to projective spaces. This gives a non-intrinsic definition of reduced invariants of complete intersections. \Cref{conj gv} is hard to approach with this definition. In this paper, we rework their construction.

In a different direction, instead of replacing the moduli space of maps with a space which dominates $\tmGwx$, one can construct a space dominated by $\tmGwx$. This has been done by moduli spaces of maps from more singular curves, such as in \cite{battistella2020reduced, battistella2020smooth}. A modular interpretation comes for free with this approach, which makes these constructions particularly beautiful. A relationship between reduced invariants and invariants from maps with cusps was established in~\cite{battistella2020reduced}. Battistella and Carocci introduce a compactification of genus two maps to projective spaces \cite{battistella2020smooth}. An example of this compactification is given in \cite{battistella2022geographical}.

More recently, reduced invariants for the quintic threefold have been compared to Gromov--Witten invariants using algebro-geometric methods by Chang and Li \cite{Chang-Li-hyperplane-property}. Chang--Li \emph{define} reduced invariants as the integral against the top Chern class of a sheaf but, as discussed above, this gives the same reduced invariants as \cite{Zinger-reduced-g1-GW}. The algebraic comparison relies on the construction of maps with fields due to Chang and Li \cite{Chang-Li-maps-with-fields}, and on Kiem--Li's cosection localised class \cite{kiem2013localizing}. This method has been employed in \cite{Lee-Oh-reduced-complete-intersections-2,Lee-Oh-reduced-complete-intersections} to extend the genus one relation between absolute and reduced Gromov--Witten invariants of complete intersections. In genus two, a similar work is done in \cite{Lee-Li-Oh-quantum-lefschetz}.

Zinger has computed reduced genus one invariants of projective hypersurfaces via localisation~\cite{Zinger-reduced-g1-CY}. The computations in \cite{Zinger-genus-one-pseudo-holomorphic} and \cite{Zinger-standard-vs-reduced} have been extended to complete intersections by Popa \cite{popa2013}.

Shortly after our work was completed, Nesterov defined Gopakumar--Vafa invariants \cite{nesterov2024unramified} by proving a relation between Gromov--Witten invariants and invariants defined using the moduli space of unramified maps \cite{kim2014compactification}. In view of this, reduced invarants of Fano threefold conjecturally agree with unramified maps invariants. We expect these invariants to be different for varieties which are not Fano.

\vspace{0.3cm}
\paragraph{\bf History and related work: Flattening and Nash transformations}

Given a scheme (or a stack) $\ffP$ and $\mathcal{\ffF}$ a coherent sheaf on it, in this paper we loosely refer to a birational map $\widetilde{\ffP}\to \ffP$ such that the pull-back of $\mathcal{\ffF}$ is better behaved as the blow-up of $\ffP$ at $\mathcal{\ffF}$. This is compatible to the notion of blow-up of a ring at a module in \cite{Villamayor}. We do not define general blow-ups of schemes (or stacks) at a sheaf, but we use the notation $\Bl_{\mathcal{\ffF}}\ffP$ for certain more specific birational transformations.

\emph{Blow-ups of sheaves} have an interesting history: they appear in different contexts and they seem to have been re-discovered several times. As a consequence, the same constructions appear in the literature with different names. Given their scattered appearance, we decided to review some of these constructions in detail and to discuss relations between them.

The first construction that we are aware of is a blow-up of a sheaf in the analytic category, introduced by Rossi. It is a short geometric construction in the little known paper \cite{Rossi} from 1968. The main focus of Rossi's paper is different and his blow-up is just a technical tool. A more general, relative version was proved later by Hironaka \cite{Hironaka-75}. Perhaps the most influential flattening theorem was proved by Raynaud--Gruson \cite{Raynaud-Gruson-71} in great generality. Their result is relative to a possibly non-Noetherian base scheme and their proof is involved. The Noetherian case is treated also in \cite{Raynaud1972}. On an integral Noetherian scheme, Rossi's construction is also known as the \emph{Nash transformation} of a sheaf \cite{Oneto-Zatini-1991}. 
In the particular case when the sheaf is the cotangent of the given scheme, it is called the \emph{Nash blow-up} \cite{Nobile}. In addition to these, Rossi's geometric construction was recently re-introduced in \cite{curto2013threefold} by Curto and Morrison. They call it a \emph{Grassmann blow-up}. None of these authors seemed to be aware of Rossi's paper\footnote{In line with previous authors, but less excusably in the age of the internet, we started this project unaware of the various constructions in the literature. Jarod Alper, Ananyo Dan, Evgeny Shinder and Michael Wemyss kindly brought several references to our attention.}. 

The most established term in the literature appears to be \emph{flattening} (or occasionally flatification). However, we consider the construction by Rossi, Nash, and Curto--Morrison a significant special case of the Raynaud--Gruson and Hironaka flattening, deserving its own name: it stands out as simple and geometric. While we acknowledge the contributions of each of the authors above, we have decided to refer to it as the \emph{Rossi construction} or Rossi blow-up.

Of particular interest to us are blowings-up which enjoy a universal property. A universal property is already mentioned in \cite{Raynaud-Gruson-71,Raynaud1972}, where they also established the connection with Fitting ideals under additional assumptions. This connection became more explicit in \cite{Oneto-Zatini-1991,Villamayor}.

A \textit{diagonalization construction} was introduced in \cite{Grivaux} for analytic manifolds and in \cite{Hu-Li-diagonalization} for algebraic stacks.
This has similarities with the Rossi construction, although its flavour is more algebraic. One of its advantages is that it is quite simple. Our approach is greatly influenced by their construction.

Nash blow-ups appear naturally in singularity theory. We mention \cite{Nobile,Spivakovsky-Sandwiched, Nash_blowup_fails} from a vast literature. The history of this problem is explained in \cite{Spivakovsky}. The Raynaud--Gruson flattening was also used in birational geometry. For example, Curto and Morrison conjectured that the Grassmann blow-up gives a theoretical way of constructing all smooth 3-fold flops. This was proved in \cite{gustavsen2018deformations}. In \cite{romano2023blow} the authors discuss singularities of certain Raynaud--Gruson blow-ups of surfaces.

The result of Raynaud--Gruson received further interest very recently: it has been reproved by Guignard \cite{guignard2021new} and we have been informed by Rydh that he is working on generalisations to stacks \cite{rydh, Rydh_equi-flattening}. After our paper was made available, a generalisation to stacks was also proved by McQuillan \cite{mcquillan2024flatteningalgebrisation}. 

Our result is a stacky construction of a particular case of the Raynaud--Gruson flattening. More precisely, in notation as in \cite[Theorem I.5.2.2]{Raynaud-Gruson-71} we prove the result for $X=S$ and $X$ an integral Noetherian Artin stack. Rydh's and McQuillan's proofs are more general and very different from ours. To our knowledge, our result is the first stacky instance of a Raynaud--Gruson flattening.

\vspace{0.3cm}
\paragraph {\bf Outline of the paper.} 
In the following we give an outline of the paper and we highlight the main results.

In \Cref{sec: background} we fix notation and briefly recall the background notions used, such as Fitting ideals and abelian cones.

In \Cref{sec:desing_sheaf} we introduce the desingularization of a sheaf on a stack (\Cref{def:desingularization}) and we review the minimal desingularization, due to Rossi (\Cref{ss:local_construction}). In \Cref{subsec:Villamayor} we give an algebraic desingularization in terms of Fitting ideals in the affine case, due to Oneto--Zatini and Villamayor. We show in \Cref{subsec: properties univ desing} that the Rossi and Villamayor constructions agree. We show several properties of the minimal desingularization of a sheaf in \Cref{subsec: properties univ desing}.

In \Cref{subsec: diagonal sheaves and morphsims}, we introduce the notion of diagonal sheaf (see \Cref{defi:diag_sheaf}). In \Cref{subsec: construction Hu Li blowup} we recall \Cref{constr: diagonalization_blowup}, due to Hu and Li, which gives the minimal diagonalization of a sheaf. This is formalized in \Cref{subsec: univ property HL blowup} via a universal property (\Cref{thm:universal_property_diagonalization}). In \Cref{subsec:properties_HL}, we collect properties of the Hu--Li blow-up, such as the existence of a morphism from the Hu--Li blow-up to the Rossi blow-up in \Cref{prop:morphism-HL-Rossi}.  The two blow-ups are not isomorphic in general, as shown in \Cref{example:Hu-Li_Rossi_different}. In \Cref{subsec: filtration diagonal sheaf} we construct a filtration of a diagonal sheaf under certain conditions (see \Cref{thm:cone of diagonalizable sheaf}). This filtration will then be used in \Cref{sec:components_abelian_cones} to describe the irreducible components of the abelian cone of a diagonal sheaf. Finally, in \Cref{subsec: remarks on minimality} we collect some remarks and examples regarding the minimality of the Hu-Li blow-up.

In \Cref{sec:desing_stacks} we generalize the Hu--Li and Rossi blow-ups to Artin stacks. These are constructed in \Cref{subsec:constr_bl_stacks}, by first applying the Rossi and Hu--Li constructions for schemes to an atlas and then gluing. This works because the Hu--Li and Rossi constructions are local, they have a universal property, and they commute with flat base-change by \Cref{prop:map_between_blowups,prop:map_between_blowups_diagonal}. We extend to stacks the results on the schematic versions of these blow-ups, such as the universal properties of desingularization in \Cref{Rossi univ} and of diagonalization in \Cref{theorem:universal_property_diag_stacks}.

\Cref{sec:components_abelian_cones} is devoted to the study the irreducible components of the abelian cone $C(\ccF)$ of a diagonal sheaf. Given a coherent sheaf $\ccF$ on an integral Noetherian scheme $X$, we introduce the main component (\Cref{def: main component abelian cone}) $C(\ccF^\tf)$ of the abelian cone $C(\ccF)$ in \Cref{subsec: main component abelian cone}. General cones need not have a main component (\Cref{ex: tf cone reducible}). Furthermore, if $\ccF^\tf$ is locally free, then $C(\ccF)$ is a pushout of its main component, which is a vector bundle (see \Cref{prop: decomposition cone}). 
The remaining components are studied in \Cref{subsec: decomposing torsion}. The best result is obtained for $\ccF$ a diagonal sheaf: each component of $C(\ccF)$ is a vector bundle over its support (\Cref{theorem:cone_as_a_union}).

In \Cref{section:application-stable-maps} we use the notion of desingularization of a sheaf on a stack to define reduced Gromov Witten invariants in all genera: see \Cref{def:reduced_GW_invariants}. In \Cref{subsec:stable_maps} we recall how $\Gw$ can be naturally embedded as an open substack in an abelian cone over $\Pic$ following \cite{Chang-Li-maps-with-fields}. In \Cref{def:main moduli stable maps} we introduce the main component of $\Gw$, which is compatible with the open embedding in an abelian cone by \Cref{prop:main_component_Pr_open_in_tors_free}. In \Cref{subsec: blowups stable maps Pr} we consider a desingularization $\wPic \to \Pic$ and use it to base change $\Gw$ to a new space $\tGw$. These space is used in \Cref{subsec: def reduced GW} to define reduced Gromov-Witten invariants in any genus for a hypersurface in projective space (\Cref{def:reduced_GW_invariants}). We also show independence of the chosen desingularization in \Cref{invariance-red}. Finally, in \Cref{subsec:maps_with_fields} we recall maps with fields and consider the analogue of $\tGw$ for $p$-fields. These spaces can be used to compute reduced Gromov-Witten invariants (\Cref{red inv fields}). In \Cref{th-blow-up-maps} we describe the irreducible components of the blown-up moduli spaces.

In \Cref{sec:generalization} we extend the definition of reduced invariants to a large class of GIT quotients (see \Cref{recalling quasimaps}) and to quasimaps. In particular, we define reduced invariants for complete intersections, toric varieties and Grassmannians. For convenience, we review quasimaps to these GIT quotients in the first part of \Cref{sec:generalization}.

In \Cref{sec:comparisons} we compare the moduli spaces obtained from the Rossi desingularization and the Vakil--Zinger blow-up. While reduced invariants are independent of the birational model of $\Pic$, the induced moduli spaces can be different. We study charts of $\tGw$ and we show that the Rossi construction in genus one is different from the Vakil--Zinger blow-up.

\vspace{0.3cm}
\paragraph{\bf How to read this paper} 
Sections \ref{sec: background}--\ref{sec:components_abelian_cones} are self-contained and of independent interest. The schematic version of the results in \Cref{sec:desing_stacks} are explained in Sections \ref{sec:desing_sheaf}--\ref{sec:diagonalization}. The reader interested in reduced Gromov--Witten invariants can take the results in \Cref{sec:desing_stacks} for granted and read \Cref{section:application-stable-maps} to \Cref{sec:comparisons} directly. 

\vspace{0.3cm}
\paragraph{\bf Further work}

Our desingularizations do not come with a modular interpretation. It would be nice to have a modular interpretation of the resulting stack $\tGw$, either in the spirit of \cite{Hu-Li-GEnus-one-local-VZ-Math-ann-2010, Hu-Li-genus-two, niu1, niu2}, or a log interpretation as in \cite{ranganathan2019moduli} or \cite{kim2014compactification}.  It would be perhaps better to have a space of maps with more singular domains, as in \cite{battistella2020reduced, battistella2020smooth}. 

While a modular interpretation would be very interesting from a theoretical point of view, higher genus computations as done by Zinger in~\cite{Zinger-reduced-g1-CY} are likely to be hard. The genus two blow-up $\tGw$ already involves several rounds of blow-ups, and a localisation computation would inherit the complexity of the blow-up. We hope that our construction sheds new light on this beautiful problem and will encourage more mathematicians to work on it.

On the positive side, we hope this construction to be useful for proving \Cref{conj gv}. The main difference with \cite{Hu-Li-diagonalization} is that we blow up $\Pic$, instead of blowing up $\Gw$. The advantage of blowing up $\Pic$ is that now we have the ingredients used by Chang--Li, Lee--Oh and Lee--Li--Oh to prove \Cref{conj gv} (and therefore \Cref{conj one}) in genus one and two. More precisely, we have fairly simple moduli spaces of maps with fields over $\widetilde{\Pic}$, and these can be used to split the virtual fundamental class on $\tGw$. We hope to be able to prove \Cref{conj gv} without having explicit equations of $\Gw$, or a modular interpretation of $\widetilde{\Pic}$. We will address this problem in future work. 
\vspace{0.3cm}

\paragraph{\bf Acknowledgements} 
We are very grateful to Evgeny Shinder for useful discussions: the project started after Evgeny Shinder and Ananyo Dan pointed out to us the paper of Rossi \cite{Rossi}.  We also thank Jarod Alper and Michael Wemyss for discussions and for pointing out useful references.

We thank Aleksey Zinger for very useful correspondence and for detailed explanations on \cite{zinger2011comparison} and Francesca Carocci, J\'anos Koll\'ar, Dhruv Ranganathan, Robert Rogers, Evgeny Shinder and the anonymous referee for suggestions which contributed to an improved exposition.

We also wish to express our gratitude to David Rydh for his close reading of this draft, and for his many useful comments,  references and also   \Cref{thm:univ_flatification}.

ACR was supported by Fonds Wetenschappelijk Onderzoek (FWO) with reference G0B3123N.
CM was supported by a Dorothy Hodgkin fellowship.
EM and RP benefit from the support of the French government “Investissements d’Avenir” program integrated to France 2030, bearing the following reference ANR-11-LABX-0020-01 and the PRC ANR-17-CE40-0014. RP was supported by ERC Advanced Grant MSAG.
\section{Background}\label{sec: background}
In this section we recall several basic constructions and fix the notation used throughout the paper.
\subsection{The relative Grassmannian}\label{gras}
Let $X$ be a scheme with a fixed quasi-coherent sheaf $\ccE$. The Grassmannian functor $\underline{\Gr}_X^r(\ccE):((\Sch)/X)^{\op}\to (\Set)$ is given on objects by
\begin{equation}\label{eq:functor_of_points_Gr}
T\mapsto\{\ccE_T\twoheadrightarrow\ccQ\ | \ccQ\mbox{ is locally free of rank r }\}
\end{equation}
with $\ccE_T\coloneqq\ccE\otimes_{\ccO_X}\ccO_T$.

This functor is represented by a scheme $\Gr_X^r(\ccE)$ over $X$, which is projective if $\ccE$ is finitely generated. Moreover, the Grassmannian functor is compatible with base-change. In particular
\[\Gr_X^r(\ccO_X^{\oplus n})\cong\Gr(n,r)\times X\]
where $\Gr(n,r)$ is the usual Grassmannian of $(n-r)$-dimensional subspaces of $\bbC^n$ relative to a point.
Since it represents a functor, the relative Grassmannian $\Gr_X^r(\ccE)$ comes with a universal sheaf and a universal quotient sheaf, which is locally free of rank $r$:
\[
\ccE_{\Gr_X^r(\ccE)}\twoheadrightarrow\ccQ_{\Gr_X^r(\ccE)}.
\]

The relative Grassmannian admits a  The Pl\"{u}cker embedding:
\[
\lambda_{n,r}: \Gr_X^r(\ccO^{\oplus n})\to \Gr_X^1\left(\bigwedge^r \ccO^{\oplus n}\right)\cong \bbP_X^{m-1},
\]
with $m= {n\choose r}$. For the last isomorphism, consider an $X$-scheme $T$. A point of $\Gr_X^1(\ccO^{\oplus m})(T)$ is a surjection
\[
\ccO_T^{\oplus m}\twoheadrightarrow \ccL
\]
with $\ccL$ a line bundle on $T$. This is a pair of a line bundle and an $m$-tuple of generating sections, which is an object of $\bbP_X^{m-1}(T)$.

\subsection{Fitting ideals} \label{subsec:Fitting_ideals}
\begin{definition}
Let $M$ be a finitely presented $R$-module. Let $F\xrightarrow{\varphi} G\to M\to 0$ be a presentation with $F$ and $G$ free modules and $\rk(G)=r$. Given $-1\leq i < \infty$, the $i$-th Fitting ideal $F_i(M)$ of $M$ is the ideal generated by all $(r-i)\times (r-i)$-minors of the matrix associated to $\varphi$ after fixing basis of $F$ and $G$. We use the convention that $F_i(M) = R$ if $r-i\leq 0$ and $F_{-1}(M)=0$.
\end{definition}

Intrinsically, $F_i(M)$ is the image of the map $\bigwedge^{r-i} F\otimes \bigwedge^{r-i}G^*\to R$ induced by $\bigwedge^{r-i}\varphi\colon \bigwedge^{r-i} F \to \bigwedge^{r-i} G$. The $i$-th Fitting ideal is well-defined in that it does not depend on the chosen presentation. Since determinants can be computed expanding by rows and columns, it follows that there are inclusions
\[
   0=F_{-1}(M)\subset F_0(M) \subset F_1(M) \subset \ldots \subset F_k(M) \subset F_{k+1}(M) \subset \ldots
\]
It follows from the definition and right-exactness of tensor product that Fitting ideals commute with base change. That is, given $R\to S$ ring homomorphism and $M$ a finitely presented $R$-module, then
\[
    F_i(M\otimes_R S) = F_i(M)\cdot S.
\]
Similarly, for a scheme $X$ and $\ccF$ a quasi-coherent $\ccO_X$-module of finite presentation, we have ideal sheaves
\[
0=F_{-1}(\ccF)\subset F_0(\ccF)\subset\dots\subset F_n(\ccF)\subset\dots\subset \ccO_X
\]
which can be defined locally as described above. For $f:Y\to X$ a morphism of schemes, we have
\[
f^{-1}F_i(\ccF)\cdot\ccO_Y = F_i(f^*\ccF).
\]
Fitting ideals describe the locus on $X$ where the sheaf $\ccF$ is locally free of some rank. More precisely, we recall the following standard result (see for example  \cite[\href{https://stacks.math.columbia.edu/tag/05P8}{Tag 05P8}]{stacks-project}).
\begin{proposition}\label{prop:Fitting_ideals_rank_connection}
For any $n$, the sheaf $\ccF$ is locally free of rank $n$ on the locally closed subscheme $V(F_{n-1}(\ccF))\setminus V(F_n(\ccF))$ of $X$.
\end{proposition}

We also have the following result \cite[Lemma 0F7M]{stacks-project} (cf \cite[Lemma 1]{Lipman})
 relating Fitting ideals to the projective dimension of a module and the local freeness of its torsion-free quotient.

\begin{proposition} \label{prop:lipman F_r=f}

Let $R$ be a ring and $M$ be a finitely presented module over $R$. Let $r\geq 0$ be such that $F_r(M)=(f)$ for some non zero divisor $f\in R$ and $F_{r-1}=0$. Then
\begin{enumerate}
    \item $M$ has projective dimension $\leq 1$.
    \item $M(f)=\ker(M\xrightarrow{\cdot} f M)$ has projective dimension $\leq 1$.
    \item $M/M(f)$ is locally free of rank $r$.
    \item $M=M/M(f)\oplus M(f)$.
    
\end{enumerate}
\end{proposition}

The result in \cite[Lemma 1]{Lipman} also gives the following useful partial converse.

\begin{notation} Let $R$ be an integral domain and
let $M$ be an $R$ module. We denote the the torsion free part of $M$ by $M^\tf\coloneq M/\tor(M)$. 
\end{notation}

\begin{proposition}\label{prop:Lipman_converse}
If $R$ is a local ring and $M$ is a finitely presented module of projective dimension $\leq 1$ with $M/\tor(M)$ locally free of rank $r$, then $F_r(M)$ is invertible and $F_{r-1}(M)=0$.
\end{proposition}

In fact, we shall make use of the following result.
\begin{proposition}\label{thm: Lipman}
If $R$ is an integral domain and $M$ is a finitely presented module of projective dimension $\leq 1$ with $M/\tor(M)$ locally free of rank $r$, then $F_r(M)$ is locally free and $F_{r-1}(M)=0$.
\end{proposition}

\begin{proof}
    By \Cref{prop:Lipman_converse}, all we need to do is to check that the assumptions are preserved by localization at maximal ideals of $R$.
    
    By \cite[Proposition 11.154]{rotman}, $1\geq \pd(M)\geq \pd(M\otimes_R R_m)$ for every maximal ideal $m$ in $R$.
    
    Let $m$ be a maximal ideal in $R$. Then $R_m$ is a flat $R$-module by \cite[Theorem 11.28]{rotman} and $\tor(M\otimes_R R_m) = \tor(M)\otimes_R R_m $ because $R$ is integral. It follows that $(M\otimes_R R_m)^\tf = (M)^\tf\otimes_R R_m$ is locally free.
\end{proof}

\subsection{Abelian cones}\label{def ab cones}

Let $X$ be a Noetherian scheme. We recall the notions of cone and abelian cone over $X$ and collect some basic properties.

\begin{definition}\label{def: cone}
    Let $\ccA = \bigoplus_{d\geq 0} \ccA_d$ be a graded sheaf of $\ccO_X$-algebras such that the canonical map $\ccO_X\to \ccA_0$ is an isomorphism and such that $\ccA$ is locally generated by $\ccA_1$ as an $\ccO_X$-algebra. The \textit{cone of $\ccA$} is the scheme $\Spec_X(\ccA)$ equipped with the natural projection 
    \[
      \Spec_X(\ccA)\to X.
    \]
    The cone of $\ccA$ is \textit{abelian} if the natural morphism $\Symm(\ccA_1) \to \ccA$ is an isomorphism of $\ccO_X$-algebras. A morphism of cones is a morphism over $X$ induced by a graded morphism of sheaves of $\ccO_X$-algebras.
\end{definition}

\begin{definition}\label{def: abelian cone}

    Let $\ccF$ be a coherent sheaf on $X$. The \textit{abelian cone associated to $\ccF$} is
    \[
        C_X(\ccF) = \Spec_X(\Symm(\ccF))
    \]
    equipped with the natural projection to $X$.
\end{definition}
We will omit the subscript $X$ in the formation of relative spectra and cones whenever it is possible to do so without introducing ambiguity.

\Cref{def: abelian cone} is related to the total space of a locally free sheaf. If $\ccE$ is a locally free sheaf, then $C(\ccE)$ is a vector bundle, but some authors (e.g. \cite[B.5.5]{fulton_intersection}) prefer to define the total space of $\ccE$ as 
\[
    \Tot(\ccE) \coloneqq C(\ccE^*) = \Spec(\Symm(\ccE^*)),
\]
so that the sheaf of sections of $\Tot(\ccE)$ over $X$ is $(\ccE^*)^* \simeq \ccE$ by \Cref{lem: sections and functionals}. When working with sheaves that may not be locally free, it is advisable to use $C(\ccF)$ instead of $C(\ccF^*)$, see \Cref{lem: sections and functionals} and \Cref{ex: sections and functionals}.

To an abelian cone $\pi\colon C(\ccF)\to X$ we can associate two natural sheaves of $\ccO_X$-modules. The sheaf of sections $\Sect(C(\ccF))$ of the projection $\pi$ is given by
\[
    U\mapsto \Hom_{U}(U,C(\ccF)\mid_U).
\]
The sheaf of functionals $\Fun(C(\ccF))$ is given by 
\[
    U\mapsto \Hom_{\operatorname{Ab}}(C(\ccF)\mid_U,U\times \bbA^1),
\]
where $\Hom_{\operatorname{Ab}}$ denotes morphisms of abelian cones.

\begin{lemma}\label{lem: sections and functionals}
    Given a coherent sheaf $\ccF$ in a Noetherian scheme $X$, there are natural isomorphisms of $\ccO_X$-modules 
    \begin{enumerate}
        \item\label{item: sections} $\Sect(C(\ccF)) \simeq \ccF^*$ and
        \item\label{item: functionals} $\Fun(C(\ccF))\simeq \ccF$.
    \end{enumerate}
\end{lemma}

\begin{proof}
    It is enough to prove the statements in the affine case and for global sections. Let $X=\Spec R$ and $\ccF=\widetilde{M}$ for a Noetherian ring $R$ and a coherent $R$-module $M$.

    For \ref{item: sections}, we see that 
    \begin{align*}
        \Hom_{X}(X,C(\ccF)) &= 
        \Hom_{X}(X,\Spec\Symm \ccF) \simeq 
        \Hom_{R-\operatorname{alg}}(\Symm M,R)\\ 
        & \simeq
        \Hom_{R-\operatorname{mod}}(M,R) \simeq M^*.
    \end{align*}

    For \ref{item: functionals} we have that
    \[
        \Hom_{\operatorname{Ab}}(C(\ccF),X\times \bbA^1) = 
        \Hom_{\operatorname{Gr} R-\operatorname{mod}}(\Symm R,\Symm M) \simeq
        \Hom_{R/\operatorname{mod}}(R,M) \simeq
        M,
    \]
    where $\Hom_{\operatorname{Gr} R-\operatorname{mod}}$ denotes morphisms of graded $R$-modules and where we used that $X\times \bbA^1 \simeq C_X(\ccO_X)$.
\end{proof}

\begin{example}\label{ex: sections and functionals}
    Let $X=\Spec(R)$ with $R=\bbC[x]$ and let $I=(x)$ be the ideal of the origin $0$ and let $M=R/I$ viewed as an $R$-module, that is, $M$ is the skyscrapper sheaf supported at $0$. Then 
    \[
        C(M) \simeq \Spec(\bbC[x,y]/(xy)) \subseteq X\times \bbA^1.
    \]
    According to \Cref{lem: sections and functionals}, $C(M)$ has no non-zero sections because $M^* = 0$. On the other hand, $C(M)$ has non-zero functionals
    \[
        C(M)\to X\times \bbA^1\colon (x,y)\mapsto (x,\lambda y)
    \]
    for any $\lambda\in\bbC$. 
    
    The cone of the dual is trivial, in fact $M^*=0$, so $C(M^*)=X$, which has no non-zero sections or functionals.
    
\end{example}

\begin{lemma}\label{lem: cone nonsing iff  vb}
    For a cone $\pi\colon C = \Spec(\ccA)\to X$, the following are equivalent:
    \begin{enumerate}
        \item\label{item: C vb} $C$ is a vector bundle over $X$,
        \item\label{item: C abelian and lf} $C$ is abelian and $\ccA_1$ is locally free over $X$,
       
        \item\label{item: smooth cone} $\pi$ is smooth.
    \end{enumerate}
\end{lemma}

\begin{proof}
    The implication \ref{item: C vb} $\Rightarrow$ \ref{item: smooth cone} is clear and the equivalence \ref{item: C vb} $\iff$ \ref{item: C abelian and lf} is standard. The implication \ref{item: smooth cone} $\Rightarrow$ \ref{item: C vb} can be found in \cite[Lemma 1.1]{Behrend-Fantechi}.
\end{proof}

\section{Desingularizations of coherent sheaves}\label{sec:desing_sheaf}
In this section we introduce several constructions which ``desingularize" a coherent sheaf $\ccF$ on a base scheme $X$. 

\subsection{Definition of desingularizations on stacks}\label{def:desingularization}
We define our notion of desingularization and prove that it behaves well with composition. This part can be formulated directly for algebraic stacks instead of schemes, which will be useful later.
\begin{definition}\label{def: desingularization on stacks}
Let $\ffF$ be a coherent sheaf on an integral algebraic stack $\ffX$. A \textit{desingularization} of $\ffF$ is a morphism $p:\widetilde{\ffX}\to \ffX$ such that
\begin{enumerate}
    \item $\widetilde{\ffX}$ is integral,
    \item $p$ is birational and proper,
    \item\label{item: tf locally free} $(p^*\ffF)^\tf$ is a locally free sheaf.
\end{enumerate}
\end{definition}

Let us explain why a morphism as in \Cref{def: desingularization on stacks} deserves to be called a desingularization. The surjection $\ffF\to \ffF^\tf$ induces a closed embedding $C(\ffF^\tf)\hookrightarrow C(\ffF)$ of abelian cones over $\ffX$. By \Cref{lem: cone nonsing iff  vb}, the morphism $C(\ffF^\tf)\to \ffX$ is smooth if and only if $\ffF^\tf$ is locally free. Therefore, \Cref{item: tf locally free} in \Cref{def: desingularization on stacks} is equivalent to saying that the morphism $C_{\widetilde{\ffX}}((p^\ast\ffF)^\tf)\to \widetilde{\ffX}$ is smooth. 

\begin{remark}
    If $\ffX$ is a scheme and $\ffF$ is a non-zero coherent ideal sheaf, the usual blow-up of $\ffX$ at the closed subscheme defined by $\ffF$ is a desingularization of $\ffF$.
\end{remark}

\begin{lemma}\label{pull back desing}
Let $\ffX$ be an integral algebraic stack and $\ffF$ a coherent sheaf on $\ffX$. Let $p:\widetilde{\ffX}\to \ffX$ be a desingularization of $\ffF$ and let $q:\ffY\to \widetilde{\ffX}$ be a proper birational morphism. Then, the composition $p\circ q:\ffY\to \ffX$ is a desingularization of $\ffF$.
\end{lemma}
\begin{proof} The composition $r:=p\circ q$ is birational and proper, so all we need to prove is that $(r^*\ffF)^\tf$ is locally free. In the following we show that \[(r^*\ffF)^\tf\simeq q^*((p^*\ffF)^\tf).
\]
We have a commutative diagram of sheaves on $\ffY$
\[
\xymatrix{
0\ar[r]&\ffK_1\ar[r]\ar@{-->}[d]&q^*p^*\ffF\ar[r]\ar[d]&q^*((p^*\ffF)^\tf)\ar[r]\ar@{-->}[d]&0\\
0\ar[r]&\ffK_2\ar[r]&q^*p^*\ffF\ar[r]&(q^*p^*\ffF)^\tf\ar[r]&0
}
\]
where the map $q^*p^*\ffF\to q^*((p^*\ffF)^\tf)$ is the pull-back of the surjective map 
\[p^*\ffF\to (p^*\ffF)^\tf,\]
$\ffK_1$ and $\ffK_2$ are the corresponding kernels and the solid vertical map is the identity. Since the image of the composition $\ffK_1\to (q^*p^*\ffF)^\tf$ is generically zero and $\ffX$ is irreducible, we have that the morphism $\ffK_1\to (q^*p^*\ffF)^\tf$ is zero. By the universal property of cokernels this map factors through $q^*((p^*\ffF)^\tf)$, which gives the right vertical map in the diagram. The universal property of kernels gives the left vertical map in the diagram. We have that $\ffK_1$ and $\ffK_2$ are torsion sheaves. Let $\ffK_3$ be the cokernel of $\ffK_1\to \ffK_2$. By the Snake Lemma, we have an exact sequence 
\[0\to \ffK_3\to q^*((p^*\ffF)^\tf)\to (r^*\ffF)^\tf\to 0.\]
By assumption $(p^*\ffF)^\tf$ is locally free, so $q^*((p^*\ffF)^\tf)$ is locally free. Since $\ffX$ is irreducible and $\ffK_3$ is a torsion sheaf, we get that $\ffK_3=0$. This proves the claim. 
\end{proof}

\subsection{Rossi's construction}\label{ss:local_construction}

In the following we describe the desingularization construction proposed by Rossi in the analytic setup in \cite{Rossi}, which is very geometric in nature. The same construction was studied by Oneto and Zatini in the algebraic setup in \cite{Oneto-Zatini-1991}, under the name of Nash transformation. It gives a way of desingularizing coherent sheaves on Noetherian schemes taking the closure of a graph into a Grassmannian. We present Rossi's construction for Noetherian schemes; for stacks it will be presented in \Cref{sec:desing_stacks}.

\remark{We thank David Rydh for pointing out this more concise version of the construction.}

\begin{theorem}\label{thm:univ_flatification}
Let $X$ be a reduced Noetherian scheme, $\ccF$ a coherent $\mathcal{O}_X$-module, and 
$U \subset X$ {a schematically dense open} subset where $\ccF$ is locally free of constant rank $r$.  
Then there exists a projective morphism $f:\widetilde{X} \to X$ such that 
$\widetilde{X}_U \cong U$ and $(f^*\ccF)^\tf$
is locally free of rank $r$. 
Moreover, $\widetilde{X} \to X$ is universal for the following property:  
if $g: X' \to X$ is any morphism of schemes such that $(g^*\ccF)^\tf$ is locally free of rank $r$, then there exists a unique morphism 
$X' \to \widetilde{X}$ over $X$.
\end{theorem}

\begin{proof}
Consider the relative Grassmannian
\[
\pi \colon \Gr^r_X(\ccF) \longrightarrow X,
\]
which parametrizes rank-$r$ locally free quotients of $\ccF$.  
This scheme is projective over $X$.

On the open subset $U \subseteq X$, the sheaf $\ccF|_U$ is locally free of rank $r$, so it determines a section
\[
s \colon U \longrightarrow \Gr^r_X(\ccF).
\]
Let $\widetilde{X}$ be the scheme-theoretic closure of $s(U)$ inside $\Gr_X^r(\ccF)$.  
By construction, $f: \widetilde{X} \to X$ is projective and restricts to an isomorphism 
over $U$.

Over $\widetilde{X}$, the universal quotient on the Grassmannian restricts to a 
quotient
\[
\varphi: f^*\ccF \twoheadrightarrow G,
\]
with $G$ locally free of rank $r$ and $f^*\ccF|_U\cong G|_U$.  
Since $G$ is torsion-free and $U$ is schematically dense, $\varphi$ descends to $\overline{\varphi}:(f^*\ccF)^\tf\to G$. On the other hand, the kernel of $\varphi$ is supported away from the schematically dense $U$, so it is torsion. Hence $G = (f^*\ccF)^\tf$.

For the universal property,  let $g: X' \to X$ be a morphism such that $(g^*\ccF)^\tf$ is locally free of rank $r$.  
Then the quotient $g^*\ccF \twoheadrightarrow (g^*\ccF)^\tf$ determines a unique morphism 
$s': X' \to \Gr^r_X(\ccF)$ by the universal property of the Grassmannian. To see that $s'$ factors through $\tilde{X}$ we take $V$ the inverse of the maximum open subset of $X$ such that $\ccF$ is locally free. We claim that $V$ is schematically dense in $X'$. If $V$ is not schematically dense, there exists $Z$ a subscheme of $X'$ which does not intersect $V$ and which is mapped in the complement of $U$. By replacing $Z$ with the reduced structure, we may assume that $Z$ is reduced. Then, by the maximality of $U$, $(g^{*}\ccF)|_Z$ has higher rank. This contradicts $(g^{*}\ccF)^\tf$ is locally free. Thus $g^{-1}(U)$ is schematically dense. Then particular, $s'$ and $s\circ g$ agree on $g^{-1}(U)$, so this determines the morphism and we see that it factors uniquely 
through $\widetilde{X}$, the closure of $s(U)$.  
This establishes the claimed universal property.

\end{proof}

    If $X$ is an integral Noetherian scheme, a coherent sheaf $\ccF$ is locally free over the unique generic point $\xi\in X$. We define the \textit{generic rank} of $\ccF$:
\[
\rk(\ccF)\coloneqq \rk({\ccF|_{\xi}})=r.
\] Taking this as $U$ in \Cref{thm:univ_flatification} above gives a construction desingularizing $\ccF$. In particular, the morphism $\tilde{X}\to X$ is proper and birational.

\begin{definition}[The Rossi blow-up]\label{definition:Rossi-blowup}\label{definition:Rossi_blowup}
Let $X$ be an integral Noetherian scheme and $\ccF$ a coherent sheaf. \Cref{thm:univ_flatification} applied to $U$ the generic point of $X$ defines the Rossi blow-up $\Bl_{\ccF}X\to X$, a proper birational morphism. We call it the \textit{blow-up of $X$ at the coherent sheaf $\ccF$}.
\end{definition}

    A classical example of this construction is the Nash blow-up of $X$, which is the particular case $\ccF=\Omega_X^1$. This case is related to resolution of singularities, see \cite{Spivakovsky}. The general case appeared in \cite{Oneto-Zatini-1991} under the name of Nash transform.

   If $\ccF$ is a non-zero coherent ideal sheaf on $X$, then $\Bl_\ccF(X)$ is the usual blow-up and the above recovers its construction as a relative Proj.

The Rossi blow-up thus defined is a desingularization of $\ccF$ in the sense of \Cref{def: desingularization on stacks}. In fact, it is the minimal desingularization of $\ccF$ in the sense of the universal property in \Cref{thm:univ_flatification}.

\subsection{Local descriptions}
\label{subsec:Villamayor}

The geometric construction of \Cref{ss:local_construction} has some useful affine-local characterizations, which we shall use in \Cref{subsec: properties univ desing} to establish some properties of this construction. That these are local descriptions of the Rossi blow-up presented in the previous section follows from comparing the universal property in \cite{Villamayor} to that of \Cref{thm:univ_flatification}.
The first characterization, for affine integral Noetherian schemes follows closely the the original construction by Rossi \cite{Rossi}. Let $X=\Spec R$, and $\ccF=\widetilde{M}$ be the coherent sheaf associated to a coherent $R$-module $M$.  We can find for some $n$ a surjective morphism of sheaves
\begin{equation}\label{eq:surjection_to_F}
f: \ccO_X^{\oplus n}\twoheadrightarrow \ccF.
\end{equation}
Restricted to a non-empty open subscheme $U$ where $\ccF$ is locally free, $f$ gives a $U$-point of $\Gr_U(\ccO_U^{\oplus n},r)$, which is a morphism
\[
\Gamma_f : U\to \Gr_U(\ccO_U^{\oplus n},r).
\]
Then $\Bl_{\ccF}X$ is the closure of $\Gamma_f(U)$ in $\Gr_X(\ccO_X^{\oplus n},r)$ with the reduced induced structure, equipped with the morphism 
\[
    p\colon \Bl_{\ccF}(X)=\overline{\Gamma_f(U)} \to X
\]
obtained by restricting the natural projection $\Gr_X(\ccO_X^{\oplus n},r) \to X$.
This is a Zariski-local description of \Cref{definition:Rossi_blowup}, and can be seen to be independent of the choice of presentation of $\ccF$, see \cite{Rossi}. It is sometimes useful to work with this more explicit affine construction, as we do in the proof of \Cref{prop: isomorphic-blowups-tensor-line-bundle}.

Another construction of the blow-up $\Bl_\ccF X$ can be given in terms of Fitting ideals of the sheaf $\ccF$.
This construction applies to an affine integral Noetherian scheme $X$. Most of this ideas first appeared in \cite{Oneto-Zatini-1991}. We follow the exposition in \cite{Villamayor}. 

Fitting ideals (see \Cref{subsec:Fitting_ideals}) are related to ranks of modules and flatness. Indeed, the local rank of $M$ at a prime ideal $P$ of $R$ is $r$ if and only if $F_{-1}(M)\subseteq F_0(M) \subseteq \ldots \subseteq F_{r-1}(M) \subseteq P$ but $F_r (M) \not\subset P$. As a corollary, if $R$ is a domain then the generic rank of $M$ is $r$ if and only if $F_r(M)$ is the first non-zero Fitting ideal, and moreover $M$ is flat if and only if it is free, if and only if $F_r(M)=R$ and $F_{r-1}(M)=0$. 

The relationship between Fitting ideals and local freeness of the torsion-free part comes from Lipman's theorem (\Cref{thm: Lipman}).

In particular, the proposition shows that blowing up the first non-trivial Fitting ideal of $M$ will make $(p^*M)^{\tf}$ locally free (with $p$ the blow-up morphism). However, it is possible that $M^\tf$ is already locally free on $\Spec R$ even though its first non-trivial Fitting ideal is not principal. See \Cref{rmk:first_Fitting_not_minimal} for an example. Only if the projective dimension of $M$ is at most one, is the blow-up of the first non-trivial Fitting ideal the minimal desingularization of the corresponding sheaf. In order to find a minimal transformation of $\Spec(R)$ on which $M^\tf$ is locally free, Villamayor proposes in \cite{Villamayor} the following construction.

Given a finitely presented module $M$ of generic rank $r$ over a domain $R$ with fraction field $K$, define its norm to be the fractional ideal
\begin{equation}\label{def: norm}
    \norm{M} = \mathrm{Im}\left(\bigwedge^r M\to K\simeq\bigwedge^r M\otimes_R K\right).
\end{equation}

\begin{definition}[Villamayor's blow-up]\label{def:Villamayor_blowup}
    Let $R$ be a Noetherian integral ring, let $X=\Spec R$  and let $M$ be a finitely presented $R$-module of generic rank $r$. The \textit{blow-up of $X$ along $M$} is
        \[
        p: \Bl_MR\coloneqq \Bl_{\norm{M}}R\to R,
        \]
    where $\norm{M}$ is the norm of $M$ from \eqref{def: norm}.
\end{definition}

\begin{remark}\label{rmk:first_Fitting_not_minimal}
    In general, $\Bl_M X$ is not obtained by blowing up the first non-zero Fitting ideal of $M$. Indeed, let $I\subset R$ be an ideal and $M=R/I$. On the one hand, $M^\tf = 0$ is locally free of rank 0, which is the rank of $M$, so $\Bl_M X = X$. On the other hand, the first non-zero Fitting ideal of $M$ is $F_0(M) = I$, so $\Bl_{F_0(M)} X= \Bl_M X$ if and only if $I$ is principal.
    
    However, if $M$ has generic rank $r$ and projective dimension $\leq 1$, then $\Bl_M X=\Bl_{F_r(M)}X$.
\end{remark}
Note that any two ideals $I, J$ of $R$ which are isomorphic to $\norm{M}$ as fractional ideals define the same blow-up, up to \textbf{unique} isomorphism.

The same ideas apply in \cite{Villamayor} to any (Noetherian) ring $R$ if we restrict to finitely presented $R$-modules $M$ such that $M\otimes_R K(R)$ is a free $K(R)$-module, where $K(R)$ is the total quotient ring of $R$. However, for our purposes, we shall not need that generality. 

Below, we explain the connection between this definition and the theory of Fitting ideals.

\begin{construction}[{\cite[Remark 2.1]{Villamayor}}]\label{constr: replace_by_M1}
    Let $R$ be a domain and let $M$ be a finitely presented $R$-module of rank $r$. Choose generators $m_1,\ldots, m_N$ for $M$. Then there is a short exact sequence
    \[
        0\to P \to R^N\to M\to 0.
    \]
    Since $M$ has rank $r$, there are elements $p_1,\ldots, p_{N-r}$ in $P$ which induce a morphism $R^{N-r}\to R^N$ of rank $N-r$. Let $P_1\simeq R^{N-r}$ be the free module generated by $p_1,\ldots, p_{N-r}$ and let $M_1 = R^N/P_1$, that is, the following is exact
    \[
        0\to P_1 \to R^N\to M_1\to 0.
    \]
    Then $M_1$ has projective dimension at most 1, $\rk(M_1) = \rk(M)$, there is a natural surjection $M_1\to M$ and $M_1/\tor(M_1) = M/\tor(M)$.
\end{construction}

\begin{lemma}\label{rmk: Rossi blow-up with norm}
    Under the assumptions of \Cref{def:Villamayor_blowup}, let $M_1$ be the $R$-module associated to $M$ in \Cref{constr: replace_by_M1}. Then $F_r(M_1)$ and $\norm{M}$ are isomorphic as fractional ideals over $R$. In particular,
    
    \[
        \Bl_{\norm{M}} X = \Bl_{F_r(M_1)} X.
    \]
\end{lemma}

\begin{proof}
    See \cite[Proposition 2.5]{Villamayor}.
\end{proof}
\begin{remark}\label{rmk:rossi=villa}
    It is straight-forward to check that the algebraic definition of \Cref{def:Villamayor_blowup} is equivalent to the local version of the Rossi construction by taking the Pl\"{u}cker embedding of the relative Grassmannian. We omit the details, which are present in a previous draft.
\end{remark}

\subsection{Properties of the Rossi blow-up}\label{subsec: properties univ desing}

In light of \Cref{rmk:rossi=villa}, from now on we will identify the Rossi and Villamayor blow-ups of a coherent sheaf $\ccF$ on an affine integral Noetherian scheme $X$, both of which are local descriptions of \Cref{definition:Rossi_blowup}. In this section, we collect some properties of $\Bl_\ccF X$.

\begin{proposition}[Blow-up commutes with flat pullbacks]\label{prop:map_between_blowups}\label{cor:map_between_blowups}
    Let $f:Y\to X$ be a morphism of Noetherian integral schemes and let $\ccF$ be a coherent sheaf on $X$ of generic rank $r$.    
    If $f^*\ccF$ has generic rank $r$ then there is a unique morphism
    \[
    \begin{tikzcd}
    \Bl_{f^*{\ccF}}(Y)\ar[r,"\exists ! \widetilde{f}",dashed]\ar[d]{p_Y}& \Bl_{\ccF}(X)\ar[d]{p_X}\\
    Y\ar[r,"f"] & X
    \end{tikzcd}
    \]
    making the diagram commute. If, moreover, $f$ is flat, then the  square is Cartesian.   
\end{proposition}

\begin{proof}
A unique morphism $\widetilde{f}$ making the diagram commute exists by the universal property of $\Bl_\ccF(X)$, which is \Cref{thm:univ_flatification}. To show that the diagram is Cartesian, we work locally with Villamayor's description. We use \cite[Lemma 0805]{stacks-project}, which is the analogous result for blow-ups along ideal sheaves. This requires checking that $f^{-1}\norm{\ccF} \cdot \ccO_Y = \norm{f^*\ccF}$, which holds since $\ccF$ and $f^*\ccF$ have the same rank and the norm $\norm{\cdot}$ is a determinantal ideal. 

Indeed, we can work locally. Then we have $X=\Spec(A)$, $Y=\Spec(B)$, a ring homomorphism $f^\#\colon A\to B$ and $\ccF = \widetilde{M}$ for some finitely presented $A$-module $M$. To compute $\norm{\ccF}$, we take a presentation
\[
    \begin{tikzcd}
        A^{m} \arrow{r}{\Gamma}& A^n \arrow{r} & M\arrow{r}& 0,
    \end{tikzcd}
\]
we choose a submatrix $\Gamma'$ of $\Gamma$ consisting of $n-r$ columns of $\Gamma$ and then $\norm{\ccF}$ is represented by the ideal generated by all the minors $\Delta_i(\Gamma')$ of size $(n-r)\times(n-r)$ of $\Gamma'$. The choice of $\Gamma'$ must be so that this ideal is non-zero and such a choice exists because $\rk(\ccF) = r$. Then $f^{-1}\norm{\ccF}\cdot B$ is the ideal in $B$ generated by $f^\#(\Delta_i(\Gamma'))$ for all $i$. On the other hand, tensoring by $\otimes_A B$ we get a presentation
\[
    \begin{tikzcd}
        B^{m} \arrow{r}{f^*\Gamma}& B^n \arrow{r} & M\otimes_A B\arrow{r}& 0.
    \end{tikzcd}
\]
Since $f^* \ccF = \widetilde{M\otimes_A B}$ and since $\rk(f^* \ccF) = r$, we can compute $\norm{f^* \ccF}$ in the same manner, i.e., taking all the minors of size $(n-r)\times(n-r)$ of a submatrix $(f^*\Gamma)''$ consisting of $n-r$ columns of $f^*\Gamma$. This means that $\norm{f^* \ccF}$ is generated by $\Delta_i((f^\# \Gamma)'')$. We can actually choose $\Gamma'$ and $(f^*\Gamma)''$ so that they consist of the same columns, and in that case we are done because $f^\#$ is a ring homomorphism.
\end{proof}

\begin{proposition} \label{prop: isomorphic-blowups-tensor-line-bundle} Let $X$ be an integral Noetherian scheme and $\ccL$ a line bundle on $X$. Then we have a unique isomorphism 
\[\xymatrix{\Bl_{\ccF}(X)\ar[rr]^{\exists !\widetilde{\phi}}\ar[rd]_p&&\Bl_{\ccF\otimes \ccL}(X)\ar[ld]^q\\
&X
}\]
which makes the diagram commute.
\end{proposition}

\begin{proof} Work locally for affine $X$. Let $f: \ccO_X^{\oplus n}\to \ccF$ be a surjective morphism, let $S$ denote the kernel of $f$ and let $U$ an open subset of $X$ such that $S$ is a vector bundle. We thus obtain a short exact sequence
\[
0\to S\otimes \ccL|_U\to \ccO_U^{\oplus n}\otimes \ccL|_U\stackrel{f\otimes id}{\rightarrow} \ccF\otimes \ccL|_U\to 0.
\]
By possibly shrinking $U$ we may assume we have an isomorphism $g:\ccL|_U\simeq\ccO_U$. We thus obtain a commutative diagram
\[\xymatrix{\Gamma_f\ar[r]\ar[d] &\Gr_X(\ccO_X^{\oplus n},r)\ar[d]\\
\Gamma_{f\otimes id}\ar[r]&\Gr_X(\ccO_X^{\oplus n},r)
}
\]
where the right vertical arrow is induced by $g$. This gives a morphism $\widetilde{\phi}:\overline{\Gamma}_f\to \overline{\Gamma}_{f\otimes id}$ and proves the claim.
\end{proof}

\begin{proposition}\label{prop:functoriality,sheaves,injective,exact,sequence,Villamaoyr}
    Let $X$ be a Noetherian integral scheme and $\ccE$, $\ccF, \ccG$ be coherent $\ccO_X$-modules. Assume that we have an exact sequence $0\to \ccE\to \ccF\to \ccG \to 0$. 
    \begin{enumerate}
        \item If the sequence is locally split and $\ccE$ is locally free, then there is an isomorphism $\Bl_{\ccF}X \simeq \Bl_{\ccG} X$.
        \item If $\ccG$ is locally free, then there is an isomorphism $\Bl_{\ccF}X \simeq \Bl_{\ccE} X$.
    \end{enumerate}
  
  \end{proposition}

\begin{proof}

   It is enough to prove the statements locally. Indeed, if $p_\ccF\colon \Bl_\ccF X \to X$ and $p_\ccG\colon \Bl_\ccG X \to X$ are the natural projections, then $\Bl_\ccF X \simeq \Bl_\ccG X$ if and only if $(p_\ccF^*\ccG)^\tf$ and $(p_\ccG^*\ccF)^\tf$ are locally free, and these are local statements. 
   
   Therefore, we may assume that we have $\ccF\simeq\ccE\oplus\ccG$. With this, we have that 
    \begin{equation}\label{wedge-ex-seq}
         \wedge^{\rm top}\ccF = \wedge^{\rm top}\ccE \otimes \wedge^{\rm top}\ccG.
    \end{equation}

   From the alternative description of \Cref{def:Villamayor_blowup}, or from composing with a Pl\"{u}cker embedding, it's clear that $\Bl_{\ccG}X\simeq \Bl_{\wedge^{\rm top}\ccG}X$ and $\Bl_{\ccF}X\simeq \Bl_{\wedge^{\rm top}\ccF}X \simeq \Bl_{\wedge^{\rm top}\ccE \otimes\wedge^{\rm top}\ccG}X$ using \Cref{wedge-ex-seq}.  Suppose now that $\ccE$ is locally free, then $\wedge^{\rm top}\ccE$ is a line bundle and we conclude that $\Bl_{\ccF}X \simeq \Bl_{\ccG} X$ by \Cref{prop: isomorphic-blowups-tensor-line-bundle}.
    
    If $\ccG$ is locally free, the sequence is locally split and a similar argument to the one above shows that $\Bl_{\ccF}X \simeq \Bl_{\ccE} X$.
\end{proof}

In the following we discuss a more general situation, when we have an open $U\subset X$ such that $0\to \ccE_U\to \ccF_U\to \ccG_U \to 0$, with $\ccE_U$ locally free.

\begin{proposition}\label{prop: one implies k}
     Let $\ccF$ and $\ccG$ sheaves of ranks $r+a$ and $r$ on an integral scheme $X$ and $\ccO^{n+a}\to\ccF$ and $\ccO^{n}\to\ccG$ surjective morphisms. Suppose there exists $U\subset X$ an open subset and a commutative diagram 
 \[
 \begin{tikzcd}
     0\ar[d]&0\ar[d]\\
 \ccO_U^{n}\ar[d,"i"]\ar[r]& \ccG|_U\ar[r]\ar[d,"f"]&0\\
\ccO_U^{n+a}\ar[r]\ar[d]& \ccF|_U\ar[r]\ar[d]&0\\
\ccO_U^a\ar[r,"h"]\ar[d]&\ccE_U\ar[d]\\
0&0
 \end{tikzcd}
 \]
 with the columns split short exact sequences and $h$ an isomorphism. Then, we get an induced morphism $$\Bl_{\ccG}X\to \Bl_{\ccF}X.$$ In particular, $\Bl_{\ccG}X$ is a desingularisation of $\ccF$.
  \end{proposition}    

  \begin{proof}
   It suffices to prove this for $X$ affine. The diagram in the hypothesis gives a commutative diagram  
   \[
 \xymatrix{U\ar[r]\ar[d]&X\times \Gr(n,r)\ar[d]\\
 U\ar[r]&X\times \Gr(n+a,r+a)
 }
 \]
 where the vertical map $\Gr(n,r)\to\Gr(n+a,r+a)$ is $$(V\twoheadrightarrow Q)\mapsto (V\oplus W\twoheadrightarrow Q\oplus h(W)).$$
 Here we denoted the fiber of the vector bundle $\ccO_U^a$ by $W$ and (by a slight abuse of notation) the induced map by $h$.
 This gives a morphism between the closures $\Bl_{\ccG}X\to \Bl_{\ccF}X$.

 \end{proof}
 
 We record here a fairly immediate corollary of the above proposition, which will be useful as stated in \Cref{section:application-stable-maps}.
 
 \begin{corollary}\label{cor: one implies k}
     Let $\ccF$ and $\ccG$ sheaves of ranks $r+a$ and $r$ on an integral scheme $X$ and $\ccO^{n}\to\ccF$ and $\ccO^{n+a}\to\ccG$ surjective morphisms. Suppose there exists $U\subset X$ an open subset and a commutative diagram 
 \[
 \begin{tikzcd}
    &0\ar[d]&0\ar[d]\\
 0\ar[r]&\ccG^*|_U \ar[d,"g"]\ar[r]&\ccO_U^{n} \ar[d,"q"]\\
0\ar[r]&\ccF^*|_U\ar[r]\ar[u, shift left]\ar[d]& \ccO_U^{n+a}\ar[d]\ar[u, shift left]\\
&\ccE_U^*\ar[r,"h"]\ar[d]\ar[u, shift left]&\ccO_U^a\ar[d]\ar[u, shift left]\\
&0 &0
  \end{tikzcd}
 \]
 with the columns split short exact sequences and $h$ an isomorphism. Then, we get an induced morphism $$\Bl_{\ccG}X\to \Bl_{\ccF}X.$$ In particular, $\Bl_{\ccG}X$ is a desingularisation of $\ccF$.
  \end{corollary}
  \begin{proof}
   This follows by dualising the statement in \Cref{prop: one implies k}. Note that $g$ is not the dual of $f$, but the dual of a splitting of $f$.
   
   Alternatively, one can copy the proof above. The only difference is that a map $\Gr(r, n)\to\Gr(r+a,n+a)$ is \[(S\hookrightarrow V)\mapsto (g(S)\oplus W\hookrightarrow V\oplus h(W)).\qedhere\]
 \end{proof}

\begin{remark}
    In general, $\Bl_{\ccF}X$ is not isomorphic to $\Bl_{\ccF^*}X$. For example let $X$ be a normal scheme and $\ccF$ an ideal sheaf. Then $\ccF^{*}$ is reflexive and it has rank one, so it is an invertible sheaf. This shows that $\Bl_{\ccF^{*}}X\simeq X$. If $\ccF$ is not locally free (see e.g. \Cref{point in plane}), then we have $\Bl_{\ccF}X\neq X$.
\end{remark}

\section{Diagonalization}\label{sec:diagonalization}
This section may be skipped for a first reading. The Rossi construction suffices to prove the main result \Cref{main_result_intro}. We add this new blowup because we believe that this is  a first step for proving \Cref{conj gv}.

The diagonalization process for certain coherent sheaves is introduced by Hu and Li in \cite{Hu-Li-diagonalization} and by Grivaux in \cite{Grivaux}. A coherent sheaf $\ccF$ on an integral Noetherian scheme $X$ can locally be written as the cokernel of a morphism of locally free sheaves $\varphi\colon \ccE^{-1}\to \ccE^0$. Blowing up all the Fitting ideals of $\ccF$ desingularizes both $\ccF$ and the kernel of $\varphi$, and makes the morphism $\varphi$ diagonalizable (see \Cref{def:diagonalizable}). We summarize the construction for schemes and its universal property and explore the possibility of finding a minimal blow-up which also desingularizes all the components of the abelian cone associated to $\ccF$. Applied to the moduli space of maps, this construction will be used in \Cref{section:application-stable-maps} to desingularize all the components of $\Gw$. All schemes are assumed to be Noetherian and integral in this section.

\subsection{Diagonalizable morphisms and diagonal sheaves}\label{subsec: diagonal sheaves and morphsims}
We recall the notion of diagonalizable morphism of locally free sheaves from \cite{Hu-Li-diagonalization} and introduce the notion of diagonal sheaf. We show that these two notions are equivalent in \Cref{prop:diagonalizable_definitions_equivalent}, in the sense that a morphism is diagonalizable if and only if its cokernel is diagonal.

\begin{definition}[Diagonalizable morphism {\cite[Definition 3.2]{Hu-Li-diagonalization}}]\label{def:diagonalizable}
    Let $X$ be a scheme. A morphism $\varphi\colon \ccO_X^{\oplus p}\to \ccO_X^{\oplus q}$ is \textit{diagonalizable} if there are direct sum decompositions by free sheaves
    \begin{equation}\label{eq:diagonalizable_decomposition}
            \ccO_X^{\oplus p} = G_0 \oplus \bigoplus_{i=1}^\ell G_i \quad \text{and} \quad \ccO_X^{\oplus q} = H_0 \oplus \bigoplus_{i=1}^\ell H_i
    \end{equation}
    with $\varphi(G_i)\subseteq H_i$ for $0\leq i \leq \ell$ such that
    \begin{enumerate}
        \item $\varphi\mid_{G_0} = 0$;
        \item for every $1\leq i\leq \ell$, there is an isomorphism $I_i: G_i\to H_i$; 
        \item the morphism $\varphi\mid_{G_i} : G_i \to H_i$ is given by $f_i I_i$ for some $0\neq f_i \in\Gamma(\ccO_X)$;
        \item $(f_{i+1}) \subsetneqq (f_i)$.
    \end{enumerate}
    More generally, a morphism $\varphi\colon E^{-1} \to E^0$ of locally free sheaves on $X$ is \textit{locally diagonalizable} if $X$ admits an open cover which trivializes $E^{-1}$ and $E^0$ simultaneously and on which $\varphi$ is diagonalizable.
\end{definition}

\begin{example}
    If $X=\Spec(R)$ for a principal ideal domain $R$, then every morphism $\varphi\colon R^p\to R^q$ is diagonalizable in the sense of \Cref{def:diagonalizable} and the diagonal form associated to $\varphi$ is called the Smith normal form of $\varphi$.
\end{example}

We will be interested in the coherent sheaves arising as kernels and cokernels of such diagonalizable morphisms. 

\begin{proposition}\label{ker-diag}
    Let $X$ be a Noetherian integral scheme and let $\varphi\colon E^{-1}\to E^0$ be a locally diagonalizable morphism between locally free sheaves on $X$. Then $\ker(\varphi)$ is locally free.
\end{proposition}

\begin{proof}
    The question is local, so we can assume that $E^{-1} = \ccO_X^{\oplus p}$ and $E^{0} = \ccO_X^{\oplus q}$, and that they admit decompositions as in \Cref{eq:diagonalizable_decomposition}. Then $\ker(\varphi) = G_0$ is free.
    
\end{proof}

\begin{definition}
\label{defi:diag_sheaf}
\label{defi:diag_morphism_sheaves}
We say  that a coherent sheaf $\ccF$ on a scheme $X$ is \textit{diagonal} if all the Fitting ideal sheaves $F_i(\ccF)$ are locally principal.
\end{definition}

\begin{remark}\label{rmk: diagonal implies tf lc and tor dim}
    By \Cref{thm: Lipman}, if $\ccF$ is diagonal then $\ccF^\tf = \ccF/\tor(\ccF)$ is locally free and $\ccF$ has tor-dimension at most 1.
\end{remark}

\begin{proposition}\label{prop:diagonalizable_definitions_equivalent} A morphism $\varphi:E^{-1}\to E^0$ of locally free sheaves is locally diagonalizable if and only if the coherent sheaf $\Coker(\varphi)$ is diagonal.
\end{proposition}
\begin{proof}
    This result is contained in the proof of \cite[Proposition 3.13]{Hu-Li-diagonalization}. Observe that the Fitting ideals $F_i(\ccF)$ are just the determinantal ideals $\Delta_{(q-i)\times(q-i)}(\varphi)$, where $q=\rk(E^0)$. If $\varphi$ is locally diagonalizable, take an open where it is of the form \eqref{eq:diagonalizable_decomposition}. Then the Fitting ideals of $\ccF$ are generated by products of the $f_i$'s, so are principal in this open.

    On the other hand, if $\ccF$ is diagonal, we can cover $X$ by affine opens where all $F_i(\ccF)$ are principal and where the $E^i$'s are simultaneously trivialized. We quickly sketch how \cite[Proposition 3.13]{Hu-Li-diagonalization} produces a decomposition as in \eqref{eq:diagonalizable_decomposition}, by possibly further restricting. In the affine open $\Spec(R)\subset X$, the morphism $\varphi$ is given by 
    \[
    \Gamma = (a_{i,j})
    \]
    $i\in\{1,\dots,p\}$, $j\in\{1,\dots ,q\}$, $a_{i,j}\in R$. The Fitting ideal $F_{q-1}(\ccF)=\Delta_{1\times 1}(\Gamma)$ is principal if and only if, after further localization, there is an entry $a_{i_0,j_0}$ which divides every other entry $a_{i,j}$. In that case, one can perform row and column operations to put $\Gamma$ in the following form
    \begin{equation*}
    \left(
        \begin{array}{c|ccc}
            a_{i_0,j_0} &  0 & \ldots & 0\\
            \hline
            0  &  & \\
            \vdots & & \Gamma' \\
            0  
        \end{array}
        \right)
    \end{equation*}
    with $\Gamma'$ a matrix of smaller size. The same argument works recursively since the remaining Fitting ideals of $\ccF$ and those of $\Gamma'$ differ by the principal ideal $(a_{i_0,j_0})$.
\end{proof}

\begin{example}
    Any smooth curve $X$ can be covered by affine open subschemes of the form $\Spec(R)$ with $R$ a principal ideal domain. Therefore every coherent sheaf on $X$ is locally diagonal and every morphism of locally free sheaves on $X$ is locally diagonalizable.  
\end{example}

We are interested in morphisms that transform a given coherent sheaf in a diagonal sheaf.

\begin{definition}
\begin{enumerate}
    \item \label{defi:diag,2} Given a scheme $X$ and a coherent sheaf $\ccF$, a \textit{diagonalization} of $\ccF$ is a morphism $f:\widetilde{X}\to X$ such that $f^*\ccF$ is diagonal. 
    \item\label{defi:diag,3} Given a scheme $X$ and a morphism of locally-free sheaves $\varphi: E^{-1}\to E^0$, a \textit{diagonalization} of $\varphi$ is a morphism $f:\widetilde{X}\to X$ such that $f^*\varphi$ is locally diagonalizable and $\rk(\Coker(\varphi))=\rk(f^*\Coker(\varphi))$.
\end{enumerate}
\end{definition}

\begin{remark}\label{remark:diagonal_sheaf_diagonalizing_presentation}
From \Cref{prop:diagonalizable_definitions_equivalent}, we see that diagonalizing a coherent sheaf $\ccF$ is equivalent to diagonalizing any local presentation $E^{-1}\to E^0\twoheadrightarrow\ccF$ by locally free sheaves. 
\end{remark}

\begin{remark}\label{remark:diagonalization_base_changes}
If $\varphi:E^{-1}\to E^0$ is a locally diagonalizable morphism on a scheme $X$, and $f:Y\to X$ is any morphism of Noetherian schemes, $f^*\varphi$ is locally diagonalizable.

Similarly, if $\ccF$ is diagonal, $f^*\ccF$ is diagonal.

Note that the generic ranks of $\ccF$ and $f^*\ccF$ will be different in general for non-dominant morphisms.
\end{remark}

\subsection{Construction of the Hu--Li blow-up}\label{subsec: construction Hu Li blowup}

We recall the construction of the minimal diagonalization of a sheaf, introduced in \cite{Hu-Li-diagonalization}, which we call the Hu-Li blow-up of a scheme along a sheaf.

\begin{definition}[Maximal rank]\label{def:maximal rank}
    Let $X$ be an integral Noetherian scheme and $\ccF$ a coherent sheaf of generic rank $r$. The maximal rank $\mrk(\ccF)$ is 
   \[ \mrk(\ccF)=\max_{p\in X}\{\rk(\ccF|_p)\}\geq r
   \] which is the maximum rank of $\ccF$ when restricted to a closed point of $p\in X$. Equivalently, $\mrk(\ccF)$ is such that the Fitting ideals $F_{\mrk(\ccF)}(\ccF)$ is $\ccO_X$ and $F_{\mrk(\ccF)-1}(\ccF)\neq \ccO_X$, with the convention that $F_{-1}(\ccF)=0$. 
    \end{definition}
    \begin{remark}
    The above $\mrk(\ccF)$ is finite. Indeed, the ascending chain condition on the Fitting ideals 
    \[F_{-1}(\ccF)\subset F_0(\ccF)\subset\dots\subset F_n(\ccF)\]
    guarantees that there is some $\mrk(\ccF)$ such that $F_{\mrk(\ccF)}(\ccF)=F_{\mrk(\ccF)+1}(\ccF)=\dots$. 
    
    Moreover, for any affine open $U\subset X$, the ascending chain of Fitting ideals stabilizes at $\ccO_U$, since $\ccF|_U=\widetilde{M}$ for a finitely generated module $M$. So the chain above must stabilize at $\ccO_X$.
    \end{remark}

    \begin{remark}
    There is a closed point $q\in X$ such that $\rk(\ccF|_q)=\mrk(\ccF)$, and such that we have a resolution
    \[
    \ccO_q^{\oplus p}\to\ccO_q^{\oplus \mrk(\ccF)}\to\ccF|_q\to 0.
    \]
    However, $\ccF$ may not be generated globally by $\mrk(\ccF)$ sections. Indeed, it may not be globally generated at all!
    \end{remark}

\begin{construction}[Hu--Li blow-up]\label{constr: diagonalization_blowup}
Let $X$ be an integral Noetherian scheme and $\ccF$ a coherent sheaf of general rank $r$ and maximal rank $r_2$.
    Recall from \Cref{subsec:Fitting_ideals} that the Fitting ideals of $\ccF$ satisfy a chain of inclusions $F_{-1}(\ccF) \subseteq F_0(\ccF)\subseteq \ldots$, that $F_{r_2}(\ccF) = \ccO_X$ and that $F_0(\ccF) = \ldots = F_{r-1}(\ccF) = 0$ because $\ccF$ has rank $r$. 

    Let 
    \[
        p:\bld = \Bl_{F_{r}(\ccF) \cdot \ldots \cdot F_{r_2-1}(\ccF)} X \to X
    \]
    By \cite[Lemma 080A]{stacks-project}, $\bld$ can also be constructed by successively blowing up $X$ along (the total transforms of) the Fitting ideals of $\ccF$, that is 
    
    \[  
        \begin{tikzcd}
            \bld = X_r \arrow{r}{p_r} & \ldots\arrow{r} & X_{r_2-2}\arrow{r}{p_{r_2-2}} & X_{r_2-1}\arrow{r}{p_{r_2 - 1}} & X
        \end{tikzcd}
    \]
    where  
    \begin{itemize}
         \item $X_{r_2-1} = \Bl_{F_{r_2-1}(\ccF)} X$,
        \item $X_{r_2-2} = \Bl_{F_{r_2-2}(p_{r_2-1}^\ast\ccF)} X_{r_2-1} = \Bl_{p_{r_2-1}^{-1} F_{r_2-2}(\ccF) \ccO_{X_{r_2-1}}} X_{r_2-1}$ and
        \item $X_i = \Bl_{F_{i}(p_i^* p_{i+1}^* \ldots p_{r_2-1}^*\ccF)} X_{i+1} $ for all $i$ with $r\leq i \leq r_2-2$.
    \end{itemize}
    Each $p_i$ is the natural morphism coming from the blow-up construction and we denote by $p$ the composition $p_{r_2-1}\circ \ldots \circ p_r$.
\end{construction}

\subsection{The universal property of the Hu--Li blow-up}\label{subsec: univ property HL blowup}

We are now ready to state the minimality properties for \Cref{constr: diagonalization_blowup}. We can formulate a universal property for the morphism $\varphi$ or, in light of \Cref{remark:diagonal_sheaf_diagonalizing_presentation}, we can formulate it to only depend on the cokernel sheaf $\ccF$.

\begin{theorem}[Universal property of $\bld$ \cite{Hu-Li-diagonalization}]\label{thm:univ_flatificationerty_HL_principal_Fitting_ideals}
    Let $X$ be a Noetherian integral scheme and $\ccF$ a coherent sheaf on $X$ of generic rank $r$. The natural projection $p\colon \bld\to X$ satisfies that 
    \begin{enumerate}
        \item the sheaf $p^*\ccF$ has generic rank $r$ and
        \item the Fitting ideal $F_i(p^*\ccF)$ is locally principal for all $i$.
    \end{enumerate}  
    Moreover, $p\colon \bld\to X$ satisfies the following universal property: for any morphism $f\colon Y\to X$ of Noetherian integral schemes such that 
    \begin{enumerate}
        \item the sheaf $f^*\ccF$ has generic rank $r$ and
        \item the Fitting ideal $F_i(f^*\ccF)$ is locally principal for all $i$,
    \end{enumerate}  
    there is a unique morphism $f'\colon Y\to \bld$ factoring $f$.
    \[
    \begin{tikzcd}
    Y \ar[dr,"f"']\ar[r,"{\exists !f'}", dashed]&\bld\ar[d,"p"]\\
    &X
    \end{tikzcd}
\]
\end{theorem}

\begin{proof}
    Let $r_2$ denote the maximal rank of $\ccF$. Then $r=r_2$ if and only if $\ccF$ is locally free, in which case $\bld= X$ clearly has this property. 
    
    Otherwise, we must have $r_2\geq r$, so $\bld$ is defined by the sequence of blow-ups in \Cref{constr: diagonalization_blowup}.  By construction, $p$ is dominant and each Fitting ideal of $p^\ast \ccF$ is locally principal. The universality follows from the universal property of the usual blow-up as in \cite[Proposition 7.14]{Hartshorne_AG}, using that $F_i(f^*\ccF)$ is non-zero for all $r\leq i \leq r_2$, which holds by the assumption that $f^*\ccF$ and $\ccF$ have the same generic rank.
\end{proof}

\begin{theorem}[Universal property of diagonalization \cite{Hu-Li-diagonalization}]\label{thm:universal_property_diagonalization}
    Let $\varphi:E^{-1}\to E^{0}$ be a morphism of locally-free sheaves on a Noetherian integral scheme $X$. Let $\ccF=\Coker(\varphi)$ and let $\bld$ as in  
    \Cref{constr: diagonalization_blowup}.  Then, the natural projection $p\colon \bld\to X$ is a diagonalization of $\varphi$. Moreover, $p$ satisfies the following universal property: for any morphism $f: Y\to X$ such that $f^*\varphi$ is locally diagonalizable and $\rk(f^*\ccF) = \rk(\ccF)$, there is a unique morphism $f':Y\to\bld$ factoring $f$.
    \[
    \begin{tikzcd}
    Y \ar[dr,"f"']\ar[r,"{\exists !f'}", dashed]&\bld\ar[d,"p"]\\
    &X
    \end{tikzcd}
\]
\end{theorem}

\begin{proof}
    Follows immediately from \Cref{thm:univ_flatificationerty_HL_principal_Fitting_ideals} and \Cref{prop:diagonalizable_definitions_equivalent}.
\end{proof}

\subsection{Properties of the Hu--Li blow up}\label{subsec:properties_HL}

We collect properties of $\bld$.

\begin{proposition}\label{prop:map_between_blowups_diagonal}\label{cor:map_between_blowups_diagonal}
    Let $f:Y\to X$ be a morphism of Noetherian integral schemes and let $\ccF$ be a coherent sheaf on $X$ of generic rank $r$. 
    If $f^*\ccF$ has generic rank $r$ then there is a unique morphism
    \[
    \begin{tikzcd}
    \bldY\ar[r,"\exists !\widetilde{ f}",dashed]\ar[d]& \bld\ar[d]\\
    Y\ar[r,"f"] & X
    \end{tikzcd}
    \]
    making the diagram commute. 
   
    If, moreover, $f$ is flat, then the square is Cartesian.
  
\end{proposition}

\begin{proof}
    A unique morphism $\widetilde{f}$ making the diagram commute exists by the universal property \Cref{thm:universal_property_diagonalization}. To see that diagram is Cartesian we apply \cite[Lemma 0805]{stacks-project} to each of the blow-ups defining $\bld$, using that the formation of Fitting ideals is compatible with pullbacks.
\end{proof}

\begin{proposition}\label{prop:tensor,line,HL,blowup}
    Let $X$ be a Noetherian integral scheme, let $\ccF$ a coherent sheaf on $X$ and let $\ccL$ be a line bundle on $X$. Then there is a unique isomorphism 
    \[\xymatrix{\blHL_{\ccF}(X)\ar[rr]^{\exists !\widetilde{\phi}}\ar[rd]_p&&\blHL_{\ccF\otimes \ccL}(X)\ar[ld]^q\\
    &X
    }\]
    which makes the diagram commute.
\end{proposition}

\begin{proof}
    By \Cref{thm:univ_flatificationerty_HL_principal_Fitting_ideals}, a unique factorization $\widetilde{\phi}$ of $p$ through $q$ exists if and only if $F_i(p^*(\ccF\otimes \ccL))$ is locally principal for all $i$, and this holds because $F_i(p^*\ccF)$ is locally free for all $i$. Indeed, choose an open cover of $X$ trivializing $\ccL$. The preimage by $p$ of this cover induces a cover of $\Bl_\ccF (X)$ where $p^*(\ccF\otimes \ccL) \simeq p^*\ccF$. This shows that $F_i(p^*(\ccF\otimes \ccL)) \simeq F_i(p^*\ccF)$ locally, so $\widetilde{\phi}$ exists. The same argument shows there is a unique factorization of $q$ through $p$, which must be the inverse of $\widetilde{\phi}$ by uniqueness.
\end{proof}

\begin{proposition}\label{prop:functoriality_sheaves_injective_diagonal_version}
    Let $X$ be a Noetherian integral scheme and $\ccE$, $\ccF, \ccG$ be coherent $\ccO_X$-modules. Assume that we have an exact sequence $0\to \ccE\to \ccF\to \ccG \to 0$. 
    \begin{enumerate}
        \item If the sequence is locally split and $\ccE$ is locally free, then there is an isomorphism $\blHL_{\ccF}X \simeq \blHL_{\ccG} X$.
        \item If $\ccG$ is locally free, then there is an isomorphism $\blHL_{\ccF}X \simeq \blHL_{\ccE} X$.
    \end{enumerate}

\end{proposition}  
    
\begin{proof}
   It is enough to prove the statement locally, so we may assume that we have $\ccF\simeq\ccE\oplus\ccG$. With this, we have that 
    \begin{equation}\label{eq: Fitting ideals direct sum}
         F_\ell(\ccE \oplus \ccG) = \sum_{k+k'=\ell} F_k(\ccE) F_{k'}(\ccG)
    \end{equation}
    by \cite[Lemma 07ZA]{stacks-project}. If $\ccG$ is locally free, the sequence is locally split, therefore by symmetry it is enough to show one of the statements. Without loss of generality, suppose that $\ccG$ is locally free, therefore $F_{k'}(\ccG) = 0$ for all $k'<\rk(\ccG)$ and $F_{k'}(\ccG) = \ccO_X$ for all $k'\geq \rk(\ccG)$ by \cite[Lemma 07ZD]{stacks-project}. Combining this fact with the chain of inclusions $F_0(\ccE)\subseteq F_1(\ccE)\subseteq \ldots$, it follows that 
    \[
        F_\ell(\ccF) = F_\ell (\ccE \oplus \ccG) = \begin{cases}
            0 &  \text{if } \ell < \rk(\ccG)\\
            F_{\ell-\rk(\ccG)}(\ccE) &\text{if } \ell \geq \rk(\ccG)
        \end{cases}
    \]
    This means that the collection of Fitting ideals of $\ccF$ and $\ccE$ agree, so $\blHL_{\ccF}X \simeq \blHL_{\ccE} X$.
\end{proof}

\begin{proposition}\label{prop:blowup_direct_sum}
    Let $X$ be a Noetherian integral scheme and $\ccF$ a coherent $\ccO_X$-module. Then for every positive integer $n$
    \[
        \blHL_\ccF X = \blHL_{\ccF^{\oplus n}} X.
    \]
\end{proposition}

\begin{proof}
    There is a natural morphism $\blHL_\ccF X  \to \blHL_{\ccF^{\oplus n}} X$ over $X$. To see this, let $p\colon \blHL_\ccF X \to X$ be the natural projection. Then $p^\ast \ccF$ is diagonal and it follows from \Cref{def:diagonalizable} that $p^\ast (\ccF^{\oplus n}) = (p^\ast\ccF)^{\oplus n}$ is diagonalizable too. Then apply \Cref{thm:universal_property_diagonalization} to get the desired morphism.

    Conversely, we show that there is a natural morphism $\blHL_{\ccF^{\oplus n}} X\to \blHL_\ccF X$ over $X$, which is enough to conclude the proof by the universal properties of both blow-ups. Remember that 
    \[
        \blHL_\ccF X = \Bl_{\prod_\ell F_\ell(\ccF)} X
    \]
    where the product is over all non-trivial Fitting ideals of $\ccF$, and similarly $\blHL_{\ccF^{\oplus n}} X$ is the blowup of $X$ along $\prod_\ell F_\ell(\ccF^{\oplus n})$. By \cite{Moody}, it suffices to show that $\prod_\ell F_\ell(\ccF)$ divides a power of $\prod_\ell F_\ell(\ccF^{\oplus n})$ as fractional ideals. Actually, we show that every Fitting ideal $F_\ell(\ccF^{\oplus n})$ is a product of certain Fitting ideals $F_k(\ccF)$, with each $k$ appearing at least once as $\ell$ varies, and this is clearly enough. 
    
    By \cite[Lemma 10.10]{Determinantal_rings}, if $A$ is any $\bbQ$-algebra, if $M = (a_{i,j})$ is any matrix with coefficients in $A$ and if $\Delta_i$ denotes the ideal generated by all minors of $M$ of size $i\times i$, then 
    \[
        \Delta_i \Delta_j \subseteq \Delta_{i+1}\Delta_{j-1}
    \]
    whenever $i\leq j-2$. From this, we can conclude that if $\ell = ds+r$ with $r\in \{0,\ldots, s-1\}$
    \begin{equation}\label{eq:sums_products_determinantal_ideals}
        \sum_{j_1 + \ldots + j_s =\ell} \prod_i \Delta_{j_i} = \Delta_{d}^{s-r}\Delta_{d+1}^{r}.
    \end{equation}
    To conclude, remember that locally $\ccF$ is the cokernel of a morphism $\varphi\colon E^{-1}\to E^{0}$, that $F_i(\ccF)$ is the ideal $\Delta_{r_2-i}(\varphi)$ of minors in $\varphi$ of size $r_2-i$, where $r_2 = \rk (E^{0})$, and the expression for Fitting ideals of direct sums \Cref{eq: Fitting ideals direct sum}.
\end{proof}

\begin{example}
    Take $n=2$ in \Cref{prop:blowup_direct_sum}. Then \Cref{eq:sums_products_determinantal_ideals} is equivalent to
    \[
        F_\ell(\ccF\oplus \ccF) = \sum_{k+k'=\ell} F_k(\ccF) F_{k'}(\ccF) = \begin{cases}
            F_{r_2-k}^2 & \text{ if } \ell = 2k\\
            F_{r_2-k}F_{r_2-k-1} & \text{ if } \ell = 2k +1
        \end{cases}
    \]
    where $r_2$ the maximal rank of $\ccF$ as in \Cref{constr: diagonalization_blowup}.
\end{example}

\begin{proposition}\label{prop:morphism-HL-Rossi}
    Let $X$ be a Noetherian integral scheme and $\ccF$ a coherent sheaf on $X$. Then there is a natural morphism $\blHL_\ccF X \to \Bl_\ccF X$.
\end{proposition}

\begin{proof}
    Let $\pi \colon \blHL_\ccF X \to X$ be the natural projection and let $r = \rk (\ccF)$. 

    By \Cref{thm:univ_flatification}, it suffices to check that $(\pi^* \ccF)^\tf$ is locally free of rank $r$. This can be checked locally. If $X=\Spec (R)$ for a local ring $R$, the result follows from \Cref{thm: Lipman}.
\end{proof}

\subsection{The filtration of a diagonal sheaf}\label{subsec: filtration diagonal sheaf}

Given a diagonal sheaf $\ccF$,  we construct a filtration $\ccF_\bullet$ such that $\ccF_i/\ccF_{i-1}$ is locally free on a Cartier divisor $D_i$ (\Cref{lemma:filtration_rkzero}). This filtration will be used in \Cref{subsec: decomposing torsion} to describe the irreducible components of the abelian cone of a diagonal sheaf (see \Cref{theorem:cone_as_a_union}).

\begin{lemma}\label{lemma:first_step_filtration}
Let $\ccF$ be a diagonal coherent sheaf of generic rank $r$ on a Noetherian integral scheme $X$. Then there is a short exact sequence
\[
    0\to \ccK\to \ccF\to \ccF^{\tf}\to 0
\]
with $\ccF^{\tf}$ locally free of rank $r$ on $X$ and $\ccK$ a diagonal coherent sheaf of generic rank 0.
\end{lemma}

\begin{proof}
    Let $\ccK = \tor(\ccF)$. Then $\ccF^\tf$ is locally free by \Cref{rmk: diagonal implies tf lc and tor dim} and $\ccK$ is diagonal by the proof of \Cref{prop:functoriality_sheaves_injective_diagonal_version}.
\end{proof}

\begin{construction}[c.f. {\cite[Tag 0ESU]{stacks-project}}]\label{lemma:filtration_rkzero}

Let $\ccF$ be a diagonal coherent sheaf of generic rank zero and $\mrk(\ccF)=n$ on an integral scheme $X$. In the following, we construct an increasing  filtration $\ccF_\bullet$:
\[
0=\ccF_0\subset \ccF_1\subset \dots  \ccF_{n-1}\subset \ccF_n =\ccF
\]
and effective Cartier divisors $D_i$ such that for each $i$, the sheaf
\[
\ccF_i/\ccF_{i-1}
\]
is locally free of rank $i$ on the closed locally principal subscheme defined by $D_i$.
 
Our formulation differs from the one in the reference, so we present the construction of the filtration in our context. 
We can work locally and assume that $\mathcal{F}$ has a presentation which is diagonalizable in the sense of Definition \ref{def:diagonalizable} that is
$\varphi: \ccO_X^{\oplus n} \to \ccO_X^{\oplus n}$ where $\varphi$ is the diagonal matrix 
\[
\varphi=\Diag\left(\overbrace{f_1,\ldots ,f_1}^{n_1}, \overbrace{f_2,\ldots, f_2}^{n_2},\ldots,\overbrace{f_k, \ldots ,f_k}^{n_k}\right)
\]
with $n_1+\dots+n_k=n$ and non-zero $f_i$'s satisfying $(f_{i+1})\subsetneqq (f_{i})$. Note that locally $\ccF$ may not attain its maximal rank $n$, but we can always choose $f_1$ to be a unit to obtain a presentation of the correct rank.

Since $\ccF$ is diagonal, it has tor dimension at most 1 by \Cref{rmk: diagonal implies tf lc and tor dim}, therefore it admits a presentation by a \textit{square} matrix $\varphi$.

Since we are working over a domain, $(f_{i+1})\subseteq (f_{i})$ is equivalent to $f_i|f_{i+1}$. We can define effective Cartier divisors $D_1,\ldots,D_n$ by taking ratios of successive entries of $\varphi$:
\begin{align*}
&D_{n}=F_{n-1}=(f_1)\\
&D_i=\left(\frac{\varphi_{n-i+1,n-i+1}}{\varphi_{n-i,n-i}}\right).
\end{align*}
In other words, $D_i$ is the ideal generated by the ratio of the entries in position $n-i+1$ and $n-i$ in $\varphi$. Note that, while the generators of the ideals are only well-defined up to a unit, the ideals themselves are well-defined and do not depend on the chosen presentation of $\varphi$. In fact, they can be expressed as differences of the Fitting ideals of $\ccF$, which are independent of the chosen presentation. 

The divisors $D_i$ give closed locally principal subschemes of $X$, which are defined by $(f_{k+1}/f_{k})$ if $i=n-\sum_{j=1}^{k}n_k$ and are empty otherwise.

We define the increasing filtration of $\ccF_\bullet$ as follows. We set $\ccF_n:=\ccF$, and define $\ccF_{n-1}$ as the cokernel of the morphism
$\varphi':=\varphi/f_1$. That is,
\begin{equation}\label{eqn:construction_filtered}
    \begin{tikzcd}
    0\ar[r] & 0\ar[r] & (\ccO_X/(f_1))^{\oplus n} \ar[r,"\sim"]& \ccF_n/\ccF_{n-1}\ar[r] & 0\\
     0\ar[r] & \ccO_X^{\oplus n}\ar[r,"\varphi"] \arrow[u,twoheadrightarrow] & \ccO_X^{\oplus n} \ar[r]\arrow[u,twoheadrightarrow]& \ccF_n\ar[r] \arrow[u,twoheadrightarrow]& 0\\
    0 \ar[r]& \ccO_X^{\oplus n}\ar[r,"\varphi'"]\arrow[hookrightarrow]{u}{\id} & \ccO_X^{\oplus n}\arrow[hookrightarrow]{u}{f_1\cdot\id} \ar[r]& \ccF_{n-1}\arrow[hookrightarrow]{u}\ar[r] & 0.
    \end{tikzcd}
    \end{equation}

As $\ccO_{D_1}=\ccO_X/(f_1)$, the graded piece $\ccF_n/\ccF_{n-1}$ is locally free of rank $n$ on $D_n$. 

 Now, $\varphi'$ can be given by the diagonal matrix 
\[
\varphi'=\Diag\left(\overbrace{1,\dots,1}^{n_1},\overbrace{f_2/f_1,\dots,f_2/f_1}^{n_2},\dots,\overbrace{f_k/f_1,\dots,f_k/f_1}^{n_k}\right).
\]
We can pass to $\varphi'':\ccO_X^{\oplus n-1}\to\ccO_X^{\oplus n-1}$ by removing the first entry. Clearly, $\ccF_{n-1}=\Coker{\varphi ''}$. Then we can iterate the construction in \eqref{eqn:construction_filtered}, factoring our multiplication by the first entry $\varphi''_1$ of $\varphi''$
\[
    \begin{tikzcd}
    0\ar[r] & 0\ar[r] & (\ccO_X/(\varphi''_1))^{\oplus n-1} \ar[r,"\sim"]& \ccF_{n-1}/\ccF_{n-2}\ar[r] & 0\\
     0\ar[r] & \ccO_X^{\oplus n-1}\ar[r,"\varphi''"] \arrow[u,twoheadrightarrow] & \ccO_X^{\oplus n-1} \ar[r]\arrow[u,twoheadrightarrow]& \ccF_{n-1}\ar[r] \arrow[u,twoheadrightarrow]& 0\\
    0 \ar[r]& \ccO_X^{\oplus n-1}\ar[r,"\varphi'''"]\arrow[hookrightarrow]{u}{\id} & \ccO_X^{\oplus n-1}\arrow[hookrightarrow]{u}{\varphi''_1\cdot\id} \ar[r]& \ccF_{n-2}\arrow[hookrightarrow]{u}\ar[r] & 0.
    \end{tikzcd}
\]
This defines the next subsheaf $\ccF_{n-2}$ in the filtration and the new morphism $\varphi'''$. If $n_1>1$, $(\varphi''_1)=(1)$, so we will have $\ccF_{n-2}=\ccF_{n-1}$ and $D_{n-1}$ defining the empty subscheme.
Note that the sub-schemes defined by $D_{n-1},\dots,D_{n-n_1+1}$ are all empty, and the filtration is constant until $\ccF_{n-n_1-1}$, which is the cokernel of 
\[
\ccO_X^{\oplus n-n_1}\xlongrightarrow{\psi} \ccO_X^{\oplus n-n_1}
\]
with \[\psi=\Diag\left(\overbrace{1,\dots,1}^{n_2},\overbrace{f_3/f_2,\dots,f_3/f_2}^{n_3},\dots,\overbrace{f_k/f_2,\dots,f_k/f_2}^{n_k}\right)
\]
and 
$D_{n-n_1}=(f_1/f_2)$.
Iterating this construction clearly provides a filtration and a collection of effective divisors which satisfy the claims in the lemma.

The divisors $D_i$ are defined globally in terms of Fitting ideals, and do not depend on the local expression of the matrix. 
In fact, unpacking the argument above, we can check that

\begin{equation}\label{equations_of_fi}
    \begin{pmatrix}
        F_0\\\vdots\\F_{n-1}
    \end{pmatrix}
    =
    \begin{pmatrix}
    1&2&\cdots&n-1&n\\
    &1&2&\cdots&n-1\\
        &&\ddots&\ddots&\vdots\\
    &&&&\\
    &&&1&2 \\
    &&&&1
    \end{pmatrix}
    \begin{pmatrix}
    D_1\\\vdots\\D_n
    \end{pmatrix}.
\end{equation}
    Then, 
    \begin{equation}\label{equations_of_di}
    \begin{pmatrix}
        D_1\\\vdots\\D_{n}
    \end{pmatrix}
    =
    \begin{pmatrix}
    1 & -2 & 1 &  & \\

    &\ddots &\ddots& \ddots&\\
     & & 1 & -2 & 1\\
     &  & & 1&-2 \\
     &   &  & & 1
    \end{pmatrix}
    \begin{pmatrix}
    F_0\\\vdots\\F_{n-1}
    \end{pmatrix}.\qedhere
    \end{equation}

\end{construction}

With the notations of \Cref{lemma:filtration_rkzero}, the associated graded sheaf to this filtration is 
\[
    \ccE=\bigoplus_i \ccE_i,
\]
where $\ccE_i=\ccF_i/\ccF_{i-1}$.

We present an example to illustrate \Cref{lemma:filtration_rkzero}.

\begin{example}\label{example:filtration,diag,module}
    Take $R=\bbC[x,y,z]$, $X=\Spec R$. Let $\ccF=\widetilde{M}$ be the diagonal sheaf defined by
    \[
    0\to R^{\oplus 4}\xlongrightarrow{\varphi=\begin{pmatrix}
        x &0 &0& 0\\
        0 & x & 0& 0\\
        0 & 0 & xy& 0\\
        0 & 0 & 0 & xyz
    \end{pmatrix}}R^{\oplus 4}\to M\to 0.
    \]
   The divisors from the statement of \Cref{lemma:filtration_rkzero} are given by the ideals
   \begin{equation*}
    D_4= (x),
    D_3=(1),
    D_2 = (y),
    D_1=(z).
    \end{equation*}
    As a sanity check for \Cref{equations_of_fi}, we see that indeed
    \begin{align*}
        F_0 &= D_1 + 2D_2 + 3D_3 + 4D_4 = (x^4y^2z),\\
        F_1 &= D_2 + 2D_3 + 3D_4 = (x^3y),\\
        F_2 &= D_3 + 2D_4 = (x^2y),\\
        F_3 &= D_4 = (x).
    \end{align*}

    Now, all the elements of $\varphi$ are divisible by $D_4$, which is the ideal generated by the first entry.
    We set $\ccF_4=\ccF$. To obtain the next step in the filtration, $\ccF_3$, we consider the decomposition $\varphi=x\cdot\varphi'$ below
    \[
    \begin{tikzcd}
    0\ar[r] & R^{\oplus 4}\ar[r,"\varphi"] & R^{\oplus 4} \ar[r]& M\ar[r] & 0\\
    0 \ar[r]& R^{\oplus 4}\ar[r,"\varphi'"]\ar[u,"\id"] & R^{\oplus 4}\ar[u,"x\cdot\id"] \ar[r]& M_3\ar[u]\ar[r] & 0\\
    \end{tikzcd}
    \]
    we set $\ccF_3=\widetilde{M_3}$, the module defined by 
    \[
    \varphi'=\begin{pmatrix}
    1 & 0 & 0& 0\\
    0 & 1 & 0 & 0\\
    0 & 0 & y & 0\\
    0 & 0 & 0 & yz
        \end{pmatrix}
    \]
    or equivalently as the cokernel of 
    \[
    R^{\oplus 3}\xrightarrow{\varphi'=\begin{pmatrix}1 & 0 & 0 \\ 0& y & 0 \\
    0 & 0 & yz \end{pmatrix}}R^{\oplus 3}.
    \]
    Similarly, $\ccF_2=\widetilde{M_2}$ will be defined by 
        \[
    0\to R^{\oplus 2}\xrightarrow{\begin{pmatrix} y & 0 \\
    0 & yz \end{pmatrix}}R^{\oplus 2}\to M_2\to 0
    \]
    and $\ccF_1=\widetilde{M_1}$ by 
    \[
    0\to R\xrightarrow{(z)} R\to M_1\to 0.
    \]
    Finally, $M_0=0$.
    In conclusion, we obtain the filtration
    \begin{align*}
    & M=M_4=(R/(x))^{\oplus 2}\oplus R/(xy)\oplus R/(xyz)\xhookleftarrow{\begin{pmatrix}0 & 0 \\0 & 0\\ x& 0 \\0 & x\end{pmatrix}}M_3=R/(y)\oplus R/(yz)\cong\\
    &\cong M_2=R/(y)\oplus R/(yz)\xhookleftarrow{\begin{pmatrix}0 & y\end{pmatrix}}M_1=R/(z)\leftarrow 0=M_0.
    \end{align*}
    The graded pieces are 
    \begin{align*}
    &\ccE_4=\ccF_4/\ccF_3=\widetilde{\left( R/(x)\right)^{\oplus 4}}\\
    &\ccE_3=\ccF_3/\ccF_2=0\\
    &\ccE_2=\ccF_2/\ccF_1=\widetilde{\left( R/(y)\right)^{\oplus 2}}\\
    &\ccE_1=\widetilde{R/(z)}    
    \end{align*}
    and each of the sheaves $\ccE_i$ is locally free of rank $i$ on the subscheme defined by $D_i$. Note that, for $i=3$, such subscheme is empty.
\end{example}

\begin{theorem}\label{thm:cone of diagonalizable sheaf}
Let $\ccF$ be a diagonal coherent sheaf of generic rank $r$ and maximal rank $r_2$ on a Noetherian integral scheme $X$. Then we have a filtration
\[
\ccF\supset\ccK=\ccK_{r_2-r}\supset\ccK_{r_2-r-1}\supset\dots\supset \ccK_0=0
\]
such that
\[
\ccF^{\tf}=\ccF/\ccK
\]
is locally free of rank $r$ and 
\[
\ccE_i=\ccK_{i}/\ccK_{i-1}
\]
is locally free of rank $i$ on the effective Cartier divisor  $D_i$.

\end{theorem}
\begin{proof}
Immediate by \Cref{lemma:first_step_filtration} and \Cref{lemma:filtration_rkzero}.
\end{proof}

\subsection{Remarks on minimality}\label{subsec: remarks on minimality}

We saw in \Cref{rmk:first_Fitting_not_minimal} that blowing up the first non-zero Fitting ideal of $\ccF$ is, in general, not the minimal way to make $\ccF^\tf$ locally free. Similarly, blowing up all the Fitting ideals of $\ccF$ is not the minimal way to turn $(\ccF\mid_{D_i})^\tf$ into locally free sheaves for all $i$. This is illustrated in the following examples.

\begin{example}\label{example:Hu-Li_Rossi_different}
    Take $X=\Spec(R)$ and $\ccF = \widetilde{M}$ for $M=R/I$ with $I\subset R$ a non-principal ideal. The only non-trivial Fitting ideal of $\ccF$ is $F_0(\ccF) = I$. Note that $M^\tf = 0$ is locally free and $M\mid_{V(I)}$ is locally free of rank 1. This means that $\ccF$ already has the desired property on $X$. However, blowing up all the Fitting ideals of $\ccF$ results in $\Bl_I (X)$, which is isomorphic to $X$ if and only if $I$ is invertible. 
    
    This example also shows that given a coherent sheaf $\ccF$ on an integral scheme $X$, the blow up $\Bl_{\ccF}(X)$ from \Cref{definition:Rossi_blowup} and $\bld$ are different in general. Indeed, on the one hand, $\Bl_{\ccF}(X) = \Bl_{M}(X) = X$ because $M^\tf = 0$ is locally free. On the other hand, the only non-trivial Fitting ideal of $M$ is $F_1(M) = I$, so $\bld = \Bl_I X$. Therefore both blow ups agree if and only if $I$ is invertible.
\end{example}

\begin{example}\label{ex: min1}
    Let $P$ be the origin in $\bbA^2_{x,y}$ and consider the embedding $i\colon \bbA^2_{x,y}\to \bbA^3_{x,y,z} \colon (x,y)\mapsto (x,y,0)$. The image of $i$ is the plane $\Pi = \{z=0\}$. Let $I=(x,y)$ be the ideal of $P$ in $\bbA^2$ and consider the module $M = i_*I$ in $\bbA^3$. 

    Note that $M$ has generic rank 0, $M^\tf = 0$ and that $M\mid_{\Pi}$ is the ideal $I$, which is torsion-free. This means that $M^\tf = 0$ is already locally free, but $(M\mid_{\Pi})^\tf = M\mid_{\Pi} = I$ is not locally free over $\Pi$. 
    
    To compute $\blHL_{M} \bbA^3$, we start with the following resolution of $M$
    \begin{equation}\label{eq:resolution_example_P_Pi}   
        \begin{tikzcd}
            R^3\arrow{r}{\Gamma} & R^2 \arrow{r} & M\arrow{r} & 0
        \end{tikzcd}    
    \end{equation}
    where $R=\bbC[x,y,z]$ and
    \[
        \Gamma = \begin{pmatrix}
                y & z & 0\\
                -x & 0 & z
            \end{pmatrix}.
    \]

    The Fitting ideals of $M$ are
    \begin{itemize}
        \item $F_0(M) = z(x,y,z)$,
        \item $F_1(M) = (x,y,z)$,
        \item $F_n(M)=R$, for all $n\geq 3$
    \end{itemize}
    This reflects the fact that $M$ has rank 0 on $\bbA^3\setminus \Pi$, rank 1 on $\Pi\setminus P$ and rank 2 on $P$, as per \Cref{prop:Fitting_ideals_rank_connection}. 
 Then
    \[
        \blHL_{M} \bbA^3 = \Bl_{F_0(M)\cdot F_1(M)} \bbA^3 = \Bl_{z(x,y,z)^2} \bbA^3 = \Bl_{(x,y,z)} \bbA^3 = \Bl_P \bbA^3
    \]
    is simply the blowup of the origin in $\bbA^3$. Note that $\blHL_{M} \bbA^3$ is distinct from $\Bl_M \bbA^3=\bbA^3$ in this example as the latter does not flatten $(M\mid_{\Pi})^\tf $.
\end{example}

\begin{remark}
    If the sheaf $\ccF$ has projective dimension $\leq 1$, then the Rossi construction is equal to the blow-up of the first non-zero Fitting ideal. In this case, blowing up all of the proper non-zero ideals as in the Hu--Li construction gives a minimal resolution with the property that $(\ccF\mid_{D_i})^\tf$ is locally free for all of the $D_i$'s defined in terms of Fitting ideals by \eqref{equations_of_di}.
    For an ideal having projective dimension 1 is equivalent to being principal.
\end{remark}

\begin{remark}[Extension of sheaves]
Let $X$ be a scheme, let $Y$ be a closed subscheme an let $\ccF_Y$ be a coherent sheaf of rank $r$ on $Y$. In order to find a minimal blow up of this torsion sheaf, one may try to extend $\ccF_Y$ to $X$ as a sheaf which is not a torsion sheaf and perform a repeated Rossi construction. One can find an open cover of $X$ and blow-ups of the charts such that  on the blow-up the torsion free part of the pull-back of $\ccF$ is locally free on the support. However, the blown up charts may not glue to a global construction. Below we explain that it is always possible to find \emph{local} blow-ups.

Let $X$ be an affine scheme and let $Y$ be a closed subscheme. Let $\ccF_Y$ be a coherent sheaf of rank $r$ on $Y$ and assume that we have an exact sequence
\[\ccO_{U\cap Y}^{\oplus n-r}\stackrel{\widehat{M}}{\longrightarrow} \ccO_{U\cap Y}^{\oplus n}\to\ccF_{U\cap Y}\to 0,\]
where $\widehat{M}\in M_{n,r-n}(\Gamma(O_Y))$. 
We may assume that $U=\Spec R$ and $Y\cap U=\Spec R/I$, where $R$ is a ring and $I$ is an ideal. Let $\widehat{M}=(\hat{f}_{ij})$, with $\hat{f}_{ij}\in R/I$. We now choose $f_{ij}\in R$ a lift of $\hat{f}_{ij}$ and we denote by $M$ the matrix $(f_{ij})$. Then we have a morphism $\ccO_U^{\oplus n-r}\stackrel{M}{\longrightarrow} \ccO_U^{\oplus n}$ and an exact sequence 
\[\ccO_U^{\oplus n-r}\stackrel{M}{\longrightarrow} \ccO_U^{\oplus n}\to \ccF_U\to 0,\]
where $\ccF_U$ denotes the cokernel of the map induced by $M$. Then, we have that $\ccF_U|_{Y}=\ccF_Y$ and the resolution above induces a morphism
 \[
 U\dashrightarrow U\times \Gr(r,n).
 \]
 Since there is no canonical choice for the lift $M$, the above morphisms do not glue in general.
\end{remark}

\section{Desingularization and diagonalization on stacks}\label{sec:desing_stacks}

In this section, we show that the constructions introduced so far in \Cref{sec:desing_sheaf} and \Cref{sec:diagonalization} make sense for algebraic stacks, since they are both local and they commute with flat base-change (see respectively \Cref{prop:map_between_blowups} and \Cref{prop:map_between_blowups_diagonal}). Thus we define the desingularization and the diagonalization of a coherent sheaf on a Noetherian integral Artin stack and we establish properties of both constructions.

{We will require our integral stacks to admit an integral presentation. As pointed out to us by David Rydh, it is unclear whether an integral stack always admits a presentation by integral schemes or algebraic spaces. However, all of the stacks we apply this construction to in this work admit such a presentation. We will only introduce repeated blow-ups of smooth algebraic stacks, which have an integral presentation by construction.}

\subsection{Construction of \texorpdfstring{$\Bl_{\ffF}\ffP$}{Lg} and \texorpdfstring{$\blHL_{\ffF}\ffP$}{Lg}}\label{subsec:constr_bl_stacks}

So far, we have only constructed $\Bl_{\ccF}X$ and $\blHL_{\ccF}X$ for an affine (Noetherian, integral) scheme X. In this section we generalize these constructions over a stack $\ffP$ {which we assume admits a presentation by affine integral groupoid schemes, that is, it can be presented by a groupoid} $[U_1\rightrightarrows U_0]$ with $U_0, U_1$ affine integral schemes. In particular, this holds if $\ffP$ is a normal Noetherian integral algebraic stack with affine stabilizers. 

It is possible to obtain similar constructions for more general $\ffP$ (locally Noetherian, no restriction on the stabilizers), but it requires a two-step process of first generalizing the construction to algebraic spaces and then to stacks.

Denote by $\ffP$ a Noetherian, integral Artin stack with affine stabilizers which admits a presentation by integral schemes. 
Consider a smooth presentation of $\ffP$, i.e. a groupoid in affine schemes $(U_0,U_1,s,t,m)$ whose associated quotient stack $[U_1 \rightrightarrows U_0]$ is $\ffP$. Here $[U_1 \rightrightarrows U_0]$ denotes the stackyfication of a category fibered in groupoids $[U_1 \rightrightarrows U_0]^{\pre}$.
Recall that $U_0,U_1$ are affine schemes $m:U_1 \tensor[_s]{\times}{_t} U_1 \to U_1$ is the composition of 
arrows, $s,t:U_1 \to U_0$ are respectively source and target morphism and they are smooth morphisms. They satisfy some compatibility conditions that we 
will not use explicitly here (See \cite[\S (4.3)]{Laumon-stacks-2000} or \cite[\href{https://stacks.math.columbia.edu/tag/0441}{Definition 0441}]{stacks-project}.

The reader can think of $\ffP$ being the Picard stack $\Pic_{g,n}$. Recall, that a $S$-point of $\Pic_{g,n}$ is  a couple $(C,L)$ where $C$ is a nodal curve of genus $g$ with $n$ distinct smooth marked points  and $L$ is a line bundle over it. It is well known that $\Pic_{g,n}$ is a smooth Noetherian Artin stack over $\Spec(\bbC)$ of locally finite type.

Let $\ffF$ be a coherent sheaf on $\ffP$, i.e. we have a coherent sheaf $\ccF_0$ on $U_0$ and also a coherent sheaf $\ccF_1$ on $U_1$ with two fixed isomorphisms 
\begin{align}\label{fix:iso}
    s^*\ccF_0 \simeq \ccF_1 \simeq t^*\ccF_0
\end{align} that satisfy the cocycle condition on $U_1 \tensor[_s]{\times}{_t} U_1$. We refer to the article of Olsson \cite[Proposition 6.12]{Olsson-sheaves-artin-2007} for the equivalent definitions of coherent sheaves on an Artin stack.

We now proceed to use the smooth presentation of $\ffP$ to define a stack $\Bl_{\ffF}\ffP$ desingularizing the coherent sheaf $\ffF$. All of the following discussion holds formally identical when we consider the procedure that diagonalises $\ffF$ instead. The stack we obtain with the second procedure is denoted $\blHL_{\ffF}\ffP$.

Later, we prove that the blow-up stacks obtained in both cases are algebraic and come equipped with a representable (by a scheme), proper and birational morphism to $\ffP$.

With the theory developed in \S \ref{sec:desing_sheaf}, we can construct $\Bl_{\ccF_1} U_1$ and $\Bl_{\ccF_0}U_0$. Note that to apply the results in that section we require $U_0$, $U_1$ to be affine integral Noetherian schemes. However, running these arguments once with $U_0$, $U_1$ affine schemes shows that the blow-up construction glues for non-affine schemes and more generally algebraic spaces.

Since the morphisms $s,t:U_1\to U_0$ are smooth (hence flat), we apply flat base change for blowup of sheaves (see Proposition \ref{prop:map_between_blowups}) to $s,t$ and we get $\ts, \tta: \Bl_{\ffF_1}U_1 \to \Bl_{\ffF_0}U_0$. Using the fix isomorphisms \eqref{fix:iso}, we obtain the following Cartesian diagrams
\begin{equation}\label{diag:cart,blow}
\begin{tikzcd}
\Bl_{\ccF_1}U_1    \ar[r,"\ts"] \ar[d,"q"] \arrow[dr, phantom,"\ulcorner", very near start]  & \Bl_{\ccF_0} U_0 \ar[d,"p"]& \ar[d] \Bl_{\ccF_1}U_1    \ar[r,"\tta"] \ar[d,"q"] \arrow[dr, phantom,"\ulcorner", very near start]  & \Bl_{\ccF_0} U_0 \ar[d,"p"]\\
U_1\ar[r,"s"] & U_0 & U_1\ar[r,"t"] & U_0.
\end{tikzcd}
\end{equation}

In addition, using Cartesian diagrams on a cube, we construct a map
\[
\tm:\Bl_{\ccF_1}U_1 \tensor[_\ts]{\times}{_\tta} \Bl_{\ccF_1}U_1 \to \Bl_{\ccF_1}U_1.
\]
More precisely, we have
\begin{align*}
\Bl_{\ccF_1}U_1\tensor[_\ts]{\times}{_\tta} \Bl_{\ccF_1}U_1 &\simeq \left(\Bl_{\ccF_0}U_0 \tensor[_p]{\times}{_s}U_1 \right) \tensor[_\ts]{\times}{_\tta} \left(\Bl_{\ccF_0}U_0\tensor[_p]{\times}{_t}U_1 \right)\\
& \simeq \Bl_{\ccF_0}U_0\tensor[_p]{\times}{_s}U_1 \tensor[_s]{\times}{_t}U_1 \\
&\to \Bl_{\ccF_0}U_0\tensor[_p]{\times}{_s}U_1  \mbox{ by applying } m:U_1\tensor[_s]{\times}{_t} U_1 \to U_1\\
& \simeq \Bl_{\ccF_1}U_1 \mbox{ by the Cartesian diagram \eqref{diag:cart,blow}}.
\end{align*}
We obtain a smooth groupoid in schemes
\[
(\Bl_{\ccF_0}U_0, \Bl_{\ccF_1}U_1, \ts,\tta,\tm)
\]
with a morphism of groupoids to $(U_0,U_1,s,t,m)$.
This defines a 1-morphism 
\[
p:[\Bl_{\ccF_0}U_0\rightrightarrows\Bl_{\ccF_1}U_1]^{\pre}\to[U_0\rightrightarrows U_1]^{\pre}.
\]
Let $\Bl_{\ffF}\ffP$ denote the stackyfication of $[\Bl_{\ccF_0}U_0\rightrightarrows\Bl_{\ccF_1}U_1]^{\pre}$. By universal property, the morphism discussed above lifts to a morphisms of stacks
\[
\pi:\Bl_{\ffF}\ffP\to \ffP.
\]

\begin{remark}
    We thank David Rydh for also pointing out that the same construction of desingularization of sheaves on Artin stacks can be realized by adapting the Grassmannian interpretation of Rossi. One can view $\Bl_{\ffF}\ffP$ as the closure of a morphism from $\ffU\subset\ffP$ to the relative Grassmannian of rank $r$ quotients of $\ffF$, $\Gr_{\ffP}(\ffF,r)$. A more general version of the relative Grassmannian of \Cref{gras} was constructed by Hall and Rydh, who introduce a more general Quot scheme over algebraic stacks with mild assumptions in \cite{hall2015general}. However, we prefer to include this explicit gluing argument as it is not immediately clear to us that their assumptions hold for the stacks we need in our applications.
\end{remark}

\subsection{Properties of \texorpdfstring{$\Bl_{\ffF}\ffP$}{Lg} and \texorpdfstring{$\blHL_{\ffF}\ffP$}{Lg}}

We collect some properties of $\Bl_{\ffF}\ffP$ and $\blHL_{\ffF}\ffP$, including those of \Cref{subsec: properties univ desing,subsec:properties_HL}, which naturally extend to stacks.

\begin{theorem}\label{thm:construction-blowup-stacks}
Let $[U_1\rightrightarrows U_0]\to\ffP$ be an integral Noetherian Artin stack with affine stabilizers and an integral presentation, and let $\ffF$ be a coherent sheaf on it. 
\begin{enumerate}
\item The stacks $\Bl_{\ffF}\ffP=[\Bl_{\ccF_1} U_1 \rightrightarrows \Bl_{\ccF_0}U_0]$ and  $\blHL_{\ffF}\ffP=[\blHL_{\ccF_1} U_1 \rightrightarrows \blHL_{\ccF_0}U_0]$ are integral Noetherian Artin stacks. 
\item The morphisms $\Bl_{\ffF}\ffP\to\ffP$ and $\blHL_{\ffF}\ffP\to\ffP$ are representable proper and birational.
\end{enumerate}

\end{theorem}

\begin{proof}
Once again, we will only discuss $\Bl_{\ffF}\ffP$, as the argument for $\blHL_{\ffF}\ffP$ is identical, mutatis mutandis. 

Part (2), together with the properties of the respective constructions on schemes, implies part (1) of the theorem. To establish part (2), it suffices to compute the fiber $U_0\times_{\ffP}\Bl_{\ffF}\ffP$ and show that it is $\Bl_{\ccF_0}U_0$. Since $\Bl_{\ccF_0}U_0\to U_0$ is representable by a scheme, so will be the morphism $\Bl_{\ffF}\ffP\to\ffP$.

Now consider the 2-Cartesian diagram of categories fibered in groupoids
\begin{equation}\label{eq:diagram_groupoids}
\begin{tikzcd}
\ffX \ar[r]\ar[d]\arrow[dr, phantom,"\ulcorner", near start] & U_0 \ar[d]\\
    {[\Bl_{\ccF_1}U_1\rightrightarrows \Bl_{\ccF_0}U_0]^{\pre}}\ar[r, "p"]& {[U_1\rightrightarrows U_0]^{\pre}}
\end{tikzcd}
\end{equation}
where by abuse of notation, $U_0$ is the category fibered in sets associated to this algebraic space. One computes (see the discussion around \cite[\href{https://stacks.math.columbia.edu/tag/045G}{Tag 04Y4}]{stacks-project}) that the groupoid $\ffX$ is given by $(U_0', U'_1, s', t', m')$ where 
\begin{align*}
    U'_0 & = U_1\times_{t,U_0,p}\Bl_{\ccF_0}U_0\\
    U'_1 & = U_1\times_{t,U_0,p\cdot s}\Bl_{\ccF_1}U_1\\
    s'&:(x,y)\mapsto (x,\ts(y))\\
    t'&:(x,y)\mapsto (m(x,p(y)),\tta(y))\\(x_2,\tm(y_1,y_2)).
\end{align*}
By \eqref{diag:cart,blow}, 
\begin{align*}
U'_1&=U_1\times_{t,U_0,p\cdot s}(U_1\times_{s,U_0,p}\Bl_{\ccF_0}U_0)\\
&=(U_1\times_{t,U_0,s}U_1)\times_{t\cdot pr_2,U_0,p}\Bl_{\ccF_0}U_0\\
s'&:((x,y),z)\mapsto (x,z)\\
t'&:((x,y),z)\mapsto (y,z)
\end{align*}
From this expression, $\ffX$ is a banal groupoid whose stackyfication is equivalent to the scheme $\Bl_{\ccF_0}U_0$, as the relations $s', t'$ identify all the points of the $U_1$ factor. Then the stackyfication of \eqref{eq:diagram_groupoids} gives us a 2-Cartesian diagram:
\begin{equation}\label{eq:diagram_blowup_cover}
\begin{tikzcd}
\Bl_{\ccF_0}U_0 \ar[r,"\widetilde{\pi}"]\ar[d]\ar[d] \arrow[dr, phantom,"\ulcorner", very near start]  & U_0\ar[d]\\
\Bl_{\ffF}\ffP\ar[r,"\pi"]&\ffP    
\end{tikzcd}
\end{equation}
This discussion proves that $\pi$ is representable. Recall that a morphism of stacks is birational if there exists an isomorphism on dense open substacks on source and target (see \cite{Cavalieri-stefens-Wise-2012-polyno-tauto}). By \Cref{definition:Rossi_blowup} we deduce that $\pi$ is proper and birational.
\end{proof}

Now we prove that our construction satisfies a universal property. In particular, it will then be independent of the choice of a groupoid presentation.

\begin{theorem}[Universal property of the Rossi desingularization]\label{Rossi univ}
Let $\pi:\Bl_{\ffF}\ffP\to\ffP$ be as in \Cref{thm:construction-blowup-stacks}. Then
\begin{enumerate}
    \item The sheaf $(\pi^*\ffF)^\tf$ is locally free of the same generic rank as $\ffF$.
    \item The morphism $\pi:\Bl_{\ffF}\ffP\to\ffP$ satisfies the following universal property:
    For any morphism of stacks $p:\ffY\to\ffP$ such that $(p^*\ffF)^\tf$ is torsion-free of the same generic rank as $\ffF$, there is a unique\footnote{To be precise, there exists a morphism $p'$, unique up to a unique $2$-morphism.} morphism $p'$, which makes the following diagram 2-commutative
    \[
    \begin{tikzcd}
        \ffY\ar[r,"\exists ! p'", dashed]\ar[dr,"p" '] & \Bl_{\ffF}\ffP\ar[d,"\pi"]\\
         & \ffP
    \end{tikzcd}
    \]
    \end{enumerate}
\end{theorem}

\begin{proof}
Choose a smooth presentation of $\ffP$ by $(U_0,U_1,s,t,m)$ with $U_0$ an affine Noetherian, integral scheme and construct the blow-up via this presentation as in the previous section.
For the first statement we use the fact that $\pi$ is representable and that by the proof of \Cref{thm:construction-blowup-stacks}, $q:\Bl_{\ccF_0}U_0\to \Bl_{\ffF}\ffP$ is a smooth covering by a scheme. So it suffices to prove that $q^*((\pi^*\ffF)^{\tf})$ is locally free of the correct rank. But this is just $(\overline{\pi}^*\ccF_0)^{\tf}$ with the notation of \eqref{eq:diagram_blowup_cover}, so the result follows by \Cref{thm:univ_flatification}.

For the second statement, by fppf descent it suffices to prove that for any flat morphism from an affine Noetherian integral scheme $T\to \ffY$ we can construct a morphism
\[  \begin{tikzcd}
        & & \Bl_\ffF \ffP\arrow{d}{\pi}\\
        T\arrow{r}\arrow[dashed]{rru}{a}\arrow[bend right=30]{rr}{g} & \ffY \arrow{r} & \ffP.
    \end{tikzcd}
\]
By using the groupoid scheme presentations of $\ffP$, $\Bl_{\ffF}\ffP$ we have fixed above, we pull back to the smooth covering $U_0\to\ffP$,
\begin{equation}\label{diagram_base_change_to_U_0}  \begin{tikzcd}
        &  \Bl_{\ccF_0} U_0\arrow{d}{\overline{\pi}}\\
        T_0=T\times_{\ffP}U_0\arrow{r}\arrow[dashed]{ru}{a_0}\arrow{r}{g_0}  & U_0.
    \end{tikzcd}.
\end{equation}
Since $T\to \ffY$ is flat, we have that $(g^*\ffF)^\tf$ is locally free of rank $r$, for $T_0$ we have
\[
    \begin{tikzcd}
        T_0 \arrow{r}{g_0}\arrow{d}{b} & U_0 \arrow{d}{q}\\
        T\arrow{r}{g} & \ffP.       
    \end{tikzcd}
\]
Then
\[
    g_0^*(\ccF_0^{\tf})=(g_0^* q^* \ffF)^\tf = (b^* g^* \ffF)^\tf  = b^*((g^* \ffF)^\tf)
\]
because $q$ and $g$ are flat. So $g_0^*(\ccF_0^{\tf})$ is locally free of rank $r$.  By \Cref{thm:univ_flatification}, there is a unique morphism $a_0\colon T_0 \to \Bl_{\ccF_0} U_0$ over $g_0$: 

\[
    \begin{tikzcd}
        & \Bl_{ \ccF_0} U_0 \arrow{dr}\arrow{d}\\
        T_0 \arrow{r}{g_0}\arrow{d}{b}\arrow{ur}{a_0} & U_0 \arrow{dr}{f} & \Bl_\ffF \ffP \arrow{d}\\
        T\arrow{rr}{g} & & \ffP.    
    \end{tikzcd}
\]
Now, to show that this map descends to $a:T\to \Bl_{\ffF}\ffP$ we need to give a map of groupoids $[T_1\rightrightarrows T_0]$ to $[\Bl_{\ccF_1}U_1\rightrightarrows \Bl_{\ccF_0}U_0]$, where these are the groupoid presentations induced by $[U_1\rightrightarrows U_0]$. That is, we need to specify in addition to $a_0$ constructed above, a map
\[
a_1:T\times_{\ffP}U_1=T_1\to \Bl_{\ffF}\ffP\times_{\ffP}U_1=\Bl_{\ccF_1}U_1
\]
over $g_1:T_1\to U_1$ and show that $a_0$ and $a_1$ are compatible with the source and target maps.
We can choose $U_1$ to be an affine scheme, then the same argument we applied above to lift $g_0$ to $a_0$ produces a unique lift of $g_1$ to $a_1$. The compatibility of $(a_0,a_1)$ with the source and target maps of  $[T_1\rightrightarrows T_0]$ and $[\Bl_{\ccF_1}U_1\rightrightarrows \Bl_{\ccF_0}U_0]$ then follows by uniqueness. 

\qedhere
\end{proof}

Having established the universal property, we can now talk about the blow-up $\Bl_{\ffF}\ffP$ without specifying a presentation for $\ffP$, as all choices will produce canonically isomorphic blow-ups.

\begin{proposition}\label{prop:properties-blowup-exact sequence-linebundle-stacks}
Let $\ffF$ be a coherent sheaf on $\ffP$, a Noetherian integral algebraic stack with affine stabilizers and admitting an integral presentation.
\begin{enumerate}
\item For any line bundle $\ffL$ on $\ffP$, we have $\Bl_{\ffF\otimes \ffL} \ffP=\Bl_\ffF\ffP$ and also $\blHL_{\ffF\otimes \ffL} \ffP=\blHL_\ffF\ffP$.
\item  Let $\ffE$, $\ffF, \ffG$ be coherent $\ccO_\ffP$-modules. Assume that we have an exact sequence $0\to \ffE\to \ffF\to \ffG \to 0$. 
    \begin{enumerate}
        \item If the sequence is locally split and $\ffE$ is locally free, then there are isomorphisms $\Bl_{\ffF}X \simeq \Bl_{\ffG} X$ and $\blHL_{\ffF}X \simeq \blHL_{\ffG} X$.
        \item If $\ffG$ is locally free, then there are isomorphisms $\Bl_{\ffF}X \simeq \Bl_{\ffE} X$ and $\blHL_{\ffF}X \simeq \blHL_{\ffE} X$.
    \end{enumerate}
\end{enumerate}
\end{proposition}

\begin{proof}
For the first part of 1, let $p'\colon\Bl_{\ffF\otimes \ffL} \ffP \to \ffP$ and $p\colon\Bl_\ffF\ffP\to \ffP$ be the natural projections. By the Universal Property, \Cref{Rossi univ}, to show that $\Bl_{\ffF\otimes \ffL} \ffP=\Bl_\ffF\ffP$ it suffices to show that $((p')^*\ffF)^\tf$ and $(p^*(\ffF\otimes \ffL))^\tf$ are locally free. These statements can be checked locally and they follow from \Cref{prop: isomorphic-blowups-tensor-line-bundle}. 

Similarly, all the other statements are local, so they follow from the same statements on schemes with the two different blow-ups.
For the Rossi blow-up $\Bl_{\ffF}\ffP$, the schematic statements are Propositions \ref{prop: isomorphic-blowups-tensor-line-bundle} and \ref{prop:functoriality,sheaves,injective,exact,sequence,Villamaoyr} and for the Hu--Li blow-up $\blHL_{\ffF}\ffP$, it follows from Propositions \ref{prop:tensor,line,HL,blowup} and \ref{prop:functoriality_sheaves_injective_diagonal_version}.
\end{proof}

From the beginning of \Cref{sec:diagonalization} (Definitions  \ref{def:diagonalizable} \ref{defi:diag_morphism_sheaves} and Proposition \ref{prop:diagonalizable_definitions_equivalent}),  we can define the notion of diagonal sheaves or locally diagonalizable morphism of sheaves on Artin stacks as follows.
\begin{definition}\label{defi:diag:coh,sheaf,stack}

\begin{enumerate}
    \item A coherent sheaf $\ffF$ on $\ffP$ is \textit{diagonal} if for any scheme $S$ and morphism $f:S\to \ffP$, the sheaf $f^*\ffF$ is diagonal, that is, its Fitting ideals $F_i(f^* \ffF)$ are locally principal.
\item A \textit{diagonalization} of a coherent sheaf $\ffF$ is a morphism $\pi:\widetilde{\ffP}\to \ffP$ such that $\pi^*\ffF$ is diagonal.   
\end{enumerate}
\end{definition}

\begin{remark}
 Using the presentation of $\ffP$, we could also define that $\ffF$ is diagonal if $\ccF_0$ is.
\end{remark}

\begin{theorem}[Universal property of the diagonalization]\label{theorem:universal_property_diag_stacks}
Let $\ffP$ a Noetherian, integral, normal Artin stack admitting an integral presentation. Let $\pi:\blHL_{\ffF}\ffP\to\ffP$ be as above. Then
\begin{enumerate}
\item The sheaf $\pi^*\ffF$ is diagonal of the same generic rank as $\ffF$.
\item The blow-up $\blHL_{\ffF}\ffP$ satisfies the universal property:
For any morphism of stacks $f:\ffY\to \ffP$ such that $f^*\ffF$ is diagonal of the same generic rank as $\ffF$, there is a unique morphism $f'$, which makes the following diagram 2-commutative:
    \[
    \begin{tikzcd}
        \ffY\ar[r,"\exists ! f'", dashed]\ar[dr,"f" '] & \blHL_{\ffF}\ffP\ar[d,"\pi"]\\
         & \ffP
    \end{tikzcd}
    \]
    \end{enumerate}
\end{theorem}

\begin{proof}
The statement follows from the universal property of the Hu-Li blow-up for schemes, \Cref{thm:univ_flatificationerty_HL_principal_Fitting_ideals}, and the compatibility of the Hu-Li blow-up with flat pullback, \Cref{prop:map_between_blowups_diagonal}. The argument is the same as the proof of \Cref{Rossi univ}.
\end{proof}

\begin{remark}
    If $\ffP$ has the resolution property in the sense of \cite{Totaro_resolution}, then we have that $\pi:\Bl_{\ffF}\ffP\to\ffP$ is projective. Indeed, if $\ffP$ has the resolution property, then we have a global locally free sheaf $\ffE$ with a surjective morphism 
    \[\ffE\to\ffF\to 0.
    \]
    This allows us to define $\Bl_{\ffF}\ffP$ via the graph construction and thus the resulting stack is projective over $\ffP$. Note that projectivity is not local on the target, and thus, even though the local construction is projective, $\pi$ may not be projective. By \cite{Totaro_resolution} stacks which are not global quotient stacks do not have the resolution property. Many of the stacks that we work with are not global quotients. For more details on stacks which are not a global quotients see \cite{kresch2013flattening}.
\end{remark}

\section{Components of abelian cones}\label{sec:components_abelian_cones}

Let $\ccF$ be a diagonal sheaf on an integral Noetherian scheme, we study the irreducible components of $C(\ccF)$. We show that $C(\ccF)$ has finitely many irreducible components, which we consider with their natural reduced structure. Each irreducible component is a vector bundle supported on a closed integral subscheme. All of the cones in this section are taken over $X$, unless otherwise specified by the notation $C_{\mathrm{base}}(\mathrm{sheaf})$.

\subsection{The main component of an abelian cone}\label{subsec: main component abelian cone}

We start our study of components of cones with the main component of an abelian cone. Our study is motivated by \cite[Proposition 2.5]{Axelsson-Magnusson}, which we recall below as \Cref{prop: AM}. It states that if $\pi\colon C=\Spec(\ccA) \to X$ is a cone with $X$ integral and with $\ccA$ torsion-free outside of a closed $Z\subseteq X$, then the closure of $C\setminus \pi^{-1}(Z)$ inside $C$ is equal to $\Spec(\ccA^\tf)$.

In general, $\Spec(\ccA^\tf)$ need not be irreducible, see \Cref{ex: tf cone reducible}. However, if the cone is abelian, that is, if $\ccA = \Symm \ccF$ for a coherent sheaf $\ccF$, then $\Spec(\ccA^\tf)$ is an irreducible component that we call the main component of $C(\ccF) = \Spec \Symm \ccF$. Note that $(\Symm \ccF)^\tf$ and $\Symm (\ccF^\tf)$ need not agree in general (see \Cref{rmk: main component not equal to C(Ftf)}), but they do if $\ccF^\tf$ is locally free by \Cref{lem: symm commutes with tf if locally free}. In particular, they agree for diagonal sheaves by \Cref{rmk: diagonal implies tf lc and tor dim}.

Let $X$ be an integral Noetherian scheme,  let $\ccA$ be an $\ccO_X$-algebra with the assumptions of \Cref{def: cone} and let 
\[
    \pi\colon C=\Spec_X(\ccA)\to X
\]
be the cone associated to $\ccA$. The natural surjection $\ccA\to \ccA^\tf$ induces a closed embedding 
\[
    \Spec(\ccA^\tf) \hookrightarrow \Spec(\ccA),
\]
which we want to understand geometrically. 

\begin{notation}\label{not: closure}
    Let $X$ be a scheme and let $U\subseteq X$ be an open subscheme. We denote by $\cl_X U$ the closure of $U$ in $X$, with its reduced induced structure, and by $\scl_X U$ the schematic closure of $U$ in $X$. If $U$ is reduced, then $\scl_X U = \cl_X U$ by \cite[Lemma 056B]{stacks-project}
\end{notation}

The following result is proven in \cite{Axelsson-Magnusson} in the analytic category, but the proof holds for schemes as well.

\begin{proposition}[See Proposition 2.5 \cite{Axelsson-Magnusson}]\label{prop: AM}
    Let $U\subseteq X$ be a non-empty open such that $\ccA\mid_U$ is torsion free and let $ \pi\colon C=\Spec_X(\ccA)\to X$. Then 
    \[
        \Spec_X(\ccA^\tf) = \cl_{\Spec_X(\ccA) }(\pi^{-1}(U)).
    \]
    Furthermore, if $\pi^{-1}(U)$ is reduced, then 
    \[
        \Spec_X(\ccA^\tf) = \scl_{\Spec_X(\ccA) }(\pi^{-1}(U)).
    \]
\end{proposition}

In general, $\Spec_X(\ccA^\tf)$ may not be irreducible, see \Cref{ex: tf cone reducible}.

\begin{example}\label{ex: tf cone reducible}
    The cone $\Spec_X(\ccA^\tf)$ may not be irreducible. For example, let $R= \bbC[x]$ and let $A= R[Y,Z]/(YZ)$ viewed as a graded $R$-algebra with $Y,Z$ in degree $1$. This is a cone over $\bbA^1 = \Spec(R)$. It is clear that $A$ has no torsion as an $R$-module but $\Spec(A)$ has two irreducible components.
\end{example}

Now we focus on abelian cones. Firstly, we show in \Cref{prop: main component} that if $\ccA= \Symm \ccF$, then $\Spec_X(\ccA^\tf)$ is an irreducible component of $\Spec_X(\ccA)$.

\begin{proposition}\label{prop: main component}
    Let $X$ be an integral Noetherian scheme and let $\ccF$ be a coherent sheaf on $X$. Then $\Spec(\Symm \ccF)^\tf$ is an irreducible component of  $C(\ccF) = \Spec \Symm \ccF$.
\end{proposition}

\begin{proof}
    Let $U\subseteq X$ be a non-empty open such that $\ccF$ is locally free on $U$. Then, for $\pi\colon C(\ccF) \to X$ the projection, we have that $\pi^{-1}(U)$ is a vector bundle over $U$, thus it is integral. Let $Z$ be the unique irreducible component of $C(\ccF)$ containing $\pi^{-1}(U)$. Then
    \[
        \scl_{C(\ccF)}(\pi^{-1}(U))  = \scl_{Z}(\pi^{-1}(U))  = Z,
    \]
    where the first equality  follows from \Cref{lem: closure operation} and the second one is a basic property of the Zariski topology that the closure of an irreducible open in an irreducible space is the whole space. 
\end{proof}

\begin{definition}\label{def: main component abelian cone}
    Let $X$ be an integral Noetherian scheme and let $\ccF$ be a coherent sheaf on $X$. We say that $\Spec(\Symm \ccF)^\tf$ is the \textit{main component} of the abelian cone $C(\ccF) = \Spec \Symm \ccF$. 
\end{definition}

\begin{remark}\label{rmk: main component not equal to C(Ftf)}
    With the assumptions of \Cref{def: main component abelian cone}, it is not true in general that the main component of $C(\ccF)=\Spec\Symm \ccF$ is equal to $C(\ccF^\tf) = \Spec\Symm (\ccF^\tf)$. In fact, $C(\ccF^\tf)$ need not be irreducible (see \Cref{rmk: C(tf) not irreducible in general}).
    The underlying reason for this discrepancy is that $\Symm$ and torsion-free part do not commute in general (see \Cref{rmk: symm and tf do not commute in general}). A particular case where $C(\ccF^\tf)$ is clearly irreducible is if $\ccF^\tf$ is locally free. In that case, 
    \[
        (\Symm \ccF)^{\tf} = \Symm(\ccF^{\tf}).
    \]
    by \Cref{lem: symm commutes with tf if locally free} and so
    \begin{equation}\label{eq: main component equals C(Ftf)}
        \Spec(\Symm \ccF)^\tf = \Spec\Symm (\ccF^\tf).
    \end{equation}
    In particular, the equality \eqref{eq: main component equals C(Ftf)} is true for a diagonal sheaf $\ccF$ by \Cref{rmk: diagonal implies tf lc and tor dim}.
\end{remark}

\begin{example}\label{rmk: C(tf) not irreducible in general}
  
    This is an example of a torsion free sheaf $\ccG$ on an integral Noetherian scheme $X$ such that $\Spec \Symm \ccG$ is not irreducible. Let $X=\Spec(\bbC[x,y])$ be the affine plane, let $I=(x,y)$ be the ideal of the origin $0$ and let $M=I\oplus I$. Then $\Symm(M) \simeq \bbC[x,y,A_1,A_2,B_1,B_2]/(yA_1-xA_2,yB_1-xB_2)$, so $C(M)=\Spec\Symm(M)$ has two irreducible components: one of them is $V(A_2B_1-A_1B_2,yB_1-xB_2,yA_1-xA_2)$, which is the closure of the restriction of $C(M)$ to $X\setminus 0$; and the other one is $V(x,y)$, the fibre of $C(M)$ at $0$. Both components have dimension 4.
\end{example}

\subsection{Abelian cones as a pushout of their main component}\label{subsec: abelian cones as pushout}

We restrict now our study of components of cones to the special case of an abelian cone $C(\ccF)$ with $\ccF^\tf$ locally free. We first show that $\Symm$ and torsion-free part commute in that case (\Cref{lem: symm commutes with tf if locally free}), therefore the main component is $C(\ccF^\tf)$, which is also abelian. We show that $C(\ccF)$ admits a description as a pushout with $C(\ccF^\tf)$ as one of the factors (\Cref{prop: decomposition cone}).
\\

\begin{lemma}\label{lem: symm commutes with tf if locally free}
    Let $X$ be a Noetherian scheme and let $\ccF$ be a coherent sheaf on $X$. If $\ccF^\tf$ is locally free then
    \[
        (\Symm \ccF)^{\tf} = \Symm(\ccF^{\tf}).
    \]
\end{lemma}

\begin{proof}
    We have the following commutative diagram.
    \[
    \begin{tikzcd}
         & 0\arrow{d} & 0 \arrow{d} & 0 \arrow{d} \\
        0\arrow{r} & \tor(\Symm \ccF)\arrow{d}{i'}\arrow{r}{e} & \ker(p)\arrow{d}{i}\arrow{r}{e'} & \ker(p'')\arrow{d}{i''}\arrow{r} & 0\\
        0\arrow{r} & \tor(\Symm \ccF) \arrow{d}{p'}\arrow{r}{f} & \Symm \ccF\arrow{d}{p}\arrow{r}{f'} & (\Symm \ccF)^\tf \arrow{r}\arrow{d}{p''} & 0\\
        0\arrow{r}& \tor(\Symm (\ccF^\tf)) \arrow{d}\arrow{r}{g} & \Symm(\ccF^\tf) \arrow{d}\arrow{r}{g'} & (\Symm (\ccF^\tf))^\tf \arrow{d}\arrow{r} & 0\\
        & 0 & 0 & 0
    \end{tikzcd}
    \]
    The last two rows are clearly exact. Moreover, since $\ccF^\tf$ is locally free, we have that $\tor(\Symm (\ccF^\tf)) = 0$ and $g'$ is an isomorphism. The morphism $i'$ is the identity on $\tor(\Symm\ccF)$. The surjective morphism $p$ comes from applying $\Symm$ to the surjection $\ccF\to \ccF^\tf$, because $\Symm$ preserves surjections. The morphism $p''$ is induced by $p$ using that $\tor(\Symm (\ccF^\tf))=0$. The first row is exact by the Snake Lemma. We want to show that $\ker(p'')=0$ or, equivalently, that $e$ is an isomorphism.

    It follows from the above that we have
    \[
        \begin{tikzcd}
            0\arrow{r} & \tor(\Symm \ccF) \arrow{r}{f\circ i'} & \Symm \ccF \arrow{r}{p} & \Symm(\ccF^\tf)\arrow{r} & 0,
        \end{tikzcd}
    \]
    which is exact except possibly at $\Symm \ccF$. We conclude if we show exactness there. The inclusion $\mathrm{Im}(f\circ i')\subseteq \ker(p)$ is clear because $p\circ f = g\circ p' = 0$. 
    
    To show that $\ker(p)\subseteq \mathrm{Im}(f\circ i')$, we know that
    \[
        \tor(\ccF)\otimes \Symm^{n-1}(\ccF) \to \Symm^n(\ccF)\to \Symm^n(\ccF^\tf)\to 0
    \]
    is exact for all $n\geq 1$ by \cite[Lemma 01CJ]{stacks-project}. Note that $p$ is a morphism of graded algebras, therefore 
    \[
        \ker(p) = \bigoplus_n \ker(\Symm^n(\ccF)\to \Symm^n(\ccF^\tf)).
    \]
    It suffices to show that for each $n$, the morphism $\tor(\ccF)\otimes \Symm^{n-1}(\ccF) \to \Symm^n(\ccF)$ factors through $\tor(\Symm(\ccF))$.
    Locally, $X=\Spec(R)$ and $\ccF=\widetilde{M}$ for some $R$-module $M$. Given $\lambda = \sum_j m_1^j \otimes \ldots \otimes m_n^j\in \tor(M)\otimes \Symm^{n-1}(M)$, we can choose for each $j$ a non-zero divisor $r_j\in R$ such that $r_j m_1^j = 0$. Then $r=r_1\cdots r_j$ is a non-zero divisor and $r\lambda = 0$, so $\lambda\in \tor(\Symm(M))$.
\end{proof}

\begin{remark}\label{rmk: symm and tf do not commute in general}
    Note that \Cref{lem: symm commutes with tf if locally free} does not holds in general if we do not assume that $\ccF^\tf$ is locally free. For example, let $\ccF=\ccI$ be the ideal sheaf of a closed point $P$ on $X$. Then $(\Symm \ccI)^{\tf} = \Symm \ccI = \bigoplus_{n\geq 0} \ccI^n$ if and only if $P$ is regular. Another example is $R=\bbC[x,y]$ and $M=I\oplus I$ for $I=(x,y)$. Indeed, $M$ is torsion-free but $\Symm M$ has torsion because $x(x\otimes y - y\otimes x) = x\otimes (xy) - (xy) \otimes x = y (x\otimes x - x\otimes x) = 0$.
\end{remark}

\begin{lemma}\label{lem: support tor(M) kills tor(M)}
    Let $R$ be a ring, $I$ be an ideal in $R$ and $M$ be an $R$-module. 
    If $M^\tf$ is locally free and $I\cdot \tor(M) = 0$ then $I\cdot \tor(\Symm M) = 0$.
\end{lemma}

\begin{proof} Note that $\tor(\Symm M) = \bigoplus_{n\geq 0} \tor(\Symm^n M)$. In the proof of \Cref{lem: symm commutes with tf if locally free} we show that $\Symm^n(M^\tf) \simeq (\Symm^n M)^\tf$. The following commutative diagram is exact by \cite[Lemma 01CJ]{stacks-project}.
    \[
        \begin{tikzcd}
        & \tor(M)\otimes \Symm^{n-1}(M) \arrow{r}\arrow{d} &  \Symm^n M \arrow{r} \arrow{d} & \Symm^n(M^\tf) \arrow{r}\arrow{d} & 0\\
        0 \arrow{r} & \tor(\Symm^n M) \arrow{r} & \Symm^n M \arrow{r} & \Symm^n(M)^\tf \arrow{r} & 0
        \end{tikzcd}
    \]
    The first row is exact by \cite[Lemma 01CJ]{stacks-project}, and the second row is also exact. By the Snake Lemma, $\tor(M)\otimes \Symm^{n-1}(M)$ surjects onto $\tor(\Symm^n M)$ and the claim follows.
\end{proof}

\begin{lemma}\label{lem: decomposition cone algebraic version}
    Let $R$ be a commutative ring, $A$ be an $R$-algebra and $I$ be an ideal of $R$. If $I\cdot \tor(A) = 0$ and $A^\tf$ is locally free, then the following square is Cartesian in the category of $R$-algebras
    \[
		\begin{tikzcd}
			A \arrow{r}\arrow{d} & A^\tf\arrow{d}\\
			A\otimes R/I \arrow{r} & A^\tf\otimes R/I.
		\end{tikzcd}
	\]
\end{lemma}

\begin{proof}
 
    We have the following commutative diagram
    \[
    \begin{tikzcd}
        & 0\arrow{d}& 0\arrow{d} & 0\arrow{d}\\
        & 0\arrow{d}\arrow{r}& IA\arrow{d}{i}\arrow{r}{e'} & IA^\tf\arrow{d}{i''}\\
        0\arrow{r} & \tor(A) \arrow{d}{p'}\arrow{r}{f} & A\arrow{d}{p}\arrow{r}{f'} & A^\tf \arrow{r}\arrow{d}{p''} & 0\\
        & \tor(A)\otimes R/I \arrow{d}\arrow{r}{g} & A\otimes R/I \arrow{d}\arrow{r}{g'} & A^\tf \otimes R/I \arrow{d}\arrow{r} & 0\\
        & 0 & 0 & 0
    \end{tikzcd}
    \]
    The three columns are exact because $N\otimes R/I\simeq N/IN$ for any $R$-module $N$ and because $I\cdot \tor(A) = 0$.
   
    Observe that $g$ is injective. This is equivalent to $\mathrm{Tor}_1(R/I,A^\tf) = 0$, which holds because $A^\tf$ is locally free.
    By the Snake Lemma, the natural morphism $e'\colon IA \to IA^\tf$ induced by $f'$ is an isomorphism.

    In order to prove the lemma, one can show that the square in question is a Cartesian square of R-modules and then check that it is also a Cartesian diagram of $R$-algebras. Both can be achieved by routine diagram chasing using the fact that $e'$ is an isomorphism. 
\end{proof}

\begin{proposition}\label{prop: decomposition cone}
    Let $X$ be a Noetherian scheme, $\ccF$ a coherent sheaf on $X$  and let $\pi \colon C(\ccF)=\Spec(\Symm \ccF)\to X$ be the corresponding abelian cone. Let $i\colon Z\hookrightarrow X$ be a closed subscheme in $X$ with ideal sheaf $\ccI_Z$ such that $\ccI_Z \subseteq \Ann(\tor(\ccF))$. If $\ccF^\tf$ is locally free, then the following is a push-out of schemes
    \[
        \Spec\Symm \ccF = \Spec(\Symm \ccF^\tf) \bigsqcup_{\Spec i_\ast(\Symm(\ccF^\tf)\mid_Z)} \Spec i_\ast(\Symm(\ccF)\mid_Z)
    \]
    If, moreover, $X$ is integral, then $\Spec(\Symm \ccF^\tf)$ is an irreducible component of the abelian cone $\Spec \Symm \ccF$.  
\end{proposition}

\begin{proof}
    Locally, $X = \Spec R$ is affine,  $\ccF = \widetilde{M}$ for some finitely presented module $M$ over $R$ such that $M^\tf$ is locally free and $\ccI_Z = I$ is an ideal with $I \subseteq \Ann(\tor(M))$. Let $A = \Symm M$. Then $I\cdot \tor(A) = 0$ by \Cref{lem: support tor(M) kills tor(M)}, and 
    \Cref{lem: symm commutes with tf if locally free} ensures that $A^\tf = \Symm (M^\tf)$ is locally free. The result follows from \Cref{lem: decomposition cone algebraic version}.

    The claim about $\Spec(\Symm \ccF^\tf)$ being irreducible follows from \Cref{prop: main component} and \Cref{lem: symm commutes with tf if locally free}.
\end{proof}

\begin{remark}
Remember that the support $\supp(\ccF)$ of a coherent sheaf $\ccF$ can be defined set-theoretically by locally looking at the prime ideals where the stalk of $\ccF$ is non-zero. A scheme structure on $\supp(\ccF)$ is given by the sheaf $\Ann(\ccF)$. Therefore, the condition $\ccI_Z\subseteq \Ann(\tor(\ccF))$ in \Cref{prop: decomposition cone} implies that the closed $Z$ must contain $\supp(\tor(\ccF))$. 

Another natural scheme structure in $\supp(\ccF)$ is given by $F_0(\ccF)$, the 0-th Fitting ideal of $\ccF$. There is an inclusion $F_0(\ccF)\subseteq \Ann(\ccF)$ by \cite[Lemma 07ZA]{stacks-project}, thus in \Cref{prop: decomposition cone} we can also take the particular case where $\ccI_Z = F_0(\tor(\ccF))$.
\end{remark}

\subsection{A decomposition of the abelian cone of a diagonal sheaf}\label{subsec: decomposing torsion}

We continue our study of components of cones by further specializing to the abelian cone of a diagonal sheaf $\ccF$. The pushout description of $C(\ccF)$ in \Cref{prop: decomposition cone} is improved in \Cref{theorem:cone_as_a_union}: $C(\ccF)$ is topologically a union of vector bundles.\\

Let $\ccF$ be a diagonal sheaf on an integral Noetherian scheme $X$. Remember that $\ccF^\tf$ is locally free by \Cref{prop:morphism-HL-Rossi}.

First we reduce from rank $r$ to rank $0$. By \Cref{prop: decomposition cone} and \Cref{lem: symm commutes with tf if locally free}, we have a decomposition of $C(\ccF)$ as a pushout
\[
C(\ccF)=C(\ccF^{\tf})\bigsqcup_{C(\iota_*\ccF^{\tf}|_{{\mathrm{Supp}(\tor(\ccF))}})} C(i_*\ccF|_{\mathrm{Supp}(\tor(\ccF))})
\]
Here all cones are taken over $X$ and $C(\ccF^\tf)$ is an irreducible component by \Cref{prop: main component}. Replacing $\ccF$ by $i_*\ccF|_{\mathrm{Supp}(\tor(\ccF))}$, we may assyme that $\ccF$ has rank 0.

Let $\ccF$ be a rank 0 diagonal sheaf. Recall that, by \Cref{lemma:filtration_rkzero}, $\ccF$ has a filtration with quotients supported on some effective Cartier divisors $D_i$ for $i=1,\dots,n$. 

Consider the finite collection of closed integral subschemes $\{Z_{i}^j\}_j$, which are the irreducible components of $D_i$ taken with reduced structure. These are in the support of $\ccF$ and we will see in \Cref{lemma:extra_components} that $(\ccF|_{Z^j_{i}})^\tf$ is locally free. Note that these collections are not necessarily disjoint for different $i$'s. We denote the inclusion of $Z^j_{i}$ in $X$ simply by $\iota$, without keeping track of the indices when it is not necessary.

\begin{theorem}\label{theorem:cone_as_a_union}
    Let $\ccF$ be a diagonal sheaf of rank 0 on an integral Noetherian scheme $X$. The cone of $\ccF$ is topologically a union of finitely many irreducible components
    \[ C(\ccF)= \bigcup_{i,j} C\left((\ccF\mid_{Z^j_{i}})^\tf\right)\cup X
    \]
    where each $C\left((\ccF\mid_{Z^j_{i}})^\tf\right)$ a vector bundle supported on the integral subscheme $Z^j_{i}$.
\end{theorem}

\begin{lemma}\label{lemma:extra_components}
    With the previous notations and assumptions, the cone    \[C_{Z^j_{i}}\left((\ccF\mid_{Z^j_{i}})^\tf\right) \to Z^j_{i}\] 
    is a vector bundle of rank $r^j_{i}$, where  \[
    r^j_{i}=\max_k\{Z^j_{i}\subset D_k\}.
    \]
\end{lemma}

\begin{proof}
    Since $\tor(\ccF|_U)=\tor(\ccF)|_U$ for $U\subset X$ open, it is enough to prove it locally. We assume that $\ccF$ is the cokernel of a diagonal matrix 
    \[\Diag(f_1,\ldots,f_1, f_2,\ldots, f_2,\ldots, f_s),
    \]
    where $f_k$ divides $f_{k+1}$.
    Observe that, if $f_k\mid_{Z^j_{i}}=0$, then $f_\ell\mid_{Z^j_{i}}$ also vanishes for all $\ell>k$. Take $r^j_{i}$ as in the statement of the theorem: $Z^j_{i}$ is a component of $D_{r^j_{i}}$ and the latter divides the last ${r^j_{i}}$ entries.

    Then the matrix presentation of $\ccF$ on $Z^j_{i}$ looks like $\Diag(f_1\mid_{Z^j_{i}},\ldots,f_t\mid_{Z^j_{i}},0,\ldots,0)$ where $f_t\mid_{Z^j_{i}}\neq 0$. Since $Z^j_{i}$ is not, by assumption, a component of $Z(f_t)$ we see that the cokernel of $\Diag(f_1\mid_{Z^j_{i}},\ldots,f_t\mid_{Z^j_{i}})$ is a torsion sheaf and the torsion-free part of $\ccF\mid_{Z^j_{i}}$ is locally free of rank $r^j_{i}$.
\end{proof}
\begin{proof}[Proof of \Cref{theorem:cone_as_a_union}]
To check the claim set-theoretically, it suffices to argue that any closed point of $C(\ccF)$ is contained in at least one of the cones. Let $v\in C(\ccF)$, the projection to $X$ is $x\in X$. Then $v$ is specified by some section $x\to \ccF|_x$. If $x\notin \bigcup_{i,j}Z^j_{i}$, $\ccF|_x=0$, so we are done. Otherwise, we need to argue that $\ccF|_{x}\cong ((\ccF|_{Z^j_{i}})^{\tf})|_x$ for some $i,j$.

Let $i$ be such that $x\in Z^j_{i}$ for some $j$ but $x\notin Z^\ell_{k}$ for all $k>i$ and all $\ell$. Then $x\in D_i$ but $x\notin D_k$ for any $k>i$.

By the construction of the $D_i$'s, we know that $\supp(\tor(\ccF|_{Z^j_{i}}))\subset \bigcup_{k>i,\ell}Z^\ell_{k}$. Then $(\ccF_{Z^j_{i}}^{\tf})|_x=\ccF|_x$, and we are done.

The morphism $\bigcup_{i,j}C((\ccF|_{Z_i^j})^{\tf})\to C(\ccF)$ of topological spaces, given by the universal property of push-outs, is continuous and closed for the Zariski topology. Since we have just checked that it is also bijective, it is a homeomorphism.
\end{proof}

\begin{example}\label{point in plane} 
    The following example of blowing up the origin in $\bbA^2$ is simple, but it captures much of the essence of the decomposition in \Cref{prop: decomposition cone} (see \eqref{eq:component,Bl,orign,1} and \eqref{eq:component,Bl,orign,2}). We present it with full detail.

    Let $R=k[x,y]$ and let $I=(x,y)$ be the ideal of the origin. The sheaf $\ccF = \widetilde{I}$ on $\bbA^2 = \Spec(R)$ is torsion-free but not locally free. Since $\ccF$ is an ideal sheaf, $\Bl_{\ccF} \bbA^2 = \Bl_0 \bbA^2$ is just the usual blow up of $\bbA^2$ along the origin. Let $p\colon \Bl_0\bbA^2 \to \bbA^2$ be the natural projection. Then $p^*\ccF$ is not torsion-free, but $(p^*\ccF)^{\tf}$ is a line bundle. 
    To see that, we start with the following resolution of $\ccF$. 
    \[
        \begin{tikzcd}[ampersand replacement=\&]
            0\arrow{r} \& R\arrow{r}{
            \begin{pmatrix}
                -y \\
                x
            \end{pmatrix}
        } \& R\oplus R \arrow{r}{
            \begin{pmatrix}
                x & y
            \end{pmatrix}
            } \& I\arrow{r} \& 0.
        \end{tikzcd}
    \]
    Pulling back along $p$, we obtain a presentation of $p^*\ccF$:
    \[
        \begin{tikzcd}[ampersand replacement=\&]
             \ccO_{\Bl_0\bbA^2}\arrow{r}{
            \begin{pmatrix}
                -ey' \\
                ex'
            \end{pmatrix}
            } \& \ccO_{\Bl_0\bbA^2}\oplus \ccO_{\Bl_0\bbA^2} \arrow{r}{
            \begin{pmatrix}
                ex' & ey'
            \end{pmatrix}
            } \& p^*\ccF\arrow{r} \& 0.
        \end{tikzcd}
    \]
    Here $e$ is a local coordinate for the exceptional divisor $E\subseteq \Bl_0 \bbA^2$ and $x'$ and $y'$ correspond to the strict transforms of $x$ and $y$. This induces a commutative diagram
    \[
        \begin{tikzcd}[ampersand replacement=\&]
            \& \& 0\arrow{d} \arrow{r} \& \Coker(e)\arrow{d}{id} \\
            0\arrow{r} \& \ccO_{\Bl_0\bbA^2} \arrow{r}{\cdot e}\arrow{d}{\begin{pmatrix}
                -ey' \\
                ex'
            \end{pmatrix}} \& \ccO_{\Bl_0\bbA^2}(E)\arrow{r}\arrow{d}{\begin{pmatrix}
                y' \\
                -x'
            \end{pmatrix}} \& \Coker(e)\arrow{r}\arrow{d} \& 0\\
            0\arrow{r} \& \ccO_{\Bl_0\bbA^2}\oplus \ccO_{\Bl_0\bbA^2} \arrow{r}{\id}\arrow{d}{\begin{pmatrix}
                    x'e & y'e
            \end{pmatrix}} \& \ccO_{\Bl_0\bbA^2}\oplus \ccO_{\Bl_0\bbA^2}\arrow{r}\arrow{d} \& 0\arrow{r}\arrow{d} \& 0\\
            \& p^* \ccF \arrow{r} \& \Coker(y',-x')^t\arrow{r} \& 0
        \end{tikzcd}
    \]
    Applying the Snake Lemma and using that $\Coker(e) \simeq \ccO_E(E)$ and that $\Coker(y',-x')$ is the ideal sheaf generated by $x'$ and $y'$, we get a short exact sequence
    \[
        0\to \ccO_E(E) \to p^*\ccF\to (x',y')\to 0.
    \]
    It follows that $\tor(p^*\ccF) \simeq \ccO_E(E)$ and $(p^*\ccF)^\tf \simeq (x',y')$. In particular, $p^\ast\ccF$ is not torsion-free.

    We can also describe the geometry of the abelian cones $\Spec \Symm \ccF$ and $\Spec \Symm (p^\ast \ccF)$. We have 
    \[
        \pi_\ccF\colon \Spec \Symm \ccF = \Spec (R[X,Y]/(xY-yX)) \to \bbA^2,
    \]
    which is irreducible and singular. 

    Next we describe $\Spec \Symm (p^\ast \ccF)$. Let $S = k[x,y,x',y']/(xy'-yx')$ where the variables $x',y'$ have degree 1 and $x,y$ have degree 0. Then
    \[
        \Bl_0 \bbA^2 = \Proj(k[x,y,x',y']/(xy'-yx')).
    \]
     A local equation for $E$ is given by $e=x/x'$ or $e=y/y'$, depending on the chosen chart. Then 
     \[
        \pi_{p^*\ccF}\colon \Spec \Symm (p^\ast \ccF) = \Spec S[X,Y]/(e(x'Y-y'X)) \to \Bl_0 \bbA^2
    \]
    is reducible. It has two components:
    \begin{align}
    \label{eq:component,Bl,orign,1}    \pi_{p^*\ccF,{\rm main}}\colon C_{\rm main} &= V(x'Y-y'X) \to \Bl_0 \bbA^2,\\
    \label{eq:component,Bl,orign,2}    \pi_{p^*\ccF,{\rm tor}}\colon C_{\rm tor} &= V(e) \to \Bl_0 \bbA^2.
    \end{align}
    The main component $C_{\rm main}$ equals $\Spec\Symm (p^\ast\ccF)^\tf$ and it is a vector bundle of rank 1. Meanwhile, $C_{\rm tor}$ corresponds to $\tor(\ccF)$, it is supported over $E$ and it is a vector bundle of rank 2 over its support.
\end{example}

\section{Application to stable maps}\label{section:application-stable-maps}

In this section we apply the results in \Cref{sec:desing_stacks} to construct reduced Gromov--Witten invariants. 

Given $X$ a smooth subvariety in a projective space $\bbP^r$, there is an embedding of the moduli space of stable maps to $X$ in the moduli space of stable maps to $\bbP^r$. The mo\-du\-li space of genus \emph{zero} stable maps to a projective space $\bbP^r$ is a smooth irreducible DM stack. If $X$ is a hypersurface of degree $k$ (or more generally a complete intersection) in $\bbP^r$, there is a locally free sheaf $\ccE_k$ on the moduli space of stable maps to $\bbP^r$, such that  the moduli space of maps to $X$ is cut out by the zero locus of a section of this sheaf. These statements are not true in higher genus. In general, the moduli space of stable maps to $\bbP^r$ has several irreducible components of different dimensions. We still have a natural sheaf $\ccE_k$ equipped with a section, but $\ccE_k$ is not locally free: its rank is different on different irreducible components. 

There are several ways to use \Cref{sec:desing_stacks} to fix the above problem (see \Cref{rem: various blowups}). In this section we are concerned with finding and comparing various blow-ups the Picard stack along certain sheaves, which fix the above problem. More precisely, we consider $\wPic\to\Pic$, such that $\tGw:=\Gw\times_{\Pic}\wPic$ desingularizes $\ccE_k$. 

Under the assumption $d>2g-2$ (see Assumption \ref{assum:d,big,2g,minus2}), we define $\tmGwx$ via the following Cartesian diagram
\[
\begin{tikzcd}
\tmGwx \ar[r] \ar[d]\arrow[dr, phantom,"\ulcorner", very near start]&\tmGw
\ar[d] \\
\Gwx\ar[r]&\Gw,
\end{tikzcd}
\]
where $\tmGw$ is the main component of the cone $\tGw$ (see \Cref{main comp blow up}). We then define reduced invariants (see Definition \ref{def:reduced_GW_invariants}) via an obstruction theory on $\tmGw$ relative to $\wPic_k$ (see Theorem \ref{th-blow-up-maps}). 

We also recall maps with fields \cite{Chang-Li-maps-with-fields} and then we construct a blow-up of it which makes the resulting stack as simple as possible. The resulting stack gives an alternative definition of reduced invariants, which is not intrinsic; the relation between these two invariants is similar in spirit to a Quantum Lefschetz theorem. The definition we give is more intrinsic, but working with maps with fields instead of maps is more suited to approaching \Cref{conj one} and \Cref{conj gv}. See \cite{Chang-Li-hyperplane-property, Lee-Oh-reduced-complete-intersections-2,Lee-Oh-reduced-complete-intersections} for the proof of \Cref{conj gv} in genus one and two.

\subsection{Stable maps as open in an abelian cone}\label{subsec:stable_maps}

We recall how the moduli space of stable maps to projective space can be seen as an open substack of an abelian cone, following \cite{Chang-Li-maps-with-fields}. This observation motivates our study of components of cones in \Cref{sec:components_abelian_cones}, as components of the ambient abelian cone are related to components of stable maps.\\

Let $\ffM_{g,n}$ denote the stack of genus $g$ pre-stable curves with $n$ marked points, that is, $\ffM_{g,n}$ parametrizes connected projective at-worst-nodal curves of arithmetic genus $g$ with $n$ distinct smooth marked points. Let $\uC_{g,n}$ denote universal curve over $\ffM_{g,n}$.
Let $\Pic_{g,n,d}$  denote the Artin stack which parameterises genus $g$ pre-stable curves, with $n$ marked points, together with a line bundle of degree $d$. Let $\Pic_{g,n,d}^{\rm st}$ denote the open subset of $\Pic_{g,n,d}$ consisting of $(C, p_1\ldots p_n, L)$ which satisfy the stability condition 
\begin{equation}\label{eq: stability condition}    L^{\otimes3}\otimes\omega_C\left(\sum_{i=1}^n p_i\right) \mbox{ is ample.}
\end{equation}
 Notice that $\ffM_{g,n}$ and $\Pic_{g,n,d}$ are not separated, but they are smooth (see \cite[Lemma 0E6W]{stacks-project} and \cite[Proposition 2.11]{C-FKM}) and irreducible. The stack $\Pic_{g,n,d}$ is locally Noetherian and the stack $\Pic_{g,n,d}^{\rm st}$ is Noetherian.

 \begin{notation}\label{dropindex}
 From now on, we fix $g,n,d$ and the stability condition and we drop all the indices.
 \end{notation}
 
We define $\ffC$ the universal curve over $\Pic$ by the Cartesian diagram \eqref{eq: univ curve Pic}. Notice that we also have a universal line bundle $\ffL$ over $\ffC$. 
\begin{equation}\label{eq: univ curve Pic}
\begin{tikzcd}
\ffL \arrow[d]&\\
\ffC \arrow[d, "\pi"] \arrow[r] \arrow[dr, phantom,"\ulcorner", very near start] 
& \uC \arrow[d] \\
\Pic \arrow[r] & \ffM\text{.}
\end{tikzcd}
\end{equation}

We form the cone of sections of $\ffL$ as in Chang-Li (\cite{Chang-Li-hyperplane-property}, 
Section 2)
\begin{equation}\label{eq:CL,cone,defi}
S(\pi_* \ffL)
:=\Spec\Symm (R^1 \pi_*(\ffL^* \otimes \omega_{\ffC/\Pic}))\to \Pic.
\end{equation}
In the following we collect a list of remarks on the cone of sections defined above.

\begin{enumerate}
    \item In \cite[Proposition 2.2]{Chang-Li-maps-with-fields}, the authors show that $S(\pi_*\ffL)$ is the moduli stack parameterizing $(C,L,s)$ with $(C,L)\in \Pic$ and ${s}\in H^0(C,L)$. Be aware that our $S(\pi_*\ffL)$ is denoted by $C(\pi_*\ffL)$ in \cite{Chang-Li-maps-with-fields}.

    \item This situation is similar to the discussion in \Cref{def ab cones} about the total space of a locally free sheaf. If $\ffE$ is a locally free sheaf over $\Pic$, then sections of $\ffE$ correspond to sections of the vector bundle $\Tot(\ffE) = \Spec \Symm (\ffE^*)$ over $\Pic$, but the same is not true if $\ffE$ is not locally free.

    \item In our set-up, the sheaf $R^0\pi_*\ffL$ is not locally free. However, since we work with the universal family of curves $\ffC\to \Pic$, sections of the sheaf $R^0\pi_*\ffL$ correspond to sections of the abelian cone of its \emph{Serre dual} $R^1\pi_*(\ffL^* \otimes \omega_{\ffC/\Pic})$. 
    This is proven in  \cite[Proposition 2.2]{Chang-Li-maps-with-fields}.

    \item Note that $R^0\pi_\ast \ffL$ does not commute with base change but $R^1\pi_\ast \ffL$ does by cohomology and base change.
\end{enumerate}

For the rest of the section, let
\begin{align}\label{eq:F_sheaf}
\ffF:=R^1\pi_*(\ffL^*\otimes \omega_{\ffC/\Pic}).
\end{align}
Note that since $\pi$ is proper, we have that $\ffF$ is a coherent sheaf on $\Pic$. As defined in \Cref{def ab cones}, we consider the stack $\Spec \Symm \ffF$, which is an abelian cone stack over $\Pic$.

Let $\Gw$ be the moduli space of genus $g$, degree $d$ stable maps, with $n$ marked points. 

\begin{proposition}(\cite[Proposition 2.7]{Chang-Li-maps-with-fields}, \cite[Theorem 3.2.1]{C-FK})\label{prop: open embedding stable maps}
    The moduli space $\Gw$ is an open substack, cut out by the basepoint-free condition, of the stack
    \begin{align}\label{eq:cone,moduli}
         S(\pi_*\ffL^{\oplus r+1})=\Spec\Symm(\oplus_{i=0}^r \ffF)\to \Pic.
    \end{align}
\end{proposition}

As before, a point of this cone over $(C,L)\in \Pic$ is $(C,L,\underline{s})$ with $\underline{s}\in H^0(C,L)^{\oplus r+1}$. Note that 
\[
S(\pi_*\ffL^{\oplus r+1})= \overbrace{S(\pi_*\ffL)\times_\Pic  \cdots \times_\Pic S(\pi_*\ffL)}^{r+1 \rm\ times} \to \Pic.
\]

We define $\ccL, \ccC$ by the following Cartesian diagram
\[
\begin{tikzcd}
\ccL \arrow[d]\arrow[r]\arrow[dr, phantom,"\ulcorner", very near start]&\ffL\arrow[d]\\
\ccC \arrow[d, "\opi"] \arrow[r] \arrow[dr, phantom,"\ulcorner", very near start] 
& \ffC \arrow[d, "\pi"] \\
S(\pi_*\ffL^{\oplus r+1}) \arrow[r,"\mu"] & \Pic\text{.}
\end{tikzcd}
\]
By \cite{Chang-Li-maps-with-fields}, the complex $$\oplus_{i=0}^rR^{\bullet}\opi_*\ccL$$ is a dual obstruction theory for the natural projection $$\mu:S(\pi_*\ffL^{\oplus r+1})=\Spec\Symm (\oplus_{i=0}^r\ffF)\to\Pic.$$

This perfect obstruction theory induces a virtual fundamental class 
\[[\Gw]^\vir:=\mu^![\Pic]\in A_*(\Gw).
\]

\begin{proposition}\label{dual-cone} 
We use \Cref{dropindex}. For $\ffF$ defined in \cref{eq:F_sheaf}, we have an isomorphism of sheaves $$
\ffF^{*}:=\left(R^1\pi_*\ffL^{*}\otimes\omega_{\ffC/\Pic}\right)^*\simeq \pi_*\ffL$$ over $\Pic$.
\end{proposition}

\begin{proof}
Using Grothendieck duality we have
    \begin{align}\label{verdier}
 \notag  \left(R^{\bullet}\pi_*\ffL^{*}\otimes\omega_{\ffC/\Pic}\right)^*&= R\ccH om_{\Pic}(R^{\bullet}\pi_*\ffL^{*}\otimes\omega_{\ffC/\Pic},\ccO_\Pic)\\
 \notag &= R^{\bullet}\pi_*R\ccH om_{\Pic}(\ffL^{*}\otimes \omega_{\ffC/\Pic},\omega_{\ffC/\Pic}[1])\\
  \notag  & = R^{\bullet}\pi_*R\ccH om_{\ffC}(\ccO_{\ffC},\ffL\otimes \omega_{\ffC/\Pic}^*\otimes\omega_{\ffC/\Pic}[1])\\
    \notag   & = R^{\bullet}\pi_*R\ccH om_{\ffC}(\ccO_{\ffC},\ffL[1])\\
      &= R^{\bullet}\pi_* \ffL[1].
    \end{align}
   
On the one hand, we have that
\begin{equation}\label{lhs}h^{-1}(R^{\bullet}\pi_* \ffL[1])=\pi_*\ffL.
\end{equation}

In the following we look at an explicit resolution of $\left(R^{\bullet}\pi_*\ffL^{*}\otimes\omega_{\ffC/\Pic}\right)^*$ and compute its $h^{-1}$. This is similar to the discussion in \cite{ciocan2020quasimap}, Section 3.2. By the stability condition on $\Pic$ (see \Cref{eq: stability condition}), the universal curve over $\Pic$ is projective. This ensures that we have an ample section on $\ffC$ and we take $A$ sufficiently large so that $R^0\pi_*\ffL^{*}(-A)\otimes\omega_{\ffC/\Pic}=0$. 

We now consider the exact sequence of sheaves on the universal curve over $\Pic$
\[ 0\to \ffL^{*}(-A)\otimes\omega_{\ffC/\Pic}\to \ffL^{*}\otimes\omega_{\ffC/\Pic} \to \ffL^{*}\otimes\omega_{\ffC/\Pic}|_A\to 0.
\]
Pushing forward the above to $\Pic$, we get a long exact sequence
\begin{align} \label{dual resolution}
\notag 0\to R^0\pi_*\ffL^{*}\otimes\omega_{\ffC/\Pic} \to &R^0\pi_*\ffL^{*}\otimes\omega_{\ffC/\Pic}|_A\to\\ 
 \to &R^1\pi_*\ffL^{*}(-A)\otimes\omega_{\ffC/\Pic}\to R^1\pi_*\ffL^{*}\otimes\omega_{\ffC/\Pic} \to 0.
\end{align}
This gives 
\begin{equation}\label{dual resol}
    R^{\bullet}\pi_*\ffL^{*}\otimes\omega_{\ffC/\Pic}\simeq [R^0\pi_*\ffL^{*}\otimes\omega_{\ffC/\Pic}|_A\to R^1\pi_*\ffL^{*}(-A)\otimes\omega_{\ffC/\Pic}],
\end{equation}
with the complex on the right, being a complex of vector bundles supported in $[0,1]$. This gives
\[(R^{\bullet}\pi_*\ffL^{*}\otimes\omega_{\ffC/\Pic})^{*}\simeq [ (R^1\pi_*\ffL^{*}(-A)\otimes\omega_{\ffC/\Pic})^{*}\to (R^0\pi_*\ffL^{*}\otimes\omega_{\ffC/\Pic}|_A)^{*}],
\]
with the complex on the right being supported in $[-1,0]$.

Applying the functor $\mathcal{H}om(-, \ccO)$ to \eqref{dual resolution} and using that it is left-exact, we get
\begin{equation}\label{resF}
0\to \left(R^1\pi_*\ffL^{*}\otimes\omega_{\ffC/\Pic}\right)^*\to  (R^1\pi_*\ffL^{*}(-A)\otimes\omega_{\ffC/\Pic})^{*}\to (R^0\pi_*\ffL^{*}\otimes\omega_{\ffC/\Pic}|_A)^{*}
\end{equation}
This together with \cref{dual resol} shows that  
\begin{equation}\label{rhs}
h^{-1}\left(\left(R^{\bullet}\pi_*\ffL^{*}\otimes\omega_{\ffC/\Pic}\right)^*\right)=\left(R^1\pi_*\ffL^{*}\otimes\omega_{\ffC/\Pic}\right)^*.
\end{equation}

    \Cref{verdier,lhs,rhs} imply that 
    \[
    \left(R^1\pi_*\ffL^{*}\otimes\omega_{\ffC/\Pic}\right)^*=R^0\pi_* \ffL. \qedhere
    \]
\end{proof}

\begin{remark}\label{rem:resol l}

    As in the proof of \Cref{dual-cone}, we have an explicit resolution of $\ffF$ (see \cite{ciocan2020quasimap}, Section 3.2). Let $A$ be a sufficiently high power of a very ample section of the morphism $\ffC\to\Pic$ such that $R^1\pi_*\ffL(A)=0$. The short exact sequence
\begin{equation}\label{a twist}
0\to \ffL\to \ffL(A)\to \ffL(A)|_A\to 0.
\end{equation}
induces a long exact sequence
\begin{equation}\label{resL}
0\to R^0\pi_*\ffL\to  R^0\pi_*\ffL(A)\to  R^0\pi_*\ffL(A)|_A\to R^1\pi_*\ffL\to 0,
\end{equation}
which shows that $[R^0\pi_*\ffL(A)\to  R^0\pi_*\ffL(A)|_A]$ is quasi-isomorphic to $R^{\bullet}\pi_*\ffL$.

With our choice of $A$ we have that $R^0\pi_*\ffL(A)$ and $R^0\pi_*\ffL(A)|_A$ are locally free sheaves over $\Pic$. 
Sequence \eqref{resL} together with the fact that $R^0\pi_*\ffL(A)$ is a locally free sheaf over $\Pic$ implies that $R^0\pi_*\ffL$ is a torsion-free sheaf on $\Pic$.
\end{remark}

\begin{remark}\label{rem:resolution}
  The proof of \Cref{dual-cone} shows that we have a resolution of $\ffF:=R^1\pi_*\ffL^{*}\otimes\omega_{\ffC/\Pic}$ to the left given by \eqref{dual resolution}. The isomorphism in \Cref{verdier} shows that the complex
 \[[R^0\pi_*\ffL^{*}\otimes\omega_{\ffC/\Pic}|_A\to R^1\pi_*\ffL^{*}(-A)\otimes\omega_{\ffC/\Pic}]\]
 is dual to
  \[[R^0\pi_*\ffL(A)\to  R^0\pi_*\ffL(A)|_A].\]
  We will use this duality in \Cref{sec:comparisons}. Hu and Li work with the resolution to the right we have in \eqref{resL}. In the previous sections we used the resolution to the left \eqref{dual resolution}. Since dual morphisms have the same Fitting ideals, both morphisms give the same Hu--Li blow-up.
\end{remark}

\subsection{The main component of stable maps to $\bbP^r$}\label{subsec: main component stable maps}

In the following we look at stable maps with a lower bound on the degree (see Assumption \ref{assum:d,big,2g,minus2}). In this situation the moduli space of stable maps has a \emph{main (irreducible) component}. We discuss this main component and its relation to the main component of abelian cones.\\

We fix the following assumption from now on.

\begin{assumption}\label{assum:d,big,2g,minus2}
In the following we fix $d>2g-2$. For $C$ a smooth genus $g$ curve, $L$ a line bundle of degree $d$, and $d>2g-2$, we have that $H^1(C,L)=0$. 
This shows that for $d>2g-2$ the locus $\Gwo$ of stable maps with smooth domain is smooth. Note that, under the degree assumption, $S(\pi_*\ffL^{\oplus r+1})$ has a smooth surjective morphism  to the substack of $\Pic_{g,n,d}$ where the underlying curve is smooth, which is connected. The fibers are open in $\bbP(H^0(C,L)^{\oplus r+1})$ so they are connected. This implies that $\Gwo$ is either empty or connected. Thus $\Gwo$ is irreducible, so its closure, if non-empty, is an irreducible component of $\Gw$.

\end{assumption}

\begin{definition}[Main Component]\label{def:main moduli stable maps} 
Consider the Zariski closure in $\Gw$ of the locus $\Gwo$ where the curve is smooth.  We call this component the \textit{main component} and we denote it by $\mGw$. 
\end{definition}  

We introduced the main component of an abelian cone in \Cref{def: main component abelian cone}. In our next result, \Cref{prop:main_component_Pr_open_in_tors_free}, we show that the main component of $\Gw$ is contained in the main component of $\Spec\Symm (\oplus_{i=0}^r\ffF)$. By the proof of \cref{prop:main_component_Pr_open_in_tors_free}, on $\mGw$ the universal curve is generically smooth and $\pi_* \ccL$ is generically a vector bundle.

\begin{proposition}\label{prop:main_component_Pr_open_in_tors_free}
    We have that $\mGw$ is an open substack of
    $\Spec(\Symm\oplus_{i=0}^r\ffF)^{\tf}$.
\end{proposition}

\begin{proof}
    Let $\Gwo$ and $\Pic^{\sm}$ denote the open substacks of $\Gw$ and $\Pic$ where the curve is smooth. The first step is to show that the sheaf $\ffF=R^1\pi_*(\ffL^*\otimes \omega_{\ffC/\Pic})$ is locally free over $\Pic^\sm$.

    Let $\pi^\sm\colon \ffC^\sm\to \Pic^\sm$ denote the universal curve and $\ffL^\sm$ the universal line bundle on $\ffC^\sm$. Assumption \ref{assum:d,big,2g,minus2} and cohomology and base change ensure that $R^1\pi^\sm_\ast \ffL^\sm= 0$. 

    Therefore, $R^0\pi^\sm_\ast \ffL^\sm$ has constant rank, so it locally free. Using Serre Duality and local freeness of $\omega_{\pi^\sm}$, we see that 
    \[
        R^1\pi^\sm_*(\ffL^{\sm *} \otimes \omega_{\pi^\sm}) \simeq
        (R^0\pi^\sm_*((\ffL^{\sm *} \otimes \omega_{\pi^\sm})^*\otimes\omega_{\pi^\sm}))^* 
        \simeq (R^0\pi^\sm_*\ffL^\sm)^*
    \]
    is locally free. 
 
    In particular, the $\ccO_{\Pic}$-algebra $(\Symm\oplus_{i=0}^r\ffF)^{\tf}$ is locally free over $\Pic^\sm$.

    To simplify the notation, let $C = \Spec \Symm\oplus_{i=0}^r\ffF$ and $C^\tf = \Spec(\Symm\oplus_{i=0}^r\ffF)^{\tf}$. By \Cref{prop: AM} and using that $\Pic^\sm$ is generically reduced, we have that
    \begin{equation}\label{eq: aux closure smooth locus}
        \scl_C(\pi^{-1}(\Pic^\sm)) = C^\tf,
    \end{equation}
    where $\pi^{-1}(\Pic^\sm)$ denotes the fibre product $\Pic^\sm\times_\Pic C$, which is open in $C$.

    We conclude by the following chain of equalities
    \begin{align*}
        \mGw &= \scl_{\Gw}(\pi^{-1}(\Pic^\sm)\times_{C}\Gw) \\ &= \scl_C(\pi^{-1}(\Pic^\sm))\times_C \Gw \\ &= C^\tf \times_C \Gw.
    \end{align*}
    The first equality is the definition of $\mGw$, the second one is \Cref{lem: closure operation} (applied to $Y=C$ and $W, V$ reduced open subschemes of  $\Gw$ and $\pi^{-1} (\Pic^\sm)$ respectively), and the last one is \Cref{eq: aux closure smooth locus}.
\end{proof}

\begin{lemma}\label{lem: closure operation}
    Let $Y$ be a scheme,  let $V, W$ be reduced open subschemes of $Y$. Then
    \[
        \scl_{W} (V\times_Y W) = \scl_Y(V)\times_Y W
    \]
\end{lemma}

\begin{proof}
    Since $V$ and $W$ are reduced, so is $V\times_Y W$. This means that the schematic closure is just the topological closure with the reduced induced structure by \cite[Lemma 056B]{stacks-project}. Therefore the question is purely topological, and it is straightforward using that $\scl_Y W = \cl_Y W$ is the intersection of all the closed subsets $C$ of $Y$ that contain $W$.
\end{proof}

\subsection{Blow-ups of the moduli space of stable maps}\label{subsec: blowups stable maps Pr}

In this section we consider a desingularization $p\colon \wPic\to \Pic$ of $\ffF$ and the base change $\tGw$ of $\Gw$. By compatibility of abelian cones with pullback, $\tGw$ is an open substack of $C(\oplus_{i=0}^r p^*\ffF)$. We define the main component $\tmGw$ of $\tGw$ to be the closure of the smooth locus (see \Cref{main comp blow up}). This definition ensures that $\tmGw$ is open in the main component of the ambient abelian cone $C(\oplus_{i=0}^r p^*\ffF)$ (\Cref{prop:main_component_Pr_open_in_tors_free}), thus it is irreducible. In general, $\tmGw$ does not agree with the pullback of $\mGw$, which might be reducible (see \Cref{rmk: main component is not pull back}). 
Finally, we induce a virtual fundamental class on $\tGw$.

 We define 
 \begin{equation}\label{eq:blowup,pic,moduli,stabe,projective,space}
    p\colon \wPic \to  \Pic
 \end{equation}

 to be any desingularization (as in \Cref{def: desingularization on stacks}) of the sheaf $\ffF$ (defined in \Cref{eq:F_sheaf}). By  \cref{Rossi univ}, we have a proper birational map 
\begin{equation}
\wPic\to \Bl_{\ffF}\Pic
\end{equation}
where $\Bl_{\ffF}\Pic\rightarrow\Pic$ is the Rossi blow-up. We are mainly interested in $\wPic$ being the Rossi or the Hu-Li blow-up (see \Cref{sec:desing_stacks}).

We define
\begin{equation}\label{base change maps}
\begin{tikzcd}
\tGw \arrow[d, "\widetilde{\mu}"] \arrow[r,"\overline{p}"] \arrow[dr, phantom,"\ulcorner", very near start] 
& \ \arrow[d, "\mu"] \Gw \\
\wPic \arrow[r,"p"]  & \Pic\text{.}
\end{tikzcd}
\end{equation}
Note that $\tGw$ is proper since $p$ and $\Gw$ are proper.
Let $p:\wPic\to \Pic$ be the natural projection. Consider the Cartesian diagram
\[
\xymatrix{\widetilde\ffL\ar[r]\ar[d]&\ffL\ar[d]\\
\widetilde\ffC\ar[r]^{q}\ar[d]_{\widetilde{\pi}}&\ffC\ar[d]^{\pi}\\
\wPic\ar[r]^{p}&\Pic}
\]
where $\ffC$ is the universal curve over $\Pic$ and $\ffL$ the universal line bundle. Recall that $$\ffF=R^1\pi_*(\ffL^{*}\otimes\omega_{\pi}).$$
\begin{lemma}\label{open in tilde} In notation as before, we have an open embedding
\[
\tGw\hookrightarrow  \Spec\Symm\left(\oplus_{i=0}^r p^*\ffF\right)\cong\Spec\Symm\left(\oplus_{i=0}^r R^1\widetilde{\pi}_*(\widetilde{\ffL}^{*}\otimes\omega_{\widetilde{\pi}})\right).
\]
\end{lemma}
\begin{proof}
 By \Cref{prop: open embedding stable maps}, we have an open embedding 
 \[\Gw\subset\Spec_{\Pic}\Symm \oplus_{i=0}^r\ffF.\] Thus, $\tGw\subset\Spec_{\wPic}\Symm p^*\ffF^{\oplus r+1}$ and $p^*\ffF\cong R^1\widetilde{\pi}_*(\widetilde{\ffL}^{*}\otimes\omega_{\widetilde{\pi}})$.
\end{proof}
\begin{definition}\label{main comp blow up} 
Consider the closure in $\tGw$ of the locus where the curve is smooth of genus $g$.  We call this component the \textit{main component} and we denote it by $\tmGw$. 
\end{definition} 
\begin{proposition}\label{proper main} The following hold:
\begin{enumerate}

    \item We have an open embedding \[\tmGw\subset  \Spec\left(\Symm\oplus_{i=0}^rp^*\ffF\right)^{\tf} \simeq \Spec\Symm\left(\oplus_{i=0}^r\left(p^*\ffF\right)^{\tf}\right).\]
\item $\tmGw$ is proper and smooth over $\wPic$.
\end{enumerate}
    
\end{proposition}
\begin{proof}
    By \Cref{open in tilde} we get $\tmGw\hookrightarrow \Spec\left(\Symm\oplus_{i=0}^r p^*\ffF\right)^\tf$ is an open embedding. By cohomology and base change we have $p^*\ffF\simeq R^1\widetilde{\pi}_*(\widetilde{\ffL}^{*}\otimes\omega_{\widetilde{\pi}})$. The argument in  \Cref{prop:main_component_Pr_open_in_tors_free} applies to $R^1\widetilde{\pi}_*(\widetilde{\ffL}^{*}\otimes\omega_{\widetilde{\pi}})$ and we obtain an open embedding
    \[
    \tmGw\hookrightarrow  \Spec \left(\Symm\oplus_{i=0}^r R^1\widetilde{\pi}_*(\widetilde{\ffL}^{*}\otimes\omega_{\widetilde{\pi}})\right)^\tf.
    \]

    By construction $(p^*\ffF)^{\tf}$ is locally-free and, since the torsion-free part commutes with direct sums, the same is true for $(p^*\ffF^{\oplus r+1})^{\tf}$. This shows that $\Symm\oplus_{i=0}^r(p_k^*\ffF)^{\tf}$ is locally free therefore we have an isomorphism
    \[
        \Spec\left(\Symm\oplus_{i=0}^rp^*\ffF\right)^{\tf} \simeq \Spec\Symm\left(\oplus_{i=0}^r\left(p^*\ffF\right)^{\tf}\right)
    \]
    by \Cref{lem: symm commutes with tf if locally free}.
    
    The first part implies that $\tmGw$ is smooth over $\wPic$ since it is open in a vector bundle. The properness of $\tmGw$ follows from the properness of $\tGw$.
\end{proof}

\begin{remark}\label{rmk: main component is not pull back}
   We consider the following Cartesian diagram:
\begin{equation}
\begin{tikzcd}
\widetilde{\ccM}^{\circ}(\bbP) \arrow[d]\arrow[r] \arrow[dr, phantom,"\ulcorner", very near start] 
& \ \arrow[d, "\mu"] \mGw \\
\wPic \arrow[r,"p"]  & \Pic\text{.}
\end{tikzcd}
\end{equation}
In general, $\tmGw\hookrightarrow\widetilde{\ccM}^{\circ}(\bbP)$ is not an isomorphism and thus the diagram below is only commutative
\begin{equation}\label{diag:Bl,GW}
\begin{tikzcd}
\tmGw \arrow[d, "\widetilde{\mu}"] \arrow[r,"\overline{p}"] 
& \ \arrow[d, "\mu"] \mGw \\
\wPic \arrow[r,"p"]  & \Pic\text{.}
\end{tikzcd}
\end{equation}

This observation is a reflection of the fact that torsion-free part does not commute with pullback, see \Cref{rmk: symm and tf do not commute in general}. By pullback, \Cref{prop:main_component_Pr_open_in_tors_free} induces an open embedding
\[
    \widetilde{\ccM}^{\circ}(\bbP)\subset  \Spec p^*\left(\Symm(\oplus_{i=0}^r\ffF)\right)^{\tf},
\]

which in general need not factor through $\Spec \left(\Symm(\oplus_{i=0}^rp^*\ffF)\right)^{\tf} \not\simeq \Spec p^*\left(\Symm(\oplus_{i=0}^r\ffF)\right)^{\tf}$. Meanwhile, as in the proof of \Cref{prop:main_component_Pr_open_in_tors_free}, we have that
\[
    \tmGw = \Spec \left(\Symm(\oplus_{i=0}^rp^*\ffF)\right)^{\tf} \cap \tGw.
\]

However, $\widetilde{\ccM}^{\circ}(\bbP)$ and $\tmGw$ do give the same invariants. To see this, note that we have commuting morphisms
\[
\begin{tikzcd}
\tmGw\ar[r]\ar[rd,"\overline{p}"']&\widetilde{\ccM}^{\circ}(\bbP)\ar[d, "\overline{r}"]\\
&\mGw
\end{tikzcd}
\]
and 
\begin{equation}\label{two defs}
\overline{p}_*[\tmGw]=\overline{r}_*[\widetilde{\ccM}^{\circ}(\bbP)]=[\mGw].
\end{equation}
\end{remark}

Let $\wpi:\widetilde{\ccC}\to \tGw$ be the universal curve and let $\widehat{q}:\widetilde{\ccC}\to \ccC$ the morphism induced by $q$. The morphism $\widetilde{\mu}$ has a dual perfect obstruction theory given by a morphism
\[
\phi_{\widetilde{\mu}}: \bbT_{\widetilde{\mu}}\to\oplus_{i=0}^r R^{\bullet}\wpi_*\widehat{q}^*\ccL.
\]

This perfect obstruction theory induces a virtual fundamental class

\[[\tGw]^\vir:=\widetilde{\mu}^![\wPic]\in A_*(\tGw),
\]
where $\widetilde{\mu}^!$ is defined as in \cite{manolache2011virtual}.

\begin{remark}
    While $\Pic$ is smooth, $\wPic$ does not need to be smooth. This is not a problem, all we need for a well-defined virtual fundamental class is that $\wPic$ has pure dimension. This is true, since $\Pic$ has pure dimension and $p$ is birational. 
\end{remark}
\begin{proposition}\label{push-forward-p} We have the following equality
\[   (\overline{p})_*[\tGw]^\vir=[\Gw]^{\vir}
    \]
\end{proposition}

\begin{proof}
 Note that by cohomology and base-change we have that $R^{\bullet}\widetilde\pi_*\overline{p}^*\ccL=\overline{p}^*R^{\bullet}\overline{\pi}_*\ccL$. As $p$ is birational and proper, we have $p_*[\widetilde{\Pic}_k]=[\Pic]$. We now apply Costello's Pushforward theorem (\cite{herr2022costello}) to \Cref{base change maps} and we get
 \[
 \overline{p}_*[\tGw]^\vir=[\Gw]^\vir. \qedhere
 \]
\end{proof}

\subsection{Definition of reduced GW invariants of hypersurfaces in all genera}\label{subsec: def reduced GW}
In this section we define reduced Gromov--Witten invariants for hypersurfaces in projective spaces under Assumption \ref{assum:d,big,2g,minus2}. This is less straight-forward than for projective spaces, since we have no understanding of the geometry of moduli spaces of maps to hypersurfaces.

Let $X$ be a smooth hypersurface on $\bbP^r$ defined by the vanishing of a regular section of $s$ of $\ccO(k)$. 
We have that $\Gwx$ is cut out in $\Gw$ by the vanishing of the section $\pi_*s$ of $\pi_*\ccL^{\otimes k}$ on $\Gw$. If $\pi_*\ccL^{\otimes k}$ is a vector bundle, we can use this to define the virtual fundamental class of $\Gwx$ by virtual pullback.
This generally fails for $g\geq 1$, so we will use the blow-ups we developed to ensure that the restriction of $\pi_*\ccL^{\otimes k}$ to the main component $\mGw$ is locally free.

\begin{construction}\label{def: reduced}
    Let $X\subset\bbP^r$ be a smooth hypersurface of degree $k$.
    Let $p_k\colon\wPic_k \to \Pic$ be any desingularization of $\ffF=R^1\pi_*(\ffL^{*}\otimes\omega_{\pi})$ and $R^0\pi_*\ffL^{\otimes k}$.
 
   We define the main component $\mGwx$ of $\Gwx$ as follows
    \begin{equation}
    \begin{tikzcd}
    \mGwx \arrow[d] \arrow[r] \arrow[dr, phantom,"\ulcorner", very near start] 
    & \ \arrow[d] \mGw \\
    \Gwx\arrow[r]  & \Gw\text{.}
    \end{tikzcd}
    \end{equation} 
    Notice that $\mGwx$ may not be irreducible, but we will still refer to it as the main component. The main component $\mGwx$ does not have a perfect obstruction theory, in the following we fix this problem.
    
    We define 
    \begin{equation}\label{main x}
    \begin{tikzcd}
    \tmGwx \arrow[d] \arrow[r] \arrow[dr, phantom,"\ulcorner", very near start] 
    & \ \arrow[d] \tmGw \\
    \tGwx\arrow[r]\arrow[d]  \arrow[dr, phantom,"\ulcorner", very near start]  & \tGw\arrow[d]\\
    \Gwx\arrow[r,"i"]  & \Gw,  \end{tikzcd}
    \end{equation}
    where $\tmGw$ and $\tGw$ are defined in \Cref{subsec: blowups stable maps Pr}. Since they are defined over any desingularization of $\ffF$, we can in particular replace $p:\wPic\to\Pic$ of \eqref{eq:blowup,pic,moduli,stabe,projective,space} by $p_k:\wPic_k\to \Pic$, which also desingularizes $R^0\pi_*\ffL^{\otimes k}$ for $k$ the degree of the hypersurface. Then \Cref{proper main} implies $\tmGwx$ is proper. 
    
       To sum up, we have the following diagram. Note that some of the squares are not Cartesian (See \Cref{rmk: main component is not pull back}).
       
\begin{equation}\label{eq:ci,tilde,moduli,cube}
\begin{tikzcd}
\tmGwx\arrow[rrr] \ar[rd] \ar[ddd, "\tilde{i}"]\arrow[ddr, phantom,"\ulcorner", near start ]&&& \mGwx\ar[dl] \ar[ddd] \arrow[ldd, phantom,"\urcorner", near start ]\\
&\tGwx \arrow[dr, phantom,"\ulcorner", very near start]\ar[r] \ar[d] &\Gwx \arrow[d,"i"]&\\
&\tGw\ar[r]\ar[dd] \arrow[ddr, phantom,"\ulcorner", very near start] &\Gw\ar[dd]&\\
\tmGw \arrow[crossing over]{rrr} \ar[ru]&&&\mGw \ar[lu]\\
&\wPic_k \arrow[r,"p_k"]  & \Pic\text{.}&
\end{tikzcd}
\end{equation}

We need the following result, whose proof we delay for reasons of exposition until after \Cref{desing dual}.   

     \begin{proposition}\label{cohbc}
     Let $\tilde{i}:\tmGwx\to \tmGw$ as in \eqref{eq:ci,tilde,moduli,cube} and let $\wpi:\widetilde{\ccC}\to \tGw$ be the universal curve. Then, $\wpi_*\widetilde{\ccL}^{\otimes k}$ is locally free on $\tmGw$.
     \end{proposition}
  
    We define
    \begin{equation}\label{eq:defi,red,class,X}
    [\tmGwx]^{\vir}=\tilde{i}^![\tmGw]
    \end{equation}
where $\tilde{i}:\tmGwx\to \tmGw$, as in \Cref{cohbc}.
    Note that for any $1\leq j\leq n$ we have morphisms 
    \[\tGw\to\Gw\stackrel{ev_j}{\longrightarrow}\bbP^r
    \]
    and 
    \[\tGwx\to\Gwx\stackrel{ev_j}{\longrightarrow}X.
    \]
    By abuse of notation we denote both of these compositions by $ev_j$.
\end{construction}

Notice that the definition of the reduced virtual fundamental class in \eqref{eq:defi,red,class,X} does a priori depend on the choice of a desingularization of $\Pic$. The following proposition shows that integration against this class does not depend on the desingularization. The proof of \Cref{invariance-red} is delayed, and we first prove \Cref{factoring morphisms}.

\begin{proposition}\label{invariance-red} Under Assumption \ref{assum:d,big,2g,minus2}, let $p':\wPic'\to\Pic$ and $p'':\wPic''\to\Pic$ be birational proper maps such that 

$((p')^*\ffF)^\tf$, $((p'')^*\ffF)^\tf$, $((p')^*(R^0\pi_*\ffL^{\otimes k}))^\tf$ and $((p'')^*(R^0\pi_*\ffL^{\otimes k}))^\tf$ are locally free. Consider $\tmGwx'$ and $\tmGwx''$ defined analogously to $\tmGwx$ above. Then we have 
\[
\int_{[\tmGwx']^{\vir}}\prod ev^*\gamma_i =\int_{[\tmGwx'']^{\vir}}\prod ev^*\gamma_i \]
\end{proposition}

\Cref{invariance-red} permits us to define the reduced Gromov--Witten invariants as they are independent of the blowing-up of $\Pic$.

\begin{definition}\label{def:reduced_GW_invariants}
    For $d>2g-2$, we call \emph{reduced} Gromov--Witten invariants of $X$, the following numbers 
    \[\int_{[\tmGwx]^{\vir}}\prod ev^*\gamma_i. \]
\end{definition}

In order to prove \Cref{invariance-red}, we first prove the following.

\begin{lemma} \label{factoring morphisms} 
    Consider a commutative diagram of Artin stacks
    \[\xymatrix{\widehat{\Pic}\ar[rr]\ar[dr]&&\wPic\ar[ld]\\
    &\Pic}\]
    and let $\widehat{\ccM}_{g,n}(\bbP^r,d)$ and $\tGw$  be the corresponding fiber products $\Gw\times_{\Pic}\widehat{\Pic}$, respectively $\Gw\times_{\Pic}\wPic$.
    \begin{enumerate}
        \item We have a diagram with Cartesian squares
        \[
    \begin{tikzcd}
    \widehat{\ccM}_{g,n}(\bbP^r,d)\ar[r]\ar[d]\arrow[dr, phantom,"\ulcorner", very near start]&\tGw\ar[r]\ar[d] \arrow[dr, phantom,"\ulcorner", very near start]&\Gw\ar[d]\\
    \widehat{\Pic}\ar[r]&\wPic\ar[r]&\Pic.
    \end{tikzcd}
    \]
    
    \item  Suppose that $\widehat{\Pic}$ and $\wPic$ are desingularizations of $\ffF$. Let $\widehat{\ccM}^{\circ}_{g,n}(\bbP,d)$ be the main component of $\widehat{\ccM}_{g,n}(\bbP^r,d)$ in the sense of \Cref{main comp blow up}. Under Assumption \ref{assum:d,big,2g,minus2}, the we have a commutative diagram 

    \[
    \begin{tikzcd}
    \widehat{\ccM}^{\circ}_{g,n}(\bbP,d)\ar[r]\ar[d]&\tmGw\ar[r]\ar[d]&\mGw\ar[d]\\
    \widehat{\Pic}\ar[r]&\wPic\ar[r]&\Pic.
    \end{tikzcd}
    \]
    \end{enumerate}
\end{lemma}

\begin{proof}
    The first statement follows from the fact that the square on the right and the big square are Cartesian. This shows that the square on the left is Cartesian.
    
    For the second statement, we apply point 1 and we consider the following extended diagram in which all squares are Cartesian 
    \[
    \begin{tikzcd}
    \widehat{\ccM}^{\circ}(\bbP^r)\ar[r]\ar[d]\arrow[dr, phantom,"\ulcorner", very near start]&\widetilde{\ccM}^{\circ}(\bbP^r)\ar[r]\ar[d]\arrow[dr, phantom,"\ulcorner", very near start]&\mGw\ar[d]\\
    \widehat{\ccM}_{g,n}(\bbP^r,d)\ar[r]\ar[d]\arrow[dr, phantom,"\ulcorner", very near start]&\tGw\ar[r]\ar[d] \arrow[dr, phantom,"\ulcorner", very near start]&\Gw\ar[d]\\
    \widehat{\Pic}\ar[r]&\wPic\ar[r]&\Pic.
    \end{tikzcd}
    \]
    By the definition of the main component of the moduli space of stable maps in \Cref{def:main moduli stable maps}, we have solid maps in the diagram
    \[
    \begin{tikzcd}
    \widehat{\ccM}^{\circ}_{g,n}(\bbP^r,d)\ar[r, dashed]\ar[d]\ar[rr, bend left]&\tmGw\ar[r]\ar[d]&\mGw\ar[d]\\
    \widehat{\ccM}^{\circ}(\bbP^r)\ar[r]&\widetilde{\ccM}^{\circ}(\bbP^r)\ar[r]&\mGw.
    \end{tikzcd}
    \]
    The dashed arrow is the identity on maps with smooth domain. Since the map $\widehat{\ccM}^{\circ}(\bbP^r)\to\widetilde{\ccM}^{\circ}(\bbP^r)$ is proper it maps closed substacks to closed substacks, and thus the identity map extends to a map \[\widehat{\ccM}_{g,n}(\bbP^r,d)\to\tmGw.\qedhere\]
\end{proof}

\begin{proof}[Proof of Proposition \ref{invariance-red}]
Let $\widehat{\Pic}$ denote the closure inside the fiber product $\wPic'\times_{\Pic}\wPic''$ of locus of smooth curves. We then have a commutative diagram
\[
\begin{tikzcd}
\widehat{\Pic}\ar[r]\ar[d]&\wPic'\ar[d,"p'"]\\
\wPic''\ar[r,"{p''}"]&\Pic.
\end{tikzcd}
\]
We define $\widehat{\ccM}_{g,n}(X,d)$ by the following Cartesian diagram
\[
\xymatrix
{\widehat{\ccM}_{g,n}(X,d)\ar[r]\ar[d]&\Gwx\ar[d]\\
\widehat{\Pic}\ar[r]&\Pic
}
\]
and similarly we define
\[
\begin{tikzcd}
\tGwx'\arrow[r]\ar[d]\arrow[dr, phantom,"\ulcorner", very near start]&\Gwx\ar[d]\\
\wPic'\arrow[r,"p'"]&\Pic
\end{tikzcd}
\begin{tikzcd}
\tGwx''\ar[r]\ar[d]\arrow[dr, phantom,"\ulcorner", very near start]&\Gwx\ar[d]\\
\wPic''\ar[r,"p''"]&\Pic
\end{tikzcd}
\]
Using the notation in \eqref{main x} and \Cref{factoring morphisms}, part 2, we obtain a commutative diagram

\begin{equation}
\begin{tikzcd}
    &\widehat{\ccM}_{g,n}^{\circ}(X,d)\ar[rd, "\widehat{p}''"]\ar[ld, "\widehat{p}'"']\\
\tmGwx'\ar[rd]&&\tmGwx''\ar[ld]\\
&\mGwx
\end{tikzcd}
\end{equation}

By \Cref{proper main} we have that $\widehat{p}'$ and $\widehat{p}''$ are proper. In the following we show that they are virtually birational.

Recall from \Cref{def: reduced} that we have diagrams with Cartesian squares on the left 
\begin{equation}\label{cartone}
\begin{tikzcd}
\widehat{\ccM}^{\circ}_{g,n}(X,d)\arrow[r]\ar[d,"\widehat{p}'"]\arrow[dr, phantom,"\ulcorner", very near start]&\widehat{\ccM}^{\circ}_{g,n}(\bbP^r,d)\arrow[r,"\widehat{\mu}'"]\arrow[d,"r'"]&\widehat{\Pic}\ar[d]\\
\tmGwx'\ar[r,"i'"]&\tmGw'\arrow[r,"\mu'"]&\wPic'
\end{tikzcd}
\end{equation}
and 
\begin{equation}\label{carttwo}
    \begin{tikzcd}
\widehat{\ccM}^{\circ}_{g,n}(X,d)\arrow[r]\ar[d,"\widehat{p}''"]\arrow[dr, phantom,"\ulcorner", very near start]&\widehat{\ccM}^{\circ}_{g,n}(\bbP^r,d)\arrow[r,"\widehat{\mu}'"]\arrow[d,"r''"]&\widehat{\Pic}\ar[d]\\
\tmGwx''\ar[r,"i''"]&\tmGw''\ar[r,"\mu''"]&\wPic''
\end{tikzcd}
\end{equation}

which give

 \begin{equation}\label{i pprime}
 [\widehat{\ccM}_{g,n}^{\circ}(X,d)]^{\vir}=(i')^![\widehat{\ccM}_{g,n}^{\circ}(\bbP^r,d)]\quad \text{and}\quad
[\widehat{\ccM}_{g,n}^{\circ}(X,d)]^{\vir}=(i'')^![\widehat{\ccM}_{g,n}^{\circ}(\bbP^r,d)]. 
\end{equation}
Since $\widehat{\ccM}_{g,n}^{\circ}(X,d)$ and $\tmGw'$ are irreducible and have an isomorphic open subset, we have that $r'$ and $r''$ are birational. By \Cref{proper main} they are also proper and thus we have 
\begin{align}\label{push f pP}
r'_*[\widehat{\ccM}_{g,n}^{\circ}(\bbP^r,d)]=[\tmGw']\quad  \text{ and}\quad
r''_*[\widehat{\ccM}_{g,n}^{\circ}(\bbP^r,d)]=[\tmGw''].
\end{align}
Using \eqref{i pprime}, \eqref{push f pP} and  commutativity of pullbacks with push-forwards in \cref{cartone} and \cref{carttwo}, we get
\begin{align}\label{roof}
\begin{split}
    \widehat{p}'_*[\widehat{\ccM}_{g,n}^{\circ}(X,d)]^{\vir}&=[\tmGwx']^{\vir}\quad \text{\ and}\\
\widehat{p}''_*[\widehat{\ccM}_{g,n}^{\circ}(X,d)]^{\vir}&=[\tmGwx'']^{\vir}.
\end{split}
\end{align}
Intersecting both equations above with $\prod ev^*\gamma_i$ we get the conclusion.
\end{proof}

 In the following we discuss several ways of blowing up $\Pic$ to desingularize $R^0\pi_*\ffL^{\otimes k}$.

 \begin{lemma}\label{lem: factors through tf}
     Given $\ccF$, $\ccG$ sheaves on an integral scheme $X$, with $\ccG$ torsion free and $f:\ccF\to\ccG$ a morphism, we have that $f$ factors through 
     \[\ccF\to \ccF^\tf\to \ccG.\]
 \end{lemma}
 \begin{proof}
     Since $X$ is integral and $\ccG$ is torsion free, we have that the composition ${\rm Tor}(\ccF)\to\ccF\to \ccG$ is zero. The claim now follows from the universal property of quotients.
 \end{proof}
 
 \begin{proposition}\label{desing dual}
   In notation of \Cref{subsec: blowups stable maps Pr}, we have the following:
        \begin{enumerate}
        \item An isomorphism
         \[\psi:(p^*R^0\pi_*\ffL)^\tf\to R^0\widetilde{\pi}_*\widetilde{\ffL}.\]
        \item $\Bl_{R^1\pi_*\ffL^{\otimes -k}\otimes \omega_{\pi}}\Pic$ is a desingularization of $R^0\pi_*\ffL^{\otimes k}$, for any integer $k>0$.
        \end{enumerate}

 \end{proposition}
 
 \begin{proof}
    1. As in \Cref{rem:resol l}, let $A$ a section of $\ffC\to\Pic$ such that $R^1\pi_*(\ffL(A))=0$. By abuse of notation we denote by $A$ the pull back of $A$ to $\widetilde{\ffC}$. We have exact sequences on $\wPic$ which fit into a commutative diagram 
         
         \begin{tikzcd}
              & p^*R^0\pi_*\ffL\ar[r]\ar[d,dashed]& p^*R^0\pi_*(\ffL(A))\ar[r]\ar[d]& p^* R^0\pi_*(\ffL(A)|_A)\ar[r]\ar[d]&  p^*R^1\pi_*\ffL\ar[r]\ar[d]& 0\\
           0\ar[r]& R^0\widetilde{\pi}_*\widetilde{\ffL}\ar[r] &R^0\widetilde{\pi}_*(\widetilde{\ffL}(A))\ar[r]& R^0\widetilde{\pi}_*(\widetilde{\ffL}(A)|_A)\ar[r]&   R^1\widetilde{\pi}_*\widetilde{\ffL}\ar[r]& 0,
          \end{tikzcd}
        where the vertical arrows are obtained by cohomology and base change and the solid arrows are isomorphisms.

       By \Cref{lem: factors through tf}, we have a morphism
        \[ \psi:(p^*R^0\pi_*\ffL)^\tf\to R^0\widetilde{\pi}_*\widetilde{\ffL}\]
        which sits in a commutative diagram
        \begin{equation}\label{eq:factorisation}
        \xymatrix{ p^*R^0\pi_*\ffL\ar[r]& (p^*R^0\pi_*\ffL)^\tf\ar[r]\ar[d]& p^*R^0\pi_*(\ffL(A))\ar[d]\\
        0\ar[r]& R^0\widetilde{\pi}_*\widetilde{\ffL}\ar[r] &R^0\widetilde{\pi}_*(\widetilde{\ffL}(A)).
        }
        \end{equation}
        
        Since the Image of the map $(p^*R^0\pi_*\ffL)^\tf\to p^*R^0\pi_*(\ffL(A))$ is equal to the kernel of $p^*R^0\pi_*(\ffL(A))\to p^* R^0\pi_*(\ffL(A)|_A)$, and the Image of $ R^0\widetilde{\pi}_*\widetilde{\ffL}\to R^0\widetilde{\pi}_*\widetilde{\ffL}(A)$, is equal to the kernel of $R^0\widetilde{\pi}_*(\widetilde{\ffL}(A))\to R^0\widetilde{\pi}_*(\widetilde{\ffL}(A)|_A)$, the isomorphisms in the diagram show that
        \begin{equation}\label{eq:same im}
            {\rm Im}\left((p^*R^0\pi_*\ffL)^\tf\to p^*R^0\pi_*(\ffL(A))\right)\simeq {\rm Im}\left(R^0\widetilde{\pi}_*\widetilde{\ffL}\to R^0\widetilde{\pi}_*(\widetilde{\ffL}(A))\right).
     \end{equation}

         Since $p^*R^0\pi_*\ffL\to p^*R^0\pi_*(\ffL(A))$ and $p^*R^0\pi_*\ffL\to R^0\widetilde{\pi}_*\widetilde{\ffL}$ are generically injective we have that $(p^*R^0\pi_*\ffL)^\tf\to p^*R^0\pi_*(\ffL(A))$ and $(p^*R^0\pi_*\ffL)^\tf\to R^0\widetilde{\pi}_*\widetilde{\ffL}$ are injective. This together with \eqref{eq:same im}  shows that $\psi$ is an isomorphism.
    \\

      2. Without loss of generality we assume that $k=1$. Let $p: \Bl_{R^1\pi_*\ffL^*\otimes\omega_{\widetilde{\pi}}}\Pic\to\Pic$ denote the projection. By cohomology and base change we have $p^*(R^1\pi_*\ffL^*\otimes\omega_{\widetilde{\pi}})\simeq R^1\widetilde{\pi}_*\widetilde{\ffL}^*\otimes\omega_{\widetilde{\pi}}$. With this, we have that $(R^1\widetilde{\pi}_*\widetilde{\ffL}^*\otimes\omega_{\widetilde{\pi}})^{\tf}$ is locally free.
   
    Since  \[\left((R^1\widetilde{\pi}_*\widetilde{\ffL})^\tf\right)^*\simeq (R^1\widetilde{\pi}_*\widetilde{\ffL})^*
     \]
     and $(R^1\widetilde{\pi}_*\widetilde{\ffL}^*\otimes\omega_{\widetilde{\pi}})^{\tf}$ is locally free, we get that $(R^1\widetilde{\pi}_*\widetilde{\ffL})^{*}$ is locally free. 

      By \Cref{dual-cone} we have that $R^0\widetilde{\pi}_*\widetilde{\ffL}\simeq (R^1\widetilde{\pi}_*\widetilde{\ffL}^*\otimes\omega_{\widetilde{\pi}})^*$. This together with the above shows that $R^0\widetilde{\pi}_*\widetilde{\ffL}$ is locally free. The claim now follows from the first part of the proposition. 
 \end{proof}

\begin{proof}[Proof of \Cref{cohbc}] By construction we have that  $(p_k^*\pi_*\ffL^{\otimes k})^\tf$ is locally free on $\wPic$ and by \Cref{desing dual} we have
     \[(p_k^*\pi_*\ffL^{\otimes k})^\tf\simeq  \widetilde{\pi}_*\widetilde{\ffL}^{\otimes k}.\]
     This gives that $\widetilde{\pi}_*\widetilde{\ffL}^{\otimes k}$ is locally free on $\wPic$. Since $\mu: \tmGw\to\wPic$ is smooth, cohomology commutes with base change and thus
     \[
     \mu^*\widetilde{\pi}_*\widetilde{\ffL}^{\otimes k}\simeq \wpi_*\mu^*\widetilde{\ffL}^{\otimes k}.\]
    Since $\mu^*\widetilde{\ffL}\simeq \ccL$ this shows that $\wpi_*\widetilde{\ccL}^{\otimes k}$ is locally free on $\tmGw$.
  \end{proof}  
 
 \begin{proposition}\label{prop: no k}
        In notation 
        of \Cref{subsec: blowups stable maps Pr}, we have that $\Bl_{\ffF}\Pic$ is a desingularization of $R^0\pi_*\ffL^{\otimes k}$.
        \end{proposition} 
  \begin{proof}
  For $k=1$, the statement holds by the proposition above.

In the following we show that $\Bl_{\ffF}\Pic$ is a desingularization of $R^0\pi_*\ffL^{\otimes k}$, for any $k>0$.

Locally we choose $B$ a section such that $\ffL^{\otimes k}\simeq \ffL(B)$. Taking $A$ such that $R^1\pi_*\ffL(A)=0$, we have a diagram
  \[\xymatrix{
 0\ar[r]&R^0\pi_*\ffL\ar[r]^{\cdot A}\ar[d]&R^0\pi_*\ffL(A)\ar[d]^{\cdot B}\\
 0\ar[r]&R^0\pi_*\ffL(B)\ar[r]&R^0\pi_*\ffL(A+B).
 }
 \]
Let $U$ be the subset of $\Pic$, where $R^1\pi_*\ffL=0$. Since we work under the assumption \ref{assum:d,big,2g,minus2}, $U$ is a non-empty open subset. Then on $U$ we have the following exact sequences
\begin{align}\label{split}\notag 0\to R^0\pi_*\ffL\to &R^0\pi_*\ffL(B)\to R^0\pi_*\ffL(B)|_B\to 0\\
0\to R^0\pi_*\ffL(A)\to &R^0\pi_*\ffL(A+B)\to R^0\pi_*\ffL(A+B)|_B\to 0.
\end{align}
By possibly shrinking $U$, the section $A$ may be chosen to avoid $B$. With this, we have that multiplication with $A$ induces an isomorphism
\[R^0\pi_*\ffL(B)|_B\simeq R^0\pi_*\ffL(A+B)|_B.\]
By possibly shrinking $U$, we may assume that $R^0\pi_*\ffL(A)$ and $R^0\pi_*\ffL(A+B)$ are trivial, and that the sequences in \eqref{split} are split. The claim now follows from \Cref{cor: one implies k}.

    \end{proof}
 In genus one, following \cite{Vakil-Zinger-desingu-main-compo-,Hu-Li-GEnus-one-local-VZ-Math-ann-2010}, one can define reduced Gromov-Witten invariants of degree-$k$ hypersurfaces on $\Bl_{\ffF}\Pic$. Below we give a direct proof of this fact. The proof below does not generalise to higher genus.

\begin{proposition}\label{lem: g1 k1 desingularizes k}

    Let $g=1$. Then for every  $k\geq 1$, the sheaf $\widetilde{\pi}_{*}\ev^*\ccO(k)$ is locally free on the main component $\tGwmone$ of $\tGwone$. 

    \end{proposition}

\begin{proof}
    Fix $k\geq 1$. By \Cref{bc}, we need to show that $R^0\widetilde{\pi}_*(\ffL^{\otimes k})$ is locally free over the image $Z^\circ$ of $\tGwone$ in $\wPic$ via the forgetful morphism.

  In a neighbourhood of $(\widetilde{C},\widetilde{L})\in \wPic_1$ we can choose a section $A$ of $\ffL^{\otimes k-1}$. This gives an exact sequence
   \begin{equation}\label{g zero splitting}
      0\to R^0\widetilde{\pi}_*\widetilde{\ffL}\stackrel{\cdot A}{\longrightarrow} R^0\widetilde{\pi}_*(\widetilde{\ffL}^{\otimes k})\to R^0\widetilde{\pi}_*(\widetilde{\ffL}^{\otimes k}|_A)\to  R^1\widetilde{\pi}_*\widetilde{\ffL}\stackrel{\cdot A}{\rightarrow} R^1\widetilde{\pi}_*(\widetilde{\ffL}^{\otimes k})\to 0.
  \end{equation} 
  If $\tilde{L}$ is non-negative on each component of $\tilde{C}$, as is the case in $Z^\circ$, we have the equality
  \[
    h^1(\widetilde{C},\widetilde{L}) = h^1(\widetilde{C},\widetilde{L}^{\otimes k})
  \]
  by a Riemann-Roch computation. Since cohomology in degree 1 commutes with base change and the map $H^1(\widetilde{C},\widetilde{L})\to H^1(\widetilde{C},\widetilde{L}^{\otimes k})$ is an isomorphism, we have that the last arrow in sequence \eqref{g zero splitting} is an isomorphism. This shows that we have a short exact sequence
  
  \[
   0\to R^0\widetilde{\pi}_*\widetilde{\ffL}\to R^0\widetilde{\pi}_*(\widetilde{\ffL}^{\otimes k})\to R^0\widetilde{\pi}_*(\widetilde{\ffL}^{\otimes k}|_A)\to 0.
  \]
    Since $R^0\widetilde{\pi}_*(\ffL^{\otimes k}|_A)$ is locally free and by \Cref{desing dual} the sheaf $R^0\widetilde{\pi}_*\widetilde{\ffL}\simeq (R^1\widetilde{\pi}_*\widetilde{\ffL}\otimes\omega_{\widetilde{\pi}})^*$ is also locally free on $Z^\circ$, we get that $R^0\widetilde{\pi}_*(\widetilde{\ffL}^{\otimes k})$ is locally free.
\end{proof}
 
\begin{remark} \label{rem: various blowups}
Above we denoted by $\wPic$ any desingularization of $\ffF$.  We collect here various blow-ups of interest.
\begin{enumerate}
    \item $\Bl_{\ffF}\Pic$ in the sense of Rossi (see \Cref{subsec:constr_bl_stacks})
    \item $\Bld_{\ffF}\Pic$ in the sense of Hu--Li.
  
\end{enumerate}

We have 
\[\Bld_{\ffF}\Pic\to \Bl_{\ffF}\Pic
\quad \text{and}\quad 
\Bld_{R^1\pi_*\ffL^{\otimes k}}\Bld_{\ffF}\Pic\to \Bl_{R^1\pi_*\ffL^{\otimes k}}\Bl_{\ffF}\Pic.
\]
By \Cref{prop: no k} and by \cref{desing dual} we have 
\[
\Bl_{\ffF}\Pic\to \Bl_{R^1\pi_*\ffL^{\otimes k}}\Pic \to Bl_{R^0\pi_*\ffL^{\otimes k}}\Pic.
\]
\end{remark}

\subsection{Reduced invariants from stable maps with fields}\label{subsec:maps_with_fields}
Reduced invariants are conjecturally related to Gromov--Witten invariants \cite{Zinger-reduced-g1-CY}, \cite[Conjecture 1.1]{Hu-Li-diagonalization}. One of the main difficulties in proving such conjectures is that one needs to understand how to split the virtual fundamental class of a moduli space of stable maps among its irreducible components. What makes this task particularly difficult is that almost nothing is known about the geometry of this moduli space of stable maps.

In genus one and two the existing algebraic proofs \cite{Chang-Li-hyperplane-property, Lee-Li-Oh-quantum-lefschetz, Lee-Oh-reduced-complete-intersections, Lee-Oh-reduced-complete-intersections-2} use an additional well-behaved moduli space of maps with fields \cite{Chang-Li-maps-with-fields}. In view of these conjectures, we discuss blow-ups of maps with fields and we show that the ingredients for the low genus proof are also available in higher genus.

Given $X$ a hypersurface (or more generally a complete intersection) in $\bbP^r$, the associated moduli space of maps with fields has the following features:
    
\begin{enumerate}
    \item in genus one and two it has well-understood geometry, such as local equations and irreducible components (see \cite{Hu-Li-GEnus-one-local-VZ-Math-ann-2010,Hu-Li-genus-two});
    
    \item it has a virtual fundamental class and a localised virtual fundamental class (see \cite{kiem2013localizing}, Definition 3.3); the latter is needed because the space of maps with fields is \emph{not compact};
    
    \item the localised virtual fundamental class is supported on $\Gwx$ and it coincides up to a sign to the virtual fundamental class of $\Gwx$ (see \cite{kiem2013localizing}, Theorem 1.1).
\end{enumerate}
These properties allow us to work with the well-understood moduli space of stable maps with fields instead of $\Gwx$, whose geometry is unavailable. 

In the following we use the Hu--Li blow up the moduli space of maps with fields to define reduced invariants in this context. The main theorem of this section, \Cref{th-blow-up-maps}, shows that the blown-up moduli space of maps with fields has property (1) listed above, and even more, its irreducible components are smooth over their image in $\wPic$. This is a feature of the Hu--Li construction, which does not hold for the Rossi construction. Properties (2) and (3) are automatically satisfied. In future work we will also investigate the (intrinsic) normal cone of the space of maps with fields.

\subsubsection{Review of maps with fields}

We recall the construction and properties of the moduli space of maps with $p$-fields. This space was introduced in \cite{Chang-Li-maps-with-fields} to study higher genus Gromov-Witten invariants of the quintic threefold. \\

In the following we fix $k\in\bbZ$, $k>1$, and we consider the sheaf 
\[\pi_*(\ffL^{\oplus r+1} \oplus (\ffL^{\otimes -k}\otimes \omega_{\ffC/\Pic}))\] 
on $\Pic$ and its corresponding abelian cone
\[
    S\left(\pi_*(\ffL^{\oplus r+1}\oplus (\ffL^{^{\otimes -k}}\otimes \omega_{\ffC/\Pic})\right)=\Spec\Symm R^1\pi_*\left(\left(\ffL^*\otimes \omega_{\ffC/\Pic}\right)^{\oplus r+1}\oplus  \ffL^{\otimes k}\right)\stackrel{\mu^p}{\rightarrow}\Pic .
\]
Recall that, in our notation, we have already imposed on $\Pic= \Pic^{\rm st}_{g,n,d}$ the stability condition in \Cref{eq: stability condition}. Then, the \textit{moduli space of maps with $p$-fields} is defined in \cite[Section 3.1]{Chang-Li-maps-with-fields} as an open in 
the abelian cone
\[
    \Gw^p = S\left(\pi_*(\ffL^{\oplus r+1}\oplus (\ffL^{^{\otimes -k}}\otimes \omega_{\ffC/\Pic})\right)
\]
Therefore, an element of $\Gw^p$ over $(C,L,\underline{s})\in \Gw$ is given by a choice of a section $p\in H^0(C,L^{\otimes -k}\otimes\omega_C)$. 

Consider the Cartesian diagram
\begin{equation}
\begin{tikzcd}
\ccC^p \arrow[d, "\opi^p"] \arrow[r,"\nu^p"] \arrow[dr, phantom,"\ulcorner", very near start] 
& \ \arrow[d,"\opi"] \ccC \\
\Gw^p \arrow[r]  & \Gw\text{.}
\end{tikzcd}
\end{equation}
The complex
\begin{equation}\label{eq:pot_for_p_fields}
\mathbb{E}_{\Gw^p/\Pic}:=R^{\bullet}\opi^p_*(\oplus_{i=0}^r\ccL\oplus \ccL^{\otimes-k}\otimes\omega_{\opi^p})
\end{equation}
is a dual obstruction theory for the morphism $\mu^p$. 
The stack $\Gw^p$ is not proper, but the perfect obstruction theory admits a cosection $\sigma$, that is, a morphism 
\[\sigma \colon h^1(\bbE^*) \to \ccO_{\Gw^p}.
\]
Note that since $\Pic$ s smooth we have that the absolute obstruction of $\Gw^p$ is isomorphic to $h^1(\bbE^*)$ and the cosection lifts (see \cite[Proposition 3.5]{Chang-Li-maps-with-fields}) . This data gives a cosection localised virtual fundamental class 
 $[\Gw^p]^{\vir}_{\sigma}$.

For $X$ a smooth subvariety cut out by any regular section of $\ccO_{\bbP^r}(k)$, \cite[Theorem 1.1]{Chang_Li_invariants} states that

\begin{equation}\label{fields to hypersurface}
 [\Gw^p]^{\vir}_{\sigma}=(-1)^{(r+1)d+1-g} [\Gwx]^{\vir}\in A_{d^{\vir}(\Gwx)}(\Gwx)
\end{equation}
The particular case where $r=4$ and $k=5$ (therefore $X$ is a quintic threefold) was the motivation for introducing $p$-fields in the first place. In \cite[Theorem 1.1]{Chang-Li-maps-with-fields}, the authors proved \Cref{fields to hypersurface} at the level of invariants before it was upgraded to classes in  \cite{Chang_Li_invariants}.

\subsubsection{Blow-ups of maps with fields}\label{sec: blow up maps with fields}
    In this section we discuss a non-minimal blow-up which has good properties. 
    
    In notation as before we consider
     \begin{equation}\label{eq:def,E_k}
    \ffE_k:=R^1\pi_*\ffL^{\otimes k}.
\end{equation}

We define   
\begin{equation}\label{eq:defi,Pic_k}
\wPic_k:=\Bld_{\ffE_k}\Bld_{\ffF}\Pic.
\end{equation}

 Similarly to \Cref{base change maps}, we define $\tGw^p$ as the following Cartesian diagram:
\begin{equation}\label{bl-fields}
\begin{tikzcd}
\tGw^p \arrow[d, "\widetilde{\mu}^p"] \arrow[r, "\overline{p}_k"] \arrow[dr, phantom,"\ulcorner", very near start] 
& \ \arrow[d, "\mu^p"] \Gw^p \\
\wPic_k \arrow[r,"p_k"]  & \Pic\text{.}
\end{tikzcd}
\end{equation}
Since $\tGw^p$ is an open substack of $ \Spec\Symm p_k^*( \ffE_k\oplus\ffF^{\oplus r+1})$, we have a perfect obstruction theory for $\widetilde{\mu}^p$ exactly analogous to that in \eqref{eq:pot_for_p_fields}.

Indeed, let $\wpi:\widetilde{C}^p\to\tGw^p$ be the base-change of the universal curve of $\Gw^p$, and $\widetilde{\ccL}$ be pullback of the line bundle $\ccL$ on $\ccC^p$. The perfect obstruction theory on $\tGw^p$ is given by 
\[
\bbE_{\tGw^p/\wPic_k}=\overline{p}_k^*\mathbb{E}_{\Gw^p/\Pic}=R^{\bullet}\wpi_*(\widetilde{\ccL}^{\oplus r+1}\oplus\widetilde{\ccL}^{\otimes -k}\otimes\omega_{\wpi})
\]
The cosection $\sigma$ induces
\[
\widetilde{\sigma}=\overline{p}_k^*\sigma: h^1(\bbE_{\tGw^p/\wPic_k})\to \ccO_{\tGw^p}.
\]
One then obtains a localized virtual fundamental class as in \cite{Chang-Li-maps-with-fields}, with the additional feature that our $\wPic$ may be singular. The argument there can be adapted to our case with minimal changes. We sketch the argument below, while pointing out the differences.

We do not have an absolute obstruction theory, but only one relative to $\wPic$, so we do not have a cosection of the absolute obstruction theory. The (absolute) cosection is only needed to prove that the intrinsic normal cone is contained in the kernel of the cosection. Instead of lifting the cosection as in \cite{Chang-Li-maps-with-fields}, we work relatively to $\wPic$ and we show that normal cone of $\widetilde{\mu}^p$ is contained in the kernel cone $E(\widetilde{\sigma})$. 

For $\Gw^p$, the fact that the normal cone of $\mu^p$ is contained in $E(\sigma)$ is guaranteed by \cite[Proposition 3.5]{Chang-Li-maps-with-fields}.  Using the diagram in \Cref{bl-fields} and \cite[Remark 3.5]{manolache2011virtual} we have that \[\ffC_{\widetilde{\mu}^p}\hookrightarrow \overline{p}_k^*\ffC_{\mu^p}.
\]
As in \cite[Corollary 4.5]{kiem2013localizing}, it is enough to show that $\ffC_{\widetilde{\mu}^p}\hookrightarrow \overline{p}_k^* E(\sigma)$ on the subset where $\widetilde{\sigma}$ is surjective. With this, we have $E(\widetilde{\sigma})=\overline{p}_k^* E(\sigma)$. This shows that \[\ffC_{\widetilde{\mu}^p}\hookrightarrow E(\widetilde{\sigma}),
\]
as required. The rest of the construction follows as in the classical case. Intersecting the normal cone of $\widetilde{\mu}^p$ with the zero section of $E(\widetilde{\sigma})$, we obtain localised virtual fundamental class $[\tGw^p]_{\widetilde{\sigma}}^{\vir}$ supported on $\Gwx\times_{\Pic}\wPic_k$. See \cite[Section~3]{Chang-Li-maps-with-fields} for details.

Since $\Gw\hookrightarrow \Gw^p$ by zero section, we have $\tmGw\subset \tGw^p$. Under the assumption in \ref{assum:d,big,2g,minus2},
we have that $R^1\pi_*\ccL^{\otimes k}=0$. We define the main component of $\tGw^p$ to be $\tmGw\times_{\tGw}\tGw^p$.
Under assumption \ref{assum:d,big,2g,minus2}, the fibers of the cone \[\tmGw\times_{\tGw}\tGw^p\to \tmGw\] are trivial and $\tmGw$ is indeed a component of $\tGw^p$; in fact, it agrees with the main component of $\tGw^p$ as an abelian cone. Even though geometrically the main components of $\tGw$ and $\tGw^p$ agree, the second one carries a non-trivial virtual structure in the form of an obstruction sheaf which is the restriction of $R^1\wpi_*\widetilde{\ccL}^{\otimes - k}\otimes\omega_{\wpi}$ as well as a cosection. 

\begin{theorem}\label{th-blow-up-maps} 
Denote by $(\tGw^{p,\lambda})_{\lambda \in \Lambda}$ the irreducible components of $\tGw^p$ and $(\tGw^{\theta})_{\theta \in \Theta}$ the irreducible components of $\tGw$. 

  Let
   \begin{align*}\wpi^{p,\lambda}:\widetilde{\ccC}^\lambda&\to \tGw^{p,\lambda}\\
   \tpi^\theta:\widetilde{\ccC}^{\theta}&\to \tGw^{\theta}
   \end{align*} the pull-backs of the universal curve on $\Gw^p$.
   The following statements hold.
    \begin{enumerate}
        \item The morphism $\bar{p}_k$ of \eqref{bl-fields} is proper. 
        \item The irreducible components $\tGw^{p,\lambda}$ and $\tGw^{\theta}$ are smooth over their image in $\wPic_k$. In particular, $\tmGw$ is smooth over $\widetilde{\Pic}_k$.
        \item The sheaf $\wpi^{p,\lambda}_*ev^*\ccO(k)$ is a locally free sheaf on $\tGw^{p,\lambda}$, the sheaf $\tpi^{\theta}_*ev^*\ccO(k)$ is a locally free sheaf on $\tGw^{\theta}$. In particular, $\tpi^{\circ}_*ev^*\ccO(k)$ is a locally free sheaf on $\tmGw$.
    \end{enumerate}
\end{theorem}

\begin{proof}
1. We have that $\overline{p}_k$ is proper, as $p_k$ is proper. 

2. Consider the following diagram
\begin{equation}
\begin{tikzcd}
\Spec\Symm p_k^*(\ffF^{\oplus r+1}\oplus\ffE_k) \arrow[d] \arrow[r] \arrow[dr, phantom,"\ulcorner", very near start] 
& \ \arrow[d] \Spec\Symm\ffF^{\oplus r+1}\oplus\ffE_k \\
\wPic_k \arrow[r,"p_k"]  & \Pic\text{.}
\end{tikzcd}
\end{equation}
We have that $\tGw^p$ is an open substack of $ \Spec\Symm p_k^*( \ffE_k\oplus\ffF^{\oplus r+1})$ and by \Cref{theorem:cone_as_a_union} we have that the irreducible components of the stacks \[\Spec\Symm\ \ffF^{\oplus r+1}\oplus\ffE_k \text{\ \ \ and \ \ \ } \Spec\Symm\ffF^{\oplus r+1}
\]
are smooth over their image in $\wPic_k$. 

This shows that $\tGw^{p,\lambda}$ and $\tGw^{\theta}$ are smooth over their image in $\wPic_k$. In particular
\[
\tmGw=\Spec\Symm ( p_k^*\ffF^{\oplus r+1})^\tf \mbox { is smooth over } \wPic_k.
\]

3. Let $\mu^\lambda: \Gw^{p,\lambda}\to\Pic$ be the restriction of $\mu^p$. Let $Z^\lambda$ be the image of $\mu^\lambda$. Let $\pi^\lambda:\ffC^\lambda\to Z^\lambda$ be the restriction of $\pi$. Let $\widetilde{Z}^\lambda$ be the fiber product
\begin{equation}
\begin{tikzcd}
\widetilde{Z}^\lambda \arrow[d] \arrow[r,"p_k^\lambda"] \arrow[dr, phantom,"\ulcorner", very near start] 
& \ \arrow[d] Z^\lambda \\
\wPic_k \arrow[r,"p_k"]  & \Pic\text{.}
\end{tikzcd}
\end{equation}
Let $\tpi^\lambda:\widetilde{\ffC}^\lambda\to Z^\lambda$ be the restriction of 
$\tpi$ and let $q_k^\lambda:\widetilde{\ffC}^\lambda\to \ffC^\lambda$ be the restriction of $q_k$. By commutativity of proper push-forwards with base-change we have that 
\[(p_k^\lambda)^*R^{\bullet}\pi^\lambda_*\ffL\simeq R^{\bullet}\tpi^\lambda_*(q_k^\lambda)^*\ffL\]
Again, cohomology and base-change in the Cartesian diagram
\begin{equation}
\begin{tikzcd}
\widetilde{\ccC}^\lambda \arrow[d, "\wpi^\lambda"] \arrow[r,"\nu^\lambda"] \arrow[dr, phantom,"\ulcorner", very near start] 
& \ \arrow[d,"\tpi^\lambda"] \widetilde{\ffC}^\lambda \\
\tGw^\lambda \arrow[r,"\mu^\lambda"]  & \widetilde{Z}^\lambda\text{.}
\end{tikzcd}
\end{equation}
gives
\begin{equation}\label{bc}
    R^\bullet\wpi^{\lambda}_*ev^*\ccO(k)=(\mu^\lambda)^*R^\bullet\tpi^{\lambda}_*\ffL^{\otimes k}.
\end{equation}
We have a short exact sequence 
\[0\to R^0\tpi^{\lambda}_*\ffL^k\to E^0\stackrel{\phi}{\rightarrow} E^1\to R^1\tpi^{\lambda}_*\ffL^{\otimes k} \to 0.\]
By construction we have that $\phi$ is locally diagonal. By \Cref{bc} we have that 
\[ R^\bullet\wpi^{\lambda}_*ev^*\ccO(k)\simeq [(\mu^\lambda)^*E^0\stackrel{(\mu^\lambda)^*\phi}{\longrightarrow}(\mu^\lambda)^*E^1].\]
Since $(\mu^\lambda)^*\phi$ is locally diagonal, \Cref{ker-diag} implies that $R^0\wpi^{\lambda}_*ev^*\ccO(k)$ is locally free.

A similar argument shows that $\tpi^{\theta}_*ev^*\ccO(k)$ is a locally free sheaf on $\tGw^{\theta}$ and in particular $\tpi^\circ_*ev^*\ccO(k)$ is a locally free sheaf on $\tmGw$.
\end{proof}

\begin{proposition} \label{invariance loc inv}The localised invariants do not depend on the blow-up of $\Pic$, more precisely,
    \[(\overline{p}_k)_*[\tGw^p]_{\tilde{\sigma}}^{\vir}=[\Gw^p]_{\sigma}^{\vir}.\]
\end{proposition}

\begin{proof}
    The arguments used in \cite{Chang-Li-maps-with-fields} and its various generalizations to establish \Cref{fields to hypersurface} for the cosection localized virtual fundamental class on the right hand-side can be adapted in a straight-forward manner to establish an analogous result for the class on the left-hand side. That is,
        \begin{equation}\label{tilde filds to hypersurface}
        [\tGw^p]_{\widetilde{\sigma}}^{\vir}=(- 1)^{(r+1)d+1-g}[\tGwx]^{\vir}.
    \end{equation}
    The only issue with adapting the classical proof comes from the fact that the cosection $\widetilde{\sigma}$ does not lift to the absolute obstruction theory. The discussion a the beginning of \Cref{sec: blow up maps with fields} shows how to get around this requirement, since in our set-up all the stacks, obstruction theories and cosections are induced by base-changing the usual ones over the morphism $p_k:\wPic_k\to\Pic$.
    Let
    \[
    \tGwx=\Gwx\times_{\Pic}\wPic_k
    \]
    with $pr_1:\tGwx\to \Gwx$. Now $\tGwx$ can be given a perfect obstruction theory by pulling back that of $\Gwx$, and similarly to \Cref{push-forward-p} we have
    \begin{equation}\label{eq:pushforward_tgwx}
    (pr_1)_*[\tGwx]^{\vir}=[\Gwx]^{\vir}.
    \end{equation}
    Then the result follows from \Cref{eq:pushforward_tgwx}, \Cref{fields to hypersurface} and \Cref{tilde filds to hypersurface}.
\end{proof}

\subsubsection{Reduced invariants from maps with fields}
Let $\ffC_{\tGw^p/\widetilde{\Pic}_k}$ denote the relative intrinsic normal cone of the morphism $\widetilde{\mu}^p:\tGw^p\to\wPic_k$ \cite[Section 7]{Behrend-Fantechi}. We have $\ffC_{\tGw^p/\widetilde{\Pic}}=\cup_i \ffC_i$, where $\ffC_i$ are its (finitely many) irreducible components. By the discussion above, the main component of $\tGw^p$ is $\tmGw$, which is smooth over $\wPic_k$. So there is only one component of the intrinsic normal cone which is supported over an open of $\tmGw$, which we denote by $\ffC_0$.
The cosection-localized virtual fundamental class of $\tGw^p$ is defined as
\[
[\tGw^p]_{\widetilde{\sigma}}^{\vir}=0_{\widetilde{\sigma}}^{!}[\ffC_{\tGw^p/\widetilde{\Pic}_k}]
\]
where $0_{\widetilde{\sigma}}^{!}$ denotes the cosection-localized virtual Gysin pullback.
Then, 
\begin{equation}\label{pfield_decomp}
    \tGw^p]_{\widetilde{\sigma}}^{\vir}=\sum_i [\tGw^p]_i^{\vir}
\end{equation}for $[\tGw^p]_i^{\vir}$ the class corresponding to the component $\ffC_i$. In particular, $[\tGw^p]_0^{\vir}$ is supported over $\tmGwx$.

\begin{proposition}\label{red inv fields}
   Denote by $[\tGw^p]_0^{\vir}$ the cosection-localized virtual fundamental class corresponding to the cone $\ffC_0$.
   We have  
   \[[\tmGwx]^{\vir}=(-1)^{kd-g+1}[\tGw^p]_0^{\vir}\in A_*(\tmGwx).\]
   
\end{proposition}

\begin{proof}
    This follows the lines of proof of Corollary 4.4 in \cite{Chang-Li-hyperplane-property}. Let 
    \[
\mathbb{E}_1^{\bullet}:=R^{\bullet}\wpi^p_*(\oplus_{i=0}^r\ccL),\hspace{0.5cm}
\mathbb{E}_2^{\bullet}:=R^{\bullet}\wpi^p_*( \ccL^{\otimes-k}\otimes\omega_{\wpi^p})
\]
and let $\ccE_i=h^1/h^0(\mathbb{E}_i^{\bullet})$ and $\ccE=h^1/h^0(\mathbb{E}^{\bullet})$. We have that $\ccE_i$ is a vector bundle stack on $\tGw^p$ and $\ccE\simeq \ccE_1\oplus \ccE_2$. Let $\ccU$ be the open subset of the main component of $\tGw^p$, with consists of maps with fields with irreducible source. On $\ccU$ we have $R^1\wpi_*f^*\ccO(k)=0$, and thus $\ccU$ is also an open subset of $\tmGw$. Using that $\ccU$ is smooth and unobstructed, we see that $\ffC_{\ccU/\wPic}$ is isomorphic to the vector bundle stack $\ccE_1|_\ccU$. Since the embedding $\ffC_{\ccU/\wPic}\hookrightarrow h^1/h^0(\ccE|_\ccU)$ is
\[(\ccE_1\oplus 0)|_\ccU\hookrightarrow (\ccE_1\oplus \ccE_2)|_\ccU
\]
and $\ffC_{\tGw^p/\wPic}\hookrightarrow\ccE$ is a closed embedding, we get that $\ffC_0\simeq \ccE_1$. By the definition of the localised cosection virtual fundamental class, we get 
\[[\tGw^p]^{\vir}_0=0_{\widetilde{\sigma}}^{!}[\ffC_0]=0^![0_{\ccE_2}],\]
where $0_{\ccE_2}$ is the zero section of $\ccE_2|_{\tmGw}$. By Lemma 4.3 in \cite{Chang-Li-hyperplane-property} with the complex $R^\bullet\wpi^{p,\circ}_* \ccL^{\otimes k}$ and \Cref{th-blow-up-maps}, part 3 we get
\[[\tGw^p]^{\vir}_0=c_{\rm top}(R^1\wpi^{p,\circ}_*( \ccL^{\otimes-k}\otimes\omega_{\wpi^{p,\circ}}))\cdot [\tmGw].
\]
By Serre duality we have that 
\[c_{1-g+kd}(R^1\wpi^{p,\circ}_*( \ccL^{\otimes-k}\otimes\omega_{\wpi^p}))=(-1)^{1-g+kd}c_{1-g+kd}(R^0\wpi^p_* \ccL^{\otimes k})
\]
and thus \[[\tGw^p]^{\vir}_0=(-1)^{1-g+kd}c_{1-g+kd}(R^0\wpi^p_* \ccL^{\otimes k})\cdot [\tmGw].
\]
This proves the claim.
\end{proof}

\begin{conjecture}Let $X$ be a threefold which is a complete intersection in projective space. Then 
\[\deg[\widetilde{\ccM}_{g,n}(\bbP^r,d)^p]_i^{\vir}=c_i\deg[\tGwi]^{\vir},\]
for some $c_i\in\bbQ$ and $g_i<g$.
\end{conjecture}
\begin{remark}
    The conjecture has been proved for genus one \cite{Zinger-reduced-g1-CY,Zinger-reduced-g1-GW,Zinger-standard-vs-reduced}, \cite{Chang-Li-hyperplane-property}, \cite{Lee-Oh-reduced-complete-intersections,Lee-Oh-reduced-complete-intersections-2} and genus two \cite{Lee-Li-Oh-quantum-lefschetz}. 

    In genus $g=1$ and $X$ a Calabi--Yau threefold, the conjecture translates into
    \[\deg[\overline{\ccM}_{1,n}(X,d)]=\deg[\widetilde{\ccM}^\circ_{1,n}(X,d)]^\vir+\frac{1}{12}\deg[\overline{\ccM}_{0,n}(X,d)]^\vir.\]
\end{remark}

\section{Generalizations}\label{sec:generalization}

The constructions above also work for quasi-maps. We first summarise some of the results in \cite{C-FK,C-FKM} and then extend the definition of reduced invariants to quasi-maps and to more general targets.

\subsection{Stable quasi-maps to GIT quotients of vector spaces}\label{recalling quasimaps}

Let $C$ be a scheme. 
 Let $V$ be a vector space and $G$ a reductive algebraic group acting on $V$. For $P$ a principal $G$-bundle on $C$ we define the associated $V$-bundle a
 \[
 P\times_GV:= (P\times V)/G.
 \]
 This is a bundle with fiber $V$ over $C$.
 The quotient is taken by considering the action $g\cdot(p,v)=(p\cdot g^{-1},g\cdot v)$.

A morphism from $C$ to the stack quotient $[V/G]$ corresponds to a $G$-equivariant morphism $P\to V$ from some $G$-torsor $P$ over $C$ to $V$. That is, maps $C\to[V/G]$ correspond to sections of the vector bundle $P\times_G V\to C$.

Let $\chi(G)$ be the character group of $G$ and $\theta\in \chi(G)$ a fixed character. The character $\theta: G \to \bbC^*$ determines a one-dimensional representation $\bbC_{\theta}$ of $G$, hence a linearization of the trivial line bundle on $V$, which we denote by $L_{\theta}$. This is used to construct a GIT quotient
\[ X=V\GIT_{\theta}G
\]
 which we assume to be proper. The GIT quotient comes with a polarization \[V^s\times_G(L_\theta\mid_{V^s})=:\ccO(\theta)\to V\GIT_{\theta}G.\] 
Multiples of the chosen character define the same underlying GIT quotients with a multiple of the polarization. 

Let
\[V^s = V^s(\theta)\quad \text{and}\quad V^{ss}=V^{ss}(\theta)
\]
be the open subsets of stable (respectively, semistable) points determined by $L_{\theta}$.
We also assume:
\begin{enumerate}
\item $\emptyset\neq V^s=V^{ss}$ and
\item G acts freely on $V^s$.
\end{enumerate}

Recall that for $(C,P,u)$ a curve with a $G$-torsor $P$ and a section $u$ of the associated bundle $P\times_{G}V$, we say that $(P,u)$ has ``class'' $\beta\in \Hom_{\bbZ}(\mathrm{Pic}([V/G]),\bbZ)$ if $\beta$
is the composite map 
\[
\mathrm{Pic}([V/G])\xrightarrow{u^*}\mathrm{Pic}(C)\xrightarrow{\deg}\bbZ.
\]

\begin{definition} Let $X$ be as before, $n,g$ be positive integers and $\beta\in \Hom_{\bbZ}(\mathrm{Pic}([V/G]),\bbZ)$. An $n$-pointed, genus $g$ quasimap of class $\beta$ to $X$
consists of the data
\[(C,p_1,\ldots, p_n,P,u),\]
where
\begin{enumerate}
\item $(C, p_1,\ldots, p_n)$ is a connected, at most nodal, $n$-pointed projective curve of genus $g$,
\item $P$ is a principal $G$-bundle on $C$,
\item $u$ is a section of the induced vector bundle $P\times_G V$ on $C$, such that $(P, u)$ is of class $\beta$,
\end{enumerate}
satisfying the following generic nondegeneracy condition:
\begin{itemize}
\item there is a finite (possibly empty) set $B\subset C$ such that
for every $p \in C \backslash B$ we have $\widetilde{u}(p)\in [V^s/G]\subset [V/G]$, where $\widetilde{u}:C\to [V/G]$ is the map induced by $u$.  
\end{itemize}

The quasimap $(C, p_1,\ldots, p_n, P, u)$ is called prestable
if the set of base points $B$ is disjoint from the nodes and markings on C.
\end{definition}

In \cite[Def. 7.1.1]{C-FKM}, the authors define the length of the prestable quasimap $\q$ at the point $x\in C$, as follows:

 \begin{definition}

The length of a prestable quasimap $(C,p_1,\ldots,p_n,P,u)$ to $V\GIT_\theta G$ at $x\in C$  is
    \begin{equation}\label{eq: CFKM_definition_length}
        \ell(x)\coloneqq \min \left\{\frac{\ord_{x}(u^*s)}{m} \colon s\in H^0(V,L_{m\theta})^G, u^*s\neq 0, m>0\right\}.
    \end{equation}
   \end{definition}
 
 \begin{definition}\label{estab} Given a positive rational number $\epsilon$, a quasimap $(C, p_1,\ldots, p_k, P, u)$ is called $\epsilon$-stable if it is prestable and 
    \begin{enumerate}
    \item the line bundle $\omega_C(\sum_{i=1}^n p_i)\otimes \ccL^{\epsilon}$, where
\[
\ccL:=u^*(P\times_GL_{\theta})
\]
is ample 
\item $\epsilon\ell(x)\leq 1$, for any $x\in C$.
\end{enumerate}
\end{definition}
\begin{theorem}\cite[Theorem 5.2.1, Theorem 7.1.6]{C-FKM} The moduli space of $\epsilon$-stable quasimaps $\Qx$ is a proper DM stack with a perfect obstruction theory.
\end{theorem}

Let $\Bun$ denote the moduli stack of $G$-bundles over prestable curves, let $\pi:\ffC\to \Bun$ be the universal curve and $\ffP$ the universal principal bundle on $\ffC$. Let $\ffL:=\ffP\times_GL_{\theta}$ and let $\Bun^{\rm st,\epsilon}$ be open locus in $\Bun$ where
\[
    \omega_{\pi}(\sum p_i)\otimes\ffL^{\epsilon}
\]
is ample. 
By \cite[Theorem 3.2.5, and Section 4.2]{C-FKM}, we have an open substack $\Bun^{\rm st,\epsilon}_{\beta}\subset \Bun$, such that the forgetful morphism $\Qx\to \Bun$ factors through $\Bun^{\rm st,\epsilon}_{\beta}$ and such that $\Bun^{\rm st,\epsilon}_{\beta}$ is of finite type.

As in \Cref{dropindex}, we fix $g,n,\beta, \epsilon$ and we omit the index from the notation.

 Let $\ccV_{\ffP}=\ffP\times_G V$ be the associated vector bundle on $\ffC$. Let 
\[
\ffF:=R^1\pi_*(\ccV_{\ffP}^{*}\otimes \omega_{\pi})\] 
on $\Bun$. By definition we have that
\[
\Qx\subset \Spec\Symm \ffF\]
is an open subset obtained by imposing the open conditions in \Cref{estab}.
\begin{remark}
    Particular examples of this construction are:
    \begin{itemize}
    \item toric varieties, with $\Bun\simeq \Pic_{g,n,d_1}\times_{\ffM_{g,n}} \cdots\times_{\ffM_{g,n}}\Pic_{g,n,d_k}$ and $k$ the rank of ${\rm Pic}(X)$;
        \item Grassmannians $\Gr(k,r)$, with $\Bun$ the moduli space of rank $k$-bundles over prestable curves;
        \item complete intersections in projective spaces;
    \end{itemize}
\end{remark}
\subsection{Reduced invariants for quasimaps to GIT quotients}\label{subsec: general}
In the following we define reduced invariants for the targets in the previous section. 

As in \Cref{subsec: def reduced GW}, we define
\begin{equation}
\wBun:=\Bl_{\ffF}\Bun\xrightarrow{p}\Bun.
\end{equation}
By construction $(p^*\ffF)^{\tf}$ is locally free. We define $\tQx$ by the Cartesian diagram
\begin{equation}
\begin{tikzcd}\label{blq}
\tQx \arrow[d, "\widetilde{\mu}"] \arrow[r,"\overline{p}"] \arrow[dr, phantom,"\ulcorner", very near start] 
& \ \arrow[d, "\mu"] \Qx \\
\wBun \arrow[r,"p"]  & \Bun\text{.}
\end{tikzcd}
\end{equation}
One can see that we have an open embedding
\[
\tQx\hookrightarrow  \Spec\Symm _{\wBun}\left( p^*\ffF\right)=p^*\Spec\Symm_{\Bun} (\ffF).
\]

Let $\overline{\pi}\colon \ccC\to \Qx$  be the universal curve over $\Qx$, $\ccP$ be the universal $G$-torsor over $\ccC$ and $\mathcal{V}_{\ccP}=\ccP\times_G V$. Similarly, define ${\ccV}_{\widetilde{\ccP}}$, using $\widetilde{\ccP}$ on $\widehat{\pi}\colon \widetilde{\ccC}\to \tQx$. By \cite{Chang-Li-maps-with-fields} the morphism $\mu$ has a perfect obstruction theory equal to $R^{\bullet}\opi_*\ccV_{\ccP}$. By \Cref{blq} the morphism $\widetilde{\mu}$ has a dual perfect obstruction theory given by 
\[
\phi_{\widetilde{\mu}}: \bbT_{\widetilde{\mu}}\to
R^{\bullet}\widehat{\pi}_*{\ccV}_{\widetilde{\ccP}}.
\]

As before, these perfect obstruction theories induce virtual fundamental classes 
\[[\Qx]^\vir:=\mu^![\Bun]\in A_*(\Qx)
\]
and 

\[[\tQx]^\vir:=\widetilde{\mu}^![\wBun]\in A_*(\tQx).
\]
The following is a generalisation of Assumption \ref{assum:d,big,2g,minus2}.
\begin{assumption}\label{assum:general}
In the following we assume $\beta\in \Hom_{\bbZ}(\mathrm{Pic}([V/G]),\bbZ)$ is such that $$H^1(C,P\times_G V)=0$$ for any $C$ smooth.  
\end{assumption}
From now on we work under the Assumption \ref{assum:general}. Let $\mQx$ be the closure the locus in $\Qx$ which consists of points $(C,p_1,\ldots p_n, P,u)$, with $C$ a smooth curve. We call this the main component of $\Qx$. We define the main component $\tmQ$ of $\tQx$ as the closure of the locus which consists of points $(C,p_1,\ldots p_n, P,u)$, with $C$ a smooth curve. Note that this makes sense even when $\wPic$ does not have a modular interpretation, since $p$ is birational. Note that we have a commutative diagram
\begin{equation}
\begin{tikzcd}
\tmQ \arrow[d, "\widetilde{\mu}^\circ"] \arrow[r,"\overline{p}"]  
& \ \arrow[d, "\mu^\circ"] \mQx \\
\wBun \arrow[r,"p"]  & \Bun\text{.}
\end{tikzcd}
\end{equation}
By \Cref{assum:general}, $\widetilde{\mu}^\circ$ is a smooth morphism, in fact it is just the projection of a geometric vector bundle. This makes $\tmQ$ a smooth stack over $\wBun$, and thus an equidimensional DM stack, so we define

\[
[\tmQ]^\vir:=[\tmQ]\in A_*(\tmQ).
\]

\begin{definition}\label{def:reduced GW general}
    Let $\gamma_i\in A^*(X)$, we call \emph{reduced} quasimap invariants of $X$, the following numbers 
    \[\int_{[\tmQ]^{\vir}}\prod ev^*\gamma_i. \]
\end{definition}

We deduce the following proposition analogous  to \Cref{invariance-red} which states that these invariants do not depend of the choice of the blow-up.

\begin{proposition}\label{prop:invariance-red,complete,intersection} 
Let $p':\wBun'\to\Bun$ and $p'':\wBun''\to\Bun$ be birational projective maps such that $(p'^*\pi_*\ffF)^\tf$ (resp. $(p''^*\pi_*\ffF)^\tf$) and for any $i\in\{1,\ldots,m\}$. Consider $\tmQ'$ and $\tmQ''$ defined analogously to $\tmQ$ above. Then we have

\[
\int_{[\tmQ']^{\vir}}\prod ev^*\gamma_i =\int_{[\tmQ'']^{\vir}}\prod ev^*\gamma_i. \]

\end{proposition}
\begin{remark} Note that as in \Cref{rmk: main component is not pull back} there are different ways of defining the main component of the moduli space of maps (or quasimaps) and we use both definitions.

   \Cref{def:reduced GW general} and \Cref{{def:reduced_GW_invariants}} are related, but possibly different in the following sense:
   \begin{enumerate}
   
       \item Assumption \ref{assum:general} is usually stronger than \ref{assum:d,big,2g,minus2}. For example for a rational curve with normal bundle $\ccO_{\bbP^1}(-d)\oplus \ccO_{\bbP^1}(-d)$ Assumption \ref{assum:general} is not satisfied, while \ref{assum:d,big,2g,minus2} is satisfied for large $d$.
       \item The main components $\tmQ$ in \Cref{def: reduced} and in \Cref{subsec: general} are the same when $X$ is a projective space.
       \item The main components $\tmQ$ in \Cref{def: reduced} and in \Cref{subsec: general} are usually different when $X$ is a complete intersection.
   \end{enumerate}
\end{remark}

We make the following conjecture.
\begin{conjecture}
If $X$ is Fano, then the reduced quasimap and reduced Gromov--Witten invariants agree.
\end{conjecture}
\begin{remark}
To see the intuition behind the conjecture above we first use that for Fano varieties GW and quasimap invariants agree \cite{ciocan2020quasimap}. Then we consider the difference between standard and reduced invariants: this is conjecturally given by maps (or quasi-maps) from curves with components $C_i$ of genus $g_i>0$ and of degree $d_i$, such that $H^1(C_i, d_i)\neq 0$. The moduli space of stable maps have more such components than the space of quasimaps, namely maps from curves with a rational curve with no marked points glued to higher genus curves. All these contributions are expected to be zero.

  Since GW and quasimap invariants are different for non-Fano varieties, we expect the corresponding reduced invariants to be different if $X$ is not Fano.  
\end{remark}

\section{Desingularizations in genus one}\label{sec:comparisons}

In genus one, reduced Gromov--Witten invariants were originally defined using the desingularization constructed in \cite{Vakil-Zinger-desingu-main-compo-}. It consists of a sequence of blow-ups determined by the geometry of the moduli space $\Gwone$. In \cite{Hu-Li-GEnus-one-local-VZ-Math-ann-2010}, local equations for the blowup are determined. We aim to compare this desingularization with the one obtained using the Rossi blow-up $\Bl_\ffF \Pic$, with $\ffF$ as in \Cref{eq:F_sheaf}. In particular, we describe $\Bl_\ffF \Pic$ locally in the spirit of \cite{Hu-Li-GEnus-one-local-VZ-Math-ann-2010}.

\subsubsection{Charts}

In genus one, the original definition of refined Gromov--Witten invariants comes from \cite{Vakil-Zinger-desingu-main-compo-}. The main is idea is to apply a sequence of blow ups to $\Gw$ in order to desingularize the main component. Strictly speaking, the sequence of blow ups takes place in the stack $\cwt$ of genus-1 prestable curves endowed with a weight. Let $\Theta_k$ denote the closure of the loci in $\cwt$ of curves with $k$ trees of rational curves attached to the core. Then one should blow up $\cwt$ along the loci $\Theta_1$, $\Theta_2$, $\Theta_3$ and so on in order to produce a stack $\blcwt$. This process induces a blowup $\tGw$ of $\Gw$ via fibre product. 

Given a stratum $\tGw_\gamma$ corresponding to a weighted graph $\gamma$, local equations of $\tGw$ and the local description of $\Theta_k$ in that stratum are described explicitly in \cite{Hu-Li-GEnus-one-local-VZ-Math-ann-2010}. The purpose of this section is to summarize such local description, give coordinates for the new approach locally and compare both.

It may be helpful to keep in mind the following diagram, described below.
\[
    \begin{tikzcd}[row sep=small]
    \ccE_\ccV\ar[rr,"\widetilde{\phi}"]   & & E_\gamma \ar[d]\\
      (F=0)\ar[r] \ar[u,hook]&   (\Phi_\gamma = 0) \ar[ur,hook] & V_\gamma\\
        \ccU\ar[r]\ar[d]\ar[u] &  \ccV\ar[d]\ar[ru,"\phi"] \arrow[dr, phantom, "\square"]   & \widetilde{\ccV}\ar[l]\ar[d]\\
        U\ar[r] & \ffD_1\ar[d] \arrow[dr, phantom, "\square"]     & \widetilde{\ffD}_1\ar[l]\ar[d]&\\
                & \cwt     & \blcwt \ar[l]&
    \end{tikzcd}
\]
 Fix a weighted graph $\gamma$ with root $o$. Let $\ver$, $\tver$ and $\vers$ denote the vertices, the terminal vertices (or leaves) and the non-rooted vertices of $\gamma$, respectively. We take the natural ordering in $\ver$ making the root $o$ the minimal element. We assume that the weight in $\gamma$ is non-negative on every vertex and that $\gamma$ is terminally weighted, meaning that the vertices with non-zero weights are exactly those in $\tver$. Let $\cwt$ is the stack of genus-1 prestable curves endowed with a weight. Remember that every element $C$ parametrized by $\cwt$ has a dual (weighted) graph $\gamma$, which can be made terminally weighted and rooted by first declaring the root to be the (contraction of) the core of $C$ and then pruning along all non-terminal positively weighted nodes. We will denote by $o$ the root of any terminally weighted rooted tree, and by $a,b,\ldots$ the remaining vertices.

In the diagram, $\blcwt$ denotes the blow up of $\cwt$ described above and $\ffD_1$ is the stack of stable pairs $(C,D)$ with $D$ an effective Cartier divisors supported in the smooth locus of $C$. Fix a point $(C,D)$ in $\ffD_1$ and a map in $\Gw$ with underlying curve $C$. Then $U$ is a small open around the fixed map in $\Gw$, $\ccV$ is a smooth chart around the point $(C,D)$ in $\ffD_1$ containing the image of $U$ and $\ccE_\ccV$ is the total space of the sheaf $\rho_\ast \ccL(\ccA)^{\oplus n}$ on $\ccV$. 

Let $V_\gamma = \prod_{v\in \vers} \bbA^1$ be an affine space that serves as model for local equations. We denote by $z_a, z_b, \ldots$ the natural coordinates in $V_\gamma$. Similarly, $E_\gamma = V_\gamma \times \prod_{v\in \tver} (\bbA^1)^r$ and the coordinates on the affine space $(\bbA^1)^r$ corresponding to $a\in\tver$ will be denoted by $w_{a,1},\ldots, w_{a,r}$. The ideal $\Phi_\gamma = (\Phi_{\gamma,1},\ldots, \Phi_{\gamma,r})$ will be described explicitly in \Cref{eq:local_equations_maps}. The smooth morphism $\phi$ comes from the natural coordinates on $\ccV$, associated to the smoothing of each of the disconnecting nodes in $C$ (which are in natural bijection with $\vers$). The map $\ccU\to (F=0)$ is an open embedding. Finally, $\widetilde{\phi}$ is induced by $\phi$ and $F = \widetilde{\phi}^* \Phi_\gamma$. It is in this sense that we can think of $\Phi_\gamma$ as the equations of $\Gw_\gamma$.

Following \cite{Hu-Li-GEnus-one-local-VZ-Math-ann-2010}, given a terminally weighted rooted graph $\gamma$, the ideal $\Phi_\gamma = (\Phi_{\gamma,1},\ldots, \Phi_{\gamma,r})$ inside $V_\gamma$ can be described as 
\begin{equation}\label{eq:local_equations_maps}
    \Phi_{\gamma,i} = \sum_{v\in \tver} z_{\left[v,o\right]} w_{v,i}\quad 1\leq i\leq r,
\end{equation}
where 
\[
    z_{\left[v,o\right]} = \prod_{o\prec a \preceq v} z_a.
\]

Note that, for fixed $i$, the variables $w_{a,i}$ only appear in the $i$-th equation $\Phi_{\gamma,i}$. Due to the symmetry of the equations and the fact that all blow ups take place in $V_\gamma$, which has coordinates $\{z_a\}_{a\in\vers}$ (but not the $w_{a,i}$), in the examples below we will not write down the index $i$ in the equations $\Phi_{\gamma,i}$ nor in the variables $w_{a,i}$. For example, in the study of the equations $\Phi_{\gamma,i}$ after blowing up, it will be clear that the index $i$ is irrelevant, in the sense that the way that $\Phi_{\gamma,i}$ changes is independent of $i$.

\subsection{Local equations of desingularizations}

The local equations of the loci that must be blown up are described, following \cite{Hu-Li-GEnus-one-local-VZ-Math-ann-2010}.

Firstly, we describe how to assign an ideal $I_\gamma$ to any semistable terminally weighted rooted tree $\gamma$. Here, \textit{semistability} of $\gamma$ means that every non-root vertex with weight zero has at least two edges. 

The \textit{trunk} of $\gamma$ is the maximal chain $o=v_0\prec v_1\prec\ldots\prec v_r$ of vertices in $\gamma$ such that each vertex $v_i$ with $1\leq i < r$ has exactly one immediate descendant 
and $v_r$ is called a \textit{branch vertex} if is it not terminal. Note that $\gamma$ is a path tree if and only if it has no branch vertex.

\begin{definition}
    Let $\gamma$ be a semistable terminally weighted rooted tree with branch vertex $v$ and let $a_1,\ldots, a_k$ be the immediate descendants of $v$. To $\gamma$ we associate the ideal
    \[
        I_\gamma = (z_{a_1}, \ldots, z_{a_k})
    \]
    in $V_\gamma$.
\end{definition}
First, we must blow up $V_\gamma$ along the ideal $I_\gamma$. To describe the remaining steps we need to introduce the following operations.

\begin{definition}
    Let $\gamma$ be a terminally weighted semistable rooted tree.
    \begin{itemize}
    \item  The \textit{pruning} of $\gamma$ along a vertex $v$ is the new tree obtained by removing all the descendants of $v$ (and the corresponding edges) and declaring the weight of $v$ to be the sum of the original weight of $v$ plus the weights of all removed vertices. 
    \item The \textit{advancing} of a vertex $v$ with immediate ascendant $\overline{v}$ in $\gamma$ is a new tree obtained by replacing every edge $(\overline{v},v')$ with $v'\neq v$ by an edge $(\overline{v},v)$ and pruning along all positively weighted non-terminal vertices as long as possible. In \Cref{subsec:examples_g1} we will denote by $\gamma_v$ the advancing of $v$ in $\gamma$ and by $\gamma_v'$ the same tree before pruning.
    \item Suppose $\gamma$ has a branch vertex $v$. A \textit{monoidal transform} of $\gamma$ is a tree obtained by advancing one of the immediate descendants of $v$. The set of monoidal transforms of $\gamma$ is $\Mon(\gamma)$.
\end{itemize}
\end{definition}

It turns out that the ideal $\Phi_\gamma$ behaves nicely under monoidal transforms. Indeed, let $\gamma$ be a semistable terminally weighted rooted tree with branch vertex $v$ and let $a=a_1, a_2,\ldots, a_k$ be the immediate descendants of $v$. Let $\gamma_a$ be the tree advancing of $a$ in $\gamma$. Let $\pi\colon \widetilde{V_\gamma}\to V_\gamma$ be the blow-up of $V_\gamma$ along the ideal $I_\gamma = (z_{a_1},\ldots, z_{a_k})$. We view $\widetilde{V_\gamma}$ embedded inside $V_\gamma \times \bbP^{k-1}$. There is a natural way to associate to each generator $z_{a_i}$ of $I_\gamma$ one chart of $\bbP^{k-1}$, and thus also of $\widetilde{V_\gamma}$. We denote such chart by $\widetilde{V}_{\gamma,a_i}$. Let $\pi_a\colon \widetilde{V}_{\gamma,a}\to V_\gamma$ be the restriction of the natural projection, where $a=a_1$. Then, by the proof of \cite[Lemma 5.14]{Hu-Li-GEnus-one-local-VZ-Math-ann-2010}, one of the following must hold
\begin{itemize}
    \item either $\gamma_a$ is a path tree, and then the zero locus of $\pi_a^*(\Phi_\gamma)$ has smooth components;
    \item or $\gamma_a$ is not a path tree and then
    \begin{equation}\label{eq: Phi_gamma_pulls_back}
        \pi_a^*(\Phi_\gamma) = \Phi_{\gamma_a}.
    \end{equation}
\end{itemize}

The whole blow-up process is summarized as follows.
Fix $\gamma$. First blow up $V_\gamma$ along $I_\gamma$. The pullback of $\Phi_\gamma$ is controlled by $\Mon(\gamma)$. If $\Mon(\gamma)$ consists only of path trees, we are done. Otherwise, for every element $\gamma'$ in $\Mon(\gamma)$ which is not a path tree, blow up the chart of $\widetilde{V_{\gamma}}$ corresponding to $\gamma'$ along $I_{\gamma'}$. Continue recursively. The process concludes by \cite[Lemma 3.12]{Hu-Li-GEnus-one-local-VZ-Math-ann-2010}.


Now we want to describe $\Bl_\ffF \Pic$ locally. Namely, we want to describe which loci inside $\Pic$ we are blowing up locally. We have an exact sequence
\[
    0 \to \rho_*\ccL \to \rho_*\ccL(\ccA) \to \rho_*\ccL(\ccA)\mid_\ccA
\]
by \Cref{resL}. The change of notation is due to the fact that \Cref{resL} was global in $\Pic$, but we now work locally. 

After a careful study of the second morphism, \cite[Theorem 4.16]{Hu-Li-GEnus-one-local-VZ-Math-ann-2010} concludes that $\rho_\ast \ccL$ is the direct sum of a trivial bundle with the kernel of the morphism 
\begin{equation}\label{eq:morphism_direct_sum_phis}
    \bigoplus_{v\in\tver} \varphi_v \colon \ccO_\ccV^{\oplus \ell} \to \ccO_\ccV, 
\end{equation}
where $\ell$ is the cardinality of $\tver$ and $\varphi_v = \prod_{o\prec v' \preceq v} \zeta_v$, with $\zeta_v$ the smoothing parameter of the disconnecting node corresponding to the vertex $v$.

By \Cref{dual-cone}, the sheaf $\ffF$ can be described locally as the dual of \Cref{eq:morphism_direct_sum_phis}. In particular, by \Cref{rem: dual} we have that $\Bl_\ffF \Pic$ agrees, locally, with the blowup along the ideal generated by the entries $(\varphi_v)_{v\in \ver^t}$. In local coordinates, this ideal can be described as follows.

\begin{definition}
    Let $\gamma$ be a semistable terminally weighted rooted tree with $\tver = \{v_1,\ldots, v_t\}$. To $\gamma$ we associate the ideal
    \[
        J_\gamma = (z_{[v_1,o]},\ldots,  z_{[v_t,o]}).
    \]
\end{definition}

Similarly to \Cref{eq: Phi_gamma_pulls_back}, in the same setup we have that 
\[
    \pi_a^*(J_\gamma) = J_{\gamma_a},
\]
independently of whether $\gamma_a$ is has a branch vertex. This follows again from the proof of \cite[Lemma 5.14]{Hu-Li-GEnus-one-local-VZ-Math-ann-2010}.

\subsubsection{Examples} \label{subsec:examples_g1}

For two concrete trees $\gamma$, we compute the equations $\Phi_\gamma$ as well as the ideals $I_\gamma$ and $J_\gamma$. We describe the blow up process of Hu and Li and show that the result indeed desingularizes the main component of $\Gw$ locally. Furthermore, we check that the ideal $J_\gamma$ becomes locally principal in Hu--Li's blow-up. 

\begin{example}\label{example: charts [ab[cd]]}
Consider the following labelled graph:

\begin{minipage}{.4\textwidth}
\begin{center}
			$\gamma$ = \scalebox{1}{\Tree [.$o$ $a$ [.$b$ $c$ $d$ ]]}
		\end{center}
\end{minipage}%
\begin{minipage}{.6\textwidth}
\begin{align*}
    \Phi_\gamma &= z_a w_a + z_b(z_c w_c + z_d w_d),\\
    I_\gamma &= (z_a,z_b),\\
    J_\gamma &= (z_a,z_bz_c, z_bz_d). 
\end{align*}
\end{minipage}\vspace{1em}

Let $\widetilde{V_\gamma}$ be the blow up along $I_\gamma$, that is the zero locus of $z_a z_b'-z_b z_a'$ inside $\bbA^4_{z_a,z_b,z_c,z_d} \times \bbP^1_{z_a',z_b'}$. The chart associated to $a$ is that where $z_a'\neq 0$. Dehomogenizing amounts to the change of variables $z_b = z_b'z_a$. By doing so, we get that 
\[
    \pi_a^*(\Phi_\gamma) = z_a(w_a + z_b'(z_cw_c + z_dw_d))
\]
and that
\[
    \pi_a^*(J_\gamma) = (z_a, z_a z_b' z_c, z_a z_b' z_d) = (z_a).
\]
This means that the zero locus of $\pi_a^*(\Phi_\gamma)$ already has smooth components, so no further blowups are needed on this chart, and that $\pi_a^*(J_\gamma)$ is principal on this chart too. 

Below are the trees $\gamma_a'$ obtained by advancing $a$ without pruning, and $\gamma_a$ obtained by advancing $a$. We know that $\pi_a^* J_\gamma = J_{\gamma_a}$, but we check it in this example.  

\begin{minipage}{.2\textwidth}
\begin{center}
			$\gamma_a'$ = \scalebox{1}{\Tree [.$o$ [.$a$  [.$b$ $c$ $d$ ]]]}
\end{center}
\end{minipage}%
\begin{minipage}{.2\textwidth}
\begin{center}
			$\gamma_a$ = \scalebox{1}{\Tree [.$o$ [.$a$  ]]}
\end{center}
\end{minipage}%
\begin{minipage}{.6\textwidth}
\begin{align*}
    J_{\gamma_a} &= (z_a). 
\end{align*}
\end{minipage}\vspace{1em}

Similarly, we now look at the chart associated to $b$, where $z_b'\neq 0$. The change of variables is now $z_a = z_bz_a'$. It follows that
\[
    \pi_b^*(\Phi_\gamma) = z_b(z_a'w_a + z_cw_c + z_dw_d)
\]
and that
\[
    \pi_b^*(J_\gamma) = (z_b z_a', z_b z_c, z_b z_d) = z_b (z'_a, z_c, z_d).
\]
This means that we still need to blow up. This time the tree $\gamma_b$ obtained by advancing $b$ in $\gamma$ (no pruning is needed) is not a path tree. We also check the identities $J_{\gamma_a} = \pi_a^* J_\gamma$ and $\pi_b^*(\Phi_\gamma) = \Phi_{\gamma_b}$.

\begin{minipage}{.4\textwidth}
\begin{center}
			\scalebox{1}{\Tree [.$o$ [.$b$  [.$a$ ] [.$c$ ] [.$d$ ]]]}
		\end{center}
\end{minipage}%
\begin{minipage}{.6\textwidth}
\begin{align*}
    \Phi_{\gamma_b} &= z_b(z_aw_a + z_c w_c + z_d w_d),\\
    I_{\gamma_b} &= (z_a, z_c, z_d),\\
    J_{\gamma_b} &= z_b(z_a, z_c, z_d). 
\end{align*}
\end{minipage}\vspace{1em}

To conclude the example, we need to blow up along the ideal $(z_a,z_c,z_d)$. We collect the result below.

Advancing $a$, or equivalently looking at the chart $z_a'\neq 0$, we have

\begin{minipage}{.2\textwidth}
\begin{center}
			$\gamma_{b,a}'$ = \scalebox{1}{\Tree [.$o$ [.$b$  [.$a$  [.$c$ ] [.$d$ ]]]]}
\end{center}
\end{minipage}%
\begin{minipage}{.2\textwidth}
\begin{center}
			$\gamma_{b,a}'$ = \scalebox{1}{\Tree [.$o$ [.$b$  [.$a$  ]]]}
\end{center}
\end{minipage}%
\begin{minipage}{.6\textwidth}
\begin{align*}
    \pi_a^*\pi_b^*(\Phi_{\gamma}) = \pi_a^*(\Phi_{\gamma_b}) &= z_az_b(w_a + z_cw_c + z_d w_d),\\
    J_{\gamma_{b,a}} &= (z_az_b).
\end{align*}
\end{minipage}\vspace{1em}

Advancing $c$, or equivalently looking at the chart $z_c'\neq 0$, we have

\begin{minipage}{.2\textwidth}
\begin{center}
			$\gamma_{b,c}'$ = \scalebox{1}{\Tree [.$o$ [.$b$  [.$c$  [.$a$ ] [.$d$ ]]]]}
\end{center}
\end{minipage}%
\begin{minipage}{.2\textwidth}
\begin{center}
			$\gamma_{b,c}'$ = \scalebox{1}{\Tree [.$o$ [.$b$  [.$c$  ]]]}
\end{center}
\end{minipage}%
\begin{minipage}{.6\textwidth}
\begin{align*}
    \pi_c^*\pi_b^*(\Phi_{\gamma}) = \pi_c^*(\Phi_{\gamma_b}) &= z_bz_c(z_aw_a + w_c + z_d w_d),\\
    J_{\gamma_{b,c}} &= (z_bz_c).
\end{align*}
\end{minipage}\vspace{1em}

And finally, advancing $d$, or equivalently looking at the chart $z_d'\neq 0$, we have

\begin{minipage}{.2\textwidth}
\begin{center}
			$\gamma_{b,d}'$ = \scalebox{1}{\Tree [.$o$ [.$b$  [.$d$  [.$a$ ] [.$c$ ]]]]}
\end{center}
\end{minipage}%
\begin{minipage}{.2\textwidth}
\begin{center}
			$\gamma_{b,d}'$ = \scalebox{1}{\Tree [.$o$ [.$b$  [.$d$  ]]]}
\end{center}
\end{minipage}%
\begin{minipage}{.6\textwidth}
\begin{align*}
    \pi_d^*\pi_b^*(\Phi_{\gamma}) = \pi_d^*(\Phi_{\gamma_b}) &= z_bz_d(z_aw_a + z_cw_c + w_d),\\
    J_{\gamma_{b,d}} &= (z_bz_d).
\end{align*}
\end{minipage}
\end{example}

\begin{remark}\label{rem: dual}
    \Cref{example: charts [ab[cd]]} shows that the Rossi blow-up process of $\Gw$ is not equal to the Vakil--Zinger blow-up. This is compatible with Remark 4.4. in \cite{niu1}. Indeed, the Rossi blowup around $\gamma$ is given by $\Bl_{J_\gamma} V_\gamma$ and the Vakil--Zinger one is the iterated blow-up $\Bl_{I_{\gamma_b}} \Bl_{I_\gamma} V_\gamma$. We know there is a natural morphism
    \[
        \Bl_{I_{\gamma_b}} \Bl_{I_\gamma} V_\gamma \to \Bl_{J_\gamma} V_\gamma 
    \]    
    over $V_\gamma$, either by \Cref{prop:morphism-HL-Rossi} or because we checked that $J_\gamma$ pulls back to a principal ideal in $\Bl_{I_{\gamma_b}} \Bl_{I_\gamma} V_\gamma$. By contradiction, if there is a morphism
    \[
        f\colon \Bl_{J_\gamma} V_\gamma  \to \Bl_{I_{\gamma_b}} \Bl_{I_\gamma} V_\gamma
    \]
    over $V_\gamma$, then we get a morphism
    \[
        \widetilde{f} \colon \Bl_{J_\gamma} V_\gamma  \to \Bl_{I_\gamma} V_\gamma
    \]
    over $V_\gamma$. By \cite{Moody}, there is a fractional ideal $K$ in $V_\gamma$ and a positive integer $\alpha$ such that 
    \[
        I_\gamma \cdot K = J_\gamma^\alpha.
    \]
    This is not true for $I_\gamma = (z_a,z_b)$ and $J_\gamma = (z_a, z_b z_c, z_b z_d)$ in $V_\gamma = \bbA^4_{z_a,z_b,z_c,z_d}$.
\end{remark}

\begin{example}\label{example: charts [a[cd]b[ef]]}
We do a similar study for the following labelled graph $\gamma$:

\begin{minipage}{.3\textwidth}
\begin{center}
			\scalebox{1}{\Tree [.$o$ [.$a$ $c$ $d$ ] [.$b$ $e$ $f$ ]]}
		\end{center}
\end{minipage}%
\begin{minipage}{.7\textwidth}
\begin{align*}
    \Phi_\gamma &= z_a (z_c w_c + z_d w_d) + z_b(z_e w_e + z_f w_f),\\
    I_\gamma &= (z_a,z_b),\\
    J_\gamma &= (z_a z_c, z_a z_d,z_b z_e, z_b z_f). 
\end{align*}
\end{minipage}\vspace{1em}

After blowing up, there are two charts, corresponding to the advacings of $a$ and $b$ respectively.

\begin{minipage}{.3\textwidth}
\begin{center}
			$\gamma_a$ = \scalebox{1}{\Tree [.$o$ [.$a$ $c$ $d$  [.$b$ $e$ $f$ ]]]}
		\end{center}
\end{minipage}%
\begin{minipage}{.7\textwidth}
\begin{align*}
    \Phi_{\gamma_a} &= z_a (z_c w_c + z_d w_d + z_b(z_e w_e + z_f w_f)),\\
    I_{\gamma_a} &= (z_b,z_c,z_d),\\
    J_{\gamma_a} &= z_a (z_c, z_d, z_b z_e, z_b z_f). 
\end{align*}
\end{minipage}\vspace{1em}

\begin{minipage}{.3\textwidth}
\begin{center}
			$\gamma_b$ = \scalebox{1}{\Tree [.$o$ [.$b$ [.$a$ $c$ $d$ ] $e$ $f$  ]]}
		\end{center}
\end{minipage}%
\begin{minipage}{.7\textwidth}
\begin{align*}
    \Phi_{\gamma_b} &= z_b (z_a(z_c w_c + z_d w_d) + z_e w_e + z_f w_f),\\
    I_{\gamma_b} &= (z_a,z_e,z_f),\\
    J_{\gamma_b} &= z_b (z_a z_c, z_a z_d, z_e, z_f). 
\end{align*}
\end{minipage}\vspace{1em}

By symmetry, it is enough to understand how to proceed in one of the charts. We choose the one corresponding to $a$. We get three new charts corresponding to the vertices $c$, $d$ and $b$.

\begin{minipage}{.2\textwidth}
\begin{center}
			$\gamma_{a,c}'$ = \scalebox{1}{\Tree [.$o$ [.$a$ [.$c$ $d$  [.$b$ $e$ $f$ ]]]]}
		\end{center}
\end{minipage}%
\begin{minipage}{.2\textwidth}
\begin{center}
			$\gamma_{a,c}$ = \scalebox{1}{\Tree [.$o$ [.$a$ [.$c$ ]]]}
		\end{center}
\end{minipage}%
\begin{minipage}{.7\textwidth}
\begin{align*}
    \pi_c^*\Phi_{\gamma_a} &= z_a z_c (w_c + z_d w_d + z_b(z_e w_e + z_f w_f)),\\
    J_{\gamma_{a,c}} &= (z_a z_c). 
\end{align*}
\end{minipage}\vspace{1em}

\begin{minipage}{.2\textwidth}
\begin{center}
			$\gamma_{a,d}'$ = \scalebox{1}{\Tree [.$o$ [.$a$ [.$d$ $c$  [.$b$ $e$ $f$ ]]]]}
		\end{center}
\end{minipage}%
\begin{minipage}{.2\textwidth}
\begin{center}
			$\gamma_{a,d}$ = \scalebox{1}{\Tree [.$o$ [.$a$ [.$d$ ]]]}
		\end{center}
\end{minipage}%
\begin{minipage}{.7\textwidth}
\begin{align*}
    \pi_d^*\Phi_{\gamma_a} &= z_a z_d (z_c w_c + w_d + z_b(z_e w_e + z_f w_f)),\\
    J_{\gamma_{a,d}} &= (z_a z_d). 
\end{align*}
\end{minipage}\vspace{1em}

\begin{minipage}{.3\textwidth}
\begin{center}
			$\gamma_{a,b}$ = \scalebox{1}{\Tree [.$o$ [.$a$ [.$b$ $c$ $d$ $e$ $f$ ]]]}
		\end{center}
\end{minipage}%
\begin{minipage}{.7\textwidth}
\begin{align*}
    \Phi_{\gamma_{a,b}} &= z_a z_b (z_c w_c + z_d w_d + z_e w_e + z_f w_f),\\
    I_{\gamma_{a,b}} &= (z_c, z_d, z_e, z_f),\\
    J_{\gamma_{a,b}} &= z_a z_b (z_c, z_d, z_e, z_f). 
\end{align*}
\end{minipage}\vspace{1em}

To conclude, we need to blow up the last chart along $I_{\gamma_{a,b}}$. This produces four new charts corresponding to $c, d, e$ and $f$. We will only write down one of them since the rest are very similar.

\begin{minipage}{.2\textwidth}
\begin{center}
			$\gamma_{a,b,c}'$ = \scalebox{1}{\Tree [.$o$ [.$a$ [.$b$ [.$c$ $d$ $e$ $f$ ]]]]}
\end{center}
\end{minipage}%
\begin{minipage}{.2\textwidth}
\begin{center}
			$\gamma_{a,b,c}$ = \scalebox{1}{\Tree [.$o$ [.$a$ [.$b$ [.$c$ ]]]]}
\end{center}
\end{minipage}%
\begin{minipage}{.7\textwidth}
\begin{align*}
    \pi_c^*\Phi_{\gamma_{a,b}} &= z_a z_b z_c (w_c + z_d w_d + z_e w_e + z_f w_f),\\
    J_{\gamma_{a,b,c}} &= (z_a z_b z_c). 
\end{align*}
\end{minipage}\vspace{1em}
\end{example}

\subsubsection{Smoothness}

In genus one, $\Gwone\times_{\Pic} \Bl_{\ffF}\Pic$ has simple normal crossings following the same argument as in \cite[Theorem 5.24]{Hu-Li-GEnus-one-local-VZ-Math-ann-2010}. It is enough to show that the zero locus of the ideal $\Phi_\gamma$ becomes a simple normal crossing in the blow-up $\widetilde{V_\gamma}$ of $V_\gamma$ along the ideal $J_\gamma$. 

Remember that $\Phi_\gamma = (\Phi_{\gamma,1},\ldots, \Phi_{\gamma,r})$ with 
\[
    \Phi_{\gamma,i} = \sum_{v\in \tver} z_{\left[v,o\right]} w_{b,i}\quad 1\leq i\leq r,
\]
where 
\[
    z_{\left[v,o\right]} = \prod_{o\prec a \preceq v} z_a,
\]
and that 
\[
    J_\gamma = (z_{[v,o]})_{v\in\tver}.
\]
For given $v'\in \tver$, the pullback of the equation $\Phi_{\gamma,i}$ on the chart corresponding to $v'$ is equal to
\[
    z_{[v',o]} \left(w_{v',i} + \sum_{v'\neq v \in \tver} z_{\left[v,o\right]} w_{b,i} \right)
\]
by \cite[Lemma 5.14]{Hu-Li-GEnus-one-local-VZ-Math-ann-2010}. This proves the claim.

\subsubsection{Maps between blowups}

By \Cref{prop:morphism-HL-Rossi} there is a morphism from Vakil--Zinger's blowup to Rossi's blowup. In genus one, we can check it locally: it is equivalent to the fact that the pullback of the ideal $J_\gamma$ to each chart $\widetilde{V_\gamma}$ of the Hu--Li blowup of $V_\gamma$ is principal. We have checked this in \Cref{example: charts [ab[cd]]} and \Cref{example: charts [a[cd]b[ef]]}, More generally, we can give a proof for every $\gamma$ as follows.

By \Cref{eq: Phi_gamma_pulls_back} if $z_a$ is any of the generators of $I_\gamma$, then $\pi_a^*(J_\gamma) = J_{\gamma_a}$ where $\gamma_a$ is the advancing of $a$ in $\gamma$. In particular, it is enough to show that all the (natural) charts of $\widetilde{V_\gamma}$ correspond to path trees, which is proven in \cite[Lemma 3.14]{Hu-Li-GEnus-one-local-VZ-Math-ann-2010}.

\bibliographystyle{alpha}
\bibliography{references}

\end{document}